\documentclass[oneside,reqno]{amsart}

\usepackage{arxiv_preamble}

\title[Enumeration of Triangle-Maximal Pseudoline Arrangements]{Enumeration and Classification\\ of Triangle-Maximal Pseudoline Arrangements}
\author{Roman Parpalak}
\author{Denis Utkin}
\thanks{Email: \texttt{roman@parpalak.com} (R.\,Parpalak),
  \texttt{ud1@mail.ru} (D.\,Utkin).}
\date{July 31, 2026}

\begin{document}

\begin{abstract}
We describe algorithms for the exhaustive enumeration and classification of
simple arrangements of $n$ pseudolines ($n$ odd) maximizing the number of triangular faces.
The depth-first search enumerates reduced words for the longest permutation
$w_0$ by branching only on the even-indexed generators, using pruning constraints
imposed by the geometry of optimal arrangements.  The approach handles both perfect
arrangements with a regular triangular pattern and unavoidable deviations from it
for $n\equiv 1\pmod 6$.

The output is classified into a hierarchy of equivalence classes: by
commutation, by Euclidean transformations, and by projective transformations.  For each
projective class we recover its full symmetry group $G\subseteq S_{n+1}$
together with the orbit--stabilizer profile of its Euclidean subclasses.
Completeness of the search and classification is proved: every wiring diagram is reached.

We report full enumerations; e.g.\ for $n=27$, $85\,562\,064$ wiring diagrams
partitioned into $56\,646$ projective classes.  For larger $n$ (up to $n=93$),
where exhaustive enumeration is out of reach, we report partial (first-hit) results.
\end{abstract}

\maketitle

\section{Introduction}

The study of straight-line and pseudoline arrangements with many triangular
faces goes back to Grünbaum~\cite{grunbaum-1972}, who tabulated the known
extremal values and stated a series of conjectures.  Harborth~\cite{harborth-1985}
and Roudneff~\cite{roudneff-1986} gave constructions of pseudoline
arrangements with the maximum number of triangles.  Constructions for
straight lines were given by
Füredi and Palásti~\cite{furedi-palasti-1984},
Forge and Ramírez Alfonsín~\cite{forge-ramirez-1998}, and Bartholdi, Blanc,
and Loisel~\cite{bartholdi-blanc-loisel-2007}.  Upper bounds were tightened
by Blanc~\cite{blanc-2011}.

The classical extremal formulation in this setting is the Kobon triangle
problem~\cite{fujimura-1978,oeis-a006066}, which asks for the maximum number of triangular faces
determined by \(n\) lines in the
plane.  The Arnold problem is a related question concerning
the checkerboard coloring of the faces. In the English edition of Arnold's
Problems, it is formulated as follows~\cite[Problem 1983--4]{arnold-problems-2004}:
\begin{quote}
Let \(N\) lines be given in the real plane, and their complement be
chess-like painted black and white. What is the greatest difference between
the number of black and white regions?
\end{quote}

L. Fejes Tóth had earlier asked for the maximum possible ratio \(b/w\)
between the numbers of black and white faces in such a
coloring~\cite{fejes-toth-1975,furedi-palasti-1984}; Arnold asked
instead for the maximum difference \(b-w\).

The Kobon and Arnold problems have different formulations but their
extremal arrangements partially coincide.  To make \(b-w\) large, one must make \(b\)
as large as possible. This forces the average number of sides per black
face to be small, so in the extremal case most black faces are triangles.

In a simple arrangement of \(n\) pseudolines, the \(n(n-2)\) bounded
segments give the standard upper bound \(\lfloor n(n-2)/3\rfloor\) on the
number of triangular faces.  For \(n\equiv 3,5\pmod 6\),
\(n(n-2)\) is divisible by \(3\), so an arrangement attaining the bound
has every bounded segment lying on a triangle. Such arrangements are
called \emph{perfect}.  For \(n\equiv 1\pmod 6\), the remainder is
\(2\), so any arrangement attaining the bound has
two bounded segments not on any bounded triangle (\emph{unused} segments).

The same segment count restricted to one checkerboard color gives
an analogous upper bound \(\lfloor n(n-2)/3\rfloor\) on the number of bounded
black faces.  Perfect arrangements attain this bound as well: every
bounded black face is a triangle.  For \(n\equiv 1\pmod 6\), we introduce the colored
analog of unused segments, \emph{defects}.  A defect is a non-triangular black
face.  Arrangements attaining the maximum contain either two
quadrilateral defects or one pentagonal defect, and we call them
\emph{2-defective}.  In the odd case, one can easily show that the
Arnold optimum among pseudoline arrangements is attained on a simple one
(resolving a multiple crossing raises the black excess).  Overall, perfect
arrangements solve both the Kobon problem (including the general non-simple case,
for which the same upper bound holds~\cite{felsner-kriegel-1999}) and the
Arnold problem, and there exist 2-defective arrangements solving both as well.

We borrow the checkerboard coloring from the Arnold
problem as a constraint for enumerating simple pseudoline arrangements,
encoded as reduced words for the longest permutation.  Stretchability is
not considered: the results are combinatorial, leaving the straight-line
realization as a separate question.

We restrict to odd \(n\); the even case (pairs of parallel lines in the Arnold
problem, non-simple arrangements in the Kobon problem) is beyond the
scope of this paper.

\emph{Main results.}
We give algorithms for enumerating and classifying extremal
pseudoline arrangements, together with proofs of their completeness.
This yields the complete classification of perfect arrangements for
\(n\le 27\) and of 2-defective arrangements for \(n\le 19\) up to projective
equivalence, with the symmetry group of each class.  For example, at \(n=27\)
there are \(85\,562\,064\) wiring diagrams partitioned into \(56\,646\)
projective classes.
The same algorithms were used to find the base configuration for the
infinite family of triangle-maximal straight-line arrangements with
\(18\cdot 2^t+1\) lines~\cite{parpalak-utkin-2026}, which gives
solutions to the Kobon and Arnold problems.

\emph{Paper organization.}
\Cref{sec:background} introduces the basic notions: wiring diagrams,
generators of the symmetric group, reduced words for the longest
permutation, and the checkerboard coloring.  It establishes the
properties of defects and optimal arrangements, translating the extremal problem into one over
reduced words.

\Cref{sec:problem} states the enumeration problem precisely, surveys
prior algorithmic approaches, and outlines our approach.

\Cref{sec:perfect-search} develops an exhaustive enumeration
for perfect arrangements, exploiting their local structure; we call it the
\emph{perfect search}. Every perfect arrangement is reached
(\cref{thm:perfect-completeness}), exactly once per commutation class
(\cref{cor:perfect-class-output}).

\Cref{sec:classification} groups the search output into Euclidean and
projective equivalence classes, choosing a canonical representative for
each and recovering the symmetry group of every projective class; the
underlying operations on order matrices and the canonicalization
algorithm are developed in \cref{app:symmetries,app:canonicalization}.

\Cref{sec:2-defective-search} develops a search for 2-defective
arrangements.  The search anchors
one defect at a fixed position in the wiring diagram and enumerates the
placements of the second defect.  Completeness is proved on projective
classes (\cref{thm:2-defective-completeness}).

\Cref{sec:computational-results} reports the enumeration counts
(compared against prior results) and first-hit examples,
and explains how to extract the triangle-maximal arrangements
from the search output. \Cref{app:symmetry-drawings} illustrates those of
most notable symmetry.

Finally, \cref{app:special-self-consistency} reports empirical
self-consistency checks illustrating the completeness of the 2-defective
search, and \cref{app:data-formats} demonstrates the enumeration and
classification steps on a worked example.

\section{Background and the optimality framework}
\label{sec:background}

A \emph{pseudoline arrangement} is a finite collection of unbounded
simple curves in the affine plane, the \emph{pseudolines}, in which every two
cross transversely in exactly one point.  Throughout this paper, we
consider \emph{simple} arrangements: no three pseudolines meet at a
common point (for the general theory see~\cite{grunbaum-1972,felsner-goodman-2017}).
We work primarily with affine arrangements; we pass to the
projective closure where convenient.

\subsection{Wiring diagrams and reduced words}
\label{subsec:wiring-diagrams}
Start with an arrangement of \(n\) lines in general position, and pick an
oriented sweep line (a straight line for straight lines, another
pseudoline for pseudolines), meeting the crossings one at a time.
Label the lines \(0,1,\ldots,n-1\) in the order they
meet the sweep at its initial position, then move the sweep through all
crossings.  When the next crossing exchanges the lines currently
occupying adjacent positions \(i\) and \(i+1\), we record the adjacent
transposition \(\sigma_i\), a standard generator of the symmetric group
\(S_n\).  The permutations obtained in this way form the
\emph{allowable sequence} of the arrangement~\cite{goodman-pollack-1993}.
The recorded generators
\(\sigma_i\) encode its consecutive transitions.

For a simple arrangement the sweep records \(\binom n2\) adjacent transpositions
and takes the initial order \(0,1,\ldots,n-1\) to the reverse order
\(n-1,\ldots,1,0\).  Equivalently, it gives a reduced word for the longest
permutation \(w_0\in S_n\).

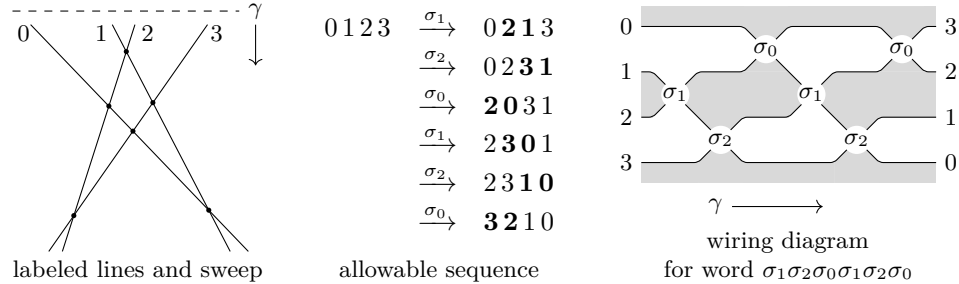
\begin{figure}[ht]
\centering
\begin{minipage}[t]{0.27\linewidth}
\centering
\begin{tikzpicture}[scale=1.2,baseline=(current bounding box.north)]
  \coordinate (L0a) at (-0.8,2.5);
  \coordinate (L0b) at (1.6,0);
  \coordinate (L1a) at (0.1,2.5);
  \coordinate (L1b) at (1.4,0);
  \coordinate (L2a) at (0.35,2.5);
  \coordinate (L2b) at (-0.45,0);
  \coordinate (L3a) at (1.15,2.5);
  \coordinate (L3b) at (-0.6,0);
  \coordinate (P01) at (intersection of L0a--L0b and L1a--L1b);
  \coordinate (P02) at (intersection of L0a--L0b and L2a--L2b);
  \coordinate (P03) at (intersection of L0a--L0b and L3a--L3b);
  \coordinate (P12) at (intersection of L1a--L1b and L2a--L2b);
  \coordinate (P13) at (intersection of L1a--L1b and L3a--L3b);
  \coordinate (P23) at (intersection of L2a--L2b and L3a--L3b);
  \draw (L0a) -- (L0b) node[pos=0.04,left,font=\small]{$0$};
  \draw (L1a) -- (L1b) node[pos=0.04,left,font=\small]{$1$};
  \draw (L2a) -- (L2b) node[pos=0.04,right,font=\small]{$2$};
  \draw (L3a) -- (L3b) node[pos=0.04,right,font=\small]{$3$};
  \foreach \p in {P01,P02,P03,P12,P13,P23}
    \fill (\p) circle (0.8pt);
  \draw[dashed] (-1.0,2.65) -- ++(2.5,0) node[right,font=\small]{$\gamma$}
    ++(0.18,-0.15) edge[solid, ->] ++(0,-0.45);
\end{tikzpicture}

\smallskip
\small labeled lines and sweep
\end{minipage}%
\hfill
\begin{minipage}[t]{0.27\linewidth}
\centering
\[
\renewcommand{\arraystretch}{1.22}
\begin{array}{rcl}
0\,1\,2\,3 & \xrightarrow{\sigma_1} & 0\,\mathbf{2\,1}\,3\\
            & \xrightarrow{\sigma_2} & 0\,2\,\mathbf{3\,1}\\
            & \xrightarrow{\sigma_0} & \mathbf{2\,0}\,3\,1\\
            & \xrightarrow{\sigma_1} & 2\,\mathbf{3\,0}\,1\\
            & \xrightarrow{\sigma_2} & 2\,3\,\mathbf{1\,0}\\
            & \xrightarrow{\sigma_0} & \mathbf{3\,2}\,1\,0
\end{array}
\]

\small
\smallskip
allowable sequence
\end{minipage}%
\hfill
\begin{minipage}[t]{0.37\linewidth}
\centering
\raisebox{-0.15cm}{%
\begin{tikzpicture}[x=0.6cm,y=0.6cm,baseline=(current bounding box.north)]
  \PseudolineWiringDiagram[fill]{4}{1,2,0,1,2,0}{}
  \draw[->] (2,-0.97) node[left,font=\small]{$\gamma$} -- ++(2,0);
\end{tikzpicture}
}

\smallskip
\small wiring diagram\\
for word $\sigma_1\sigma_2\sigma_0\sigma_1\sigma_2\sigma_0$
\end{minipage}

\caption{From a labeled line arrangement to an allowable sequence and to the
corresponding wiring diagram.}
\label{fig:sweep-word-wiring}
\end{figure}

Conversely, any reduced word for \(w_0\) can be drawn as a \emph{wiring diagram}
(also called a \emph{primitive sorting network}~\cite{knuth-1992}):
the labels start in the order \(0,1,\ldots,n-1\), and each generator
\(\sigma_i\) crosses the two wires currently occupying positions \(i\) and
\(i+1\).  The obtained diagram is itself a simple pseudoline arrangement
(up to homeomorphism of the plane); and, in the other direction, every pseudoline
arrangement is isomorphic to a wiring diagram~\cite{goodman-1980}.

A face of the wiring diagram is \emph{external} if it touches the
boundary and \emph{internal} otherwise.  In affine terms, external faces
are the unbounded faces of the arrangement and internal faces are the
bounded ones.  There are \(2n\) external faces.

Each crossing is a vertex of four faces.  Two of them lie between the
crossed wires.  The generator producing the crossing \emph{closes} the one
that ends at this crossing and \emph{opens} the new one that begins here.
The crossing is an intermediate vertex of the other two faces.
We call them the faces \emph{adjacent} to the generator.

The sweep can pass through two independent crossings in either order:
\begin{equation}
\label{eq:commute}
  \sigma_i\sigma_j=\sigma_j\sigma_i,\qquad |i-j|>1,
\end{equation}
which is the \emph{commutation} relation.
It leaves the wiring diagram unchanged, and we treat
commutation-equivalent words as encoding the same diagram.

The \emph{braid} relation
\begin{equation}
\label{eq:braid}
  \sigma_i\sigma_{i+1}\sigma_i = \sigma_{i+1}\sigma_i\sigma_{i+1}
\end{equation}
reverses three consecutive crossings among three wires.  This changes
the face structure and produces a different arrangement, so we do not
treat it as an equivalence.

\begin{observation}[Arrangements, diagrams, and words]
\label{obs:representation-multiplicity}
Each pseudoline arrangement is represented by several wiring
diagrams, and each wiring diagram is encoded by several reduced words for
\(w_0\), related by commutations~\eqref{eq:commute} (see
\cref{lem:o-matrix-commutation} for a more rigorous formulation).  Up to \(4n\) wiring
diagrams represent the same arrangement: \(2n\) choices of sweep start,
combined with the vertical reflection
\(\sigma_i\mapsto\sigma_{n-2-i}\) of the diagram.
\end{observation}

\subsection{Checkerboard coloring and projective closure}
\label{subsec:checkerboard}

The faces of a pseudoline arrangement admit a checkerboard coloring;
this can be proved by induction on \(n\).

In a wiring diagram,
at the start of the sweep, before any crossing, there are \(n+1\)
external faces, colored alternately.  Each subsequent generator opens
one new face, whose color depends only on the parity
of the generator index (\cref{fig:sweep-word-wiring}).

The spherical model represents the projective plane as a sphere with
antipodal points identified, with the line at infinity as the equator.
An affine face and its antipodal counterpart in the other hemisphere
are separated by \(n\) affine lines, hence carry opposite
colors for odd \(n\) and the same color for even \(n\).  For odd \(n\),
we pass to the projective closure by adding the line at infinity as a
line of the arrangement; this ensures the required color change at the
equator.

\subsection{Triangle bound and unused segments}

Let \(\mathcal{A}\) be a simple affine arrangement of \(n\) pseudolines, and let
\(a_3(\mathcal{A})\) be the number of bounded triangular faces.  Each pseudoline is cut
by its \(n-1\) crossings into \(n\) segments, of which \(n-2\) are bounded.
Thus \(\mathcal{A}\) has \(n(n-2)\) bounded segments.

Each bounded triangular face is bordered by three bounded segments.  A bounded
segment of \(\mathcal{A}\) borders at most one triangular face.  If two
triangles bordered it, their other sides would lie on the pseudolines \(L_1\) and
\(L_2\) through its endpoints, so \(L_1\) and \(L_2\) would cross twice, once at
each triangle's apex.  Hence
\begin{equation}
\label{eq:a3-bound}
  a_3(\mathcal{A})\le \left\lfloor \frac{n(n-2)}{3}\right\rfloor .
\end{equation}
This is the standard segment-count upper bound
\cite{blanc-2011}.

A bounded segment is \emph{used} if it borders a triangular face, and
\emph{unused} otherwise.  If \(u(\mathcal{A})\) denotes the number of
unused bounded segments, then
\[
  3a_3(\mathcal{A})+u(\mathcal{A})=n(n-2).
\]

Consider the divisibility of \(n(n-2)\) by \(3\).
For \(n\equiv 3,5\pmod 6\), the upper bound is compatible with
\(u(\mathcal{A})=0\).  A simple arrangement with \(u(\mathcal{A})=0\) is called
\emph{perfect}: all bounded segments are used by triangular faces
\cite{bartholdi-blanc-loisel-2007,blanc-2011}.
For \(n\equiv 1\pmod 6\), \(n(n-2)\) is not divisible by \(3\), so
perfect arrangements do not exist.  The maximum \(a_3(\mathcal{A})=
\lfloor n(n-2)/3\rfloor\) is attained with exactly two unused
segments.

\begin{lemma}[Triangle color in perfect arrangements]
\label{lem:perfect-color}
In a perfect arrangement, all bounded triangles share one color of the
checkerboard coloring, and all bounded non-triangular faces share the
other.
\end{lemma}

\begin{proof}
\emph{Alternation.}\enspace Consider a pseudoline \(L\) with no unused segments,
and two triangles \(t_1,t_2\) incident to two consecutive segments of \(L\)
sharing a crossing \(v\).  Here \(v\) is the crossing of \(L\)
with another pseudoline \(N\).  If \(t_1,t_2\) were on the same side of \(L\),
they would both border the segment of \(N\) with the vertex \(v\) on that side of \(L\),
so the segment would be used by two triangles (excluded above).
Hence the triangles incident to \(L\) alternate between its two sides;
consecutive ones, being vertically opposite at the shared crossing, share the
same color of the checkerboard coloring. Therefore all triangles
along \(L\) share a single color, which we denote \(c(L)\).

\emph{Propagation.}\enspace If two pseudolines \(L_1,L_2\) without unused
segments cross, their crossing point is a vertex of a bounded triangle
with sides on both pseudolines, so the triangle color is \(c(L_1)=c(L_2)\).


In a perfect arrangement every pseudoline has all its bounded
segments used, so all bounded triangles share one color, and all bounded
non-triangular faces (incident to triangles across a segment) share the other.
\end{proof}

\begin{figure}[ht]
\centering
\begin{subfigure}[b]{0.49\linewidth}
\centering
\begin{tikzpicture}[
  font=\footnotesize,
  arr/.style={-{Stealth[length=3pt,width=1.8pt]}, very thin, shorten >=2pt}
]
\useasboundingbox (0,0) rectangle (\linewidth,-0.91397\linewidth);
\node[inner sep=0, anchor=north west] (S7) at (0,0.08\linewidth)
  {
\begin{tikzpicture}[x=\linewidth,y=\linewidth,
  poly/.style={fill=gray!20},
  ext/.style={black},
  lbl/.style={inner sep=0pt,font=\scriptsize},
  plbl/.style={inner sep=0pt,font=\scriptsize}]
\useasboundingbox (0,0) rectangle (1,-0.91397);
\clip (0,0) rectangle (1,-0.91397);
\fill[poly] (0.45510,-0.33393) -- (0.54842,-0.16889) -- (0.54457,-0.32738) -- cycle;
\fill[poly] (0.27896,-0.36540) -- (0.22366,-0.23771) -- (0.35895,-0.34789) -- cycle;
\fill[poly] (0.31217,-0.43815) -- (0.13009,-0.40689) -- (0.27896,-0.36540) -- cycle;
\fill[poly] (0.42000,-0.40057) -- (0.35895,-0.34789) -- (0.45510,-0.33393) -- cycle;
\fill[poly] (0.36354,-0.54412) -- (0.31217,-0.43815) -- (0.39931,-0.44677) -- cycle;
\fill[poly] (0.47080,-0.44795) -- (0.39931,-0.44677) -- (0.42000,-0.40057) -- cycle;
\fill[poly] (0.52768,-0.50117) -- (0.47080,-0.44795) -- (0.53518,-0.44776) -- cycle;
\fill[fill=red!22] (0.66360,-0.45215) -- (0.53518,-0.44776) -- (0.54457,-0.32738) -- (0.78349,-0.31896) -- cycle; 
\fill[poly] (0.51074,-0.59734) -- (0.52768,-0.50117) -- (0.57505,-0.54292) -- cycle;
\fill[poly] (0.29517,-0.74634) -- (0.36354,-0.54412) -- (0.42358,-0.65974) -- cycle;
\fill[poly] (0.48329,-0.77157) -- (0.42358,-0.65974) -- (0.51074,-0.59734) -- cycle;
\fill[poly] (0.14275,-0.17183) -- (0.22366,-0.23771) -- (0.18219,-0.14197) -- cycle;
\fill[poly] (0.55095,-0.06458) -- (0.54842,-0.16889) -- (0.59978,-0.07806) -- cycle;
\fill[poly] (0.85329,-0.24141) -- (0.78349,-0.31896) -- (0.88776,-0.31529) -- cycle;
\fill[fill=red!22] (0.76787,-0.45571) -- (0.66360,-0.45215) -- (0.57505,-0.54292) -- (0.65334,-0.61190) -- cycle; 
\fill[poly] (0.26175,-0.84519) -- (0.29517,-0.74634) -- (0.20866,-0.80468) -- cycle;
\fill[poly] (0.53243,-0.86362) -- (0.48329,-0.77157) -- (0.46705,-0.87464) -- cycle;
\fill[poly] (0.02958,-0.43490) -- (0.13009,-0.40689) -- (0.02726,-0.38924) -- cycle;
\draw[ext] (0.00155,-0.38482) -- (0.13009,-0.40689) -- (0.31217,-0.43815) -- (0.39931,-0.44677) -- (0.47080,-0.44795) -- (0.53518,-0.44776) -- (0.66360,-0.45215) -- (0.79394,-0.45660);
\draw[ext] (0.12252,-0.15535) -- (0.22366,-0.23771) -- (0.35895,-0.34789) -- (0.42000,-0.40057) -- (0.47080,-0.44795) -- (0.52768,-0.50117) -- (0.57505,-0.54292) -- (0.67291,-0.62915);
\draw[ext] (0.17182,-0.11803) -- (0.22366,-0.23771) -- (0.27896,-0.36540) -- (0.31217,-0.43815) -- (0.36354,-0.54412) -- (0.42358,-0.65974) -- (0.48329,-0.77157) -- (0.54472,-0.88663);
\draw[ext] (0.55159,-0.03850) -- (0.54842,-0.16889) -- (0.54457,-0.32738) -- (0.53518,-0.44776) -- (0.52768,-0.50117) -- (0.51074,-0.59734) -- (0.48329,-0.77157) -- (0.46299,-0.90041);
\draw[ext] (0.61262,-0.05536) -- (0.54842,-0.16889) -- (0.45510,-0.33393) -- (0.42000,-0.40057) -- (0.39931,-0.44677) -- (0.36354,-0.54412) -- (0.29517,-0.74634) -- (0.25340,-0.86990);
\draw[ext] (0.87075,-0.22203) -- (0.78349,-0.31896) -- (0.66360,-0.45215) -- (0.57505,-0.54292) -- (0.51074,-0.59734) -- (0.42358,-0.65974) -- (0.29517,-0.74634) -- (0.18703,-0.81927);
\draw[ext] (0.91383,-0.31437) -- (0.78349,-0.31896) -- (0.54457,-0.32738) -- (0.45510,-0.33393) -- (0.35895,-0.34789) -- (0.27896,-0.36540) -- (0.13009,-0.40689) -- (0.00445,-0.44191);
\end{tikzpicture}};
\coordinate (T-intdef) at ([xshift= 0.675\linewidth, yshift=-0.40\linewidth]S7.north west);
\coordinate (T-extdef) at ([xshift= 0.67\linewidth, yshift=-0.526\linewidth]S7.north west);
\node[anchor=south, inner sep=1pt] (L-extdig)
  at ([xshift= 0.32\linewidth, yshift=-0.08\linewidth] S7.north west) {external digons};
\node[anchor=south east, inner sep=1pt, align=right] (L-intdef)
  at ([xshift= 0.96\linewidth, yshift=-0.45\linewidth] S7.north west) {internal\\defect};
\node[anchor=east, inner sep=1pt] (L-extdef)
  at ([xshift= 0.96\linewidth, yshift=-0.70\linewidth] S7.north west) {external defect};
\foreach \tx/\ty in {0.17/-0.186, 0.585/-0.10}
  \draw[arr] (L-extdig) -- ([xshift=\tx\linewidth, yshift=\ty\linewidth] S7.north west);
\draw[arr] (L-intdef) -- (T-intdef);
\draw[arr] (L-extdef) -- (T-extdef);
\coordinate (C-unused) at ([xshift=0.96\linewidth, yshift=-0.15\linewidth]S7.north west);
\node[anchor=east, inner sep=1pt] (L-unused) at (C-unused) {unused segments};
\coordinate (T-u1) at ([xshift=0.631\linewidth, yshift=-0.335\linewidth]S7.north west);
\coordinate (T-u2) at ([xshift=0.533\linewidth, yshift=-0.392\linewidth]S7.north west);
\coordinate (T-u3) at ([xshift=0.736\linewidth, yshift=-0.38\linewidth]S7.north west);
\coordinate (T-u4) at ([xshift=0.580\linewidth, yshift=-0.460\linewidth]S7.north west);
\coordinate (T-u5) at ([xshift=0.616\linewidth, yshift=-0.510\linewidth]S7.north west);
\foreach \t in {T-u1,T-u2,T-u3,T-u4,T-u5}
  \draw[arr] (L-unused) -- (\t);
\end{tikzpicture}
\end{subfigure}\hfill
\begin{subfigure}[b]{0.49\linewidth}
\centering
\begin{tikzpicture}[x=\linewidth,y=\linewidth,
  poly/.style={fill=gray!20},
  ext/.style={black},
  lbl/.style={inner sep=0pt,font=\scriptsize},
  plbl/.style={inner sep=0pt,font=\scriptsize}]
\useasboundingbox (0,0) rectangle (1,-0.91397);
\clip (0,0) rectangle (1,-0.91397);
\fill[poly] (0.55631,-0.25913) -- (0.65165,-0.11879) -- (0.65365,-0.23053) -- cycle;
\fill[poly] (0.40934,-0.31290) -- (0.37431,-0.15353) -- (0.47000,-0.28795) -- cycle;
\fill[poly] (0.24038,-0.40254) -- (0.15908,-0.32512) -- (0.32895,-0.35292) -- cycle;
\fill[poly] (0.31476,-0.47449) -- (0.14449,-0.45835) -- (0.24038,-0.40254) -- cycle;
\fill[poly] (0.42342,-0.36940) -- (0.32895,-0.35292) -- (0.40934,-0.31290) -- cycle;
\fill[poly] (0.50190,-0.33751) -- (0.47000,-0.28795) -- (0.55631,-0.25913) -- cycle;
\fill[poly] (0.44147,-0.43082) -- (0.42342,-0.36940) -- (0.47382,-0.37952) -- cycle;
\fill[poly] (0.38435,-0.53706) -- (0.31476,-0.47449) -- (0.41050,-0.48487) -- cycle;
\fill[poly] (0.46094,-0.48954) -- (0.41050,-0.48487) -- (0.44147,-0.43082) -- cycle;
\fill[poly] (0.53345,-0.39321) -- (0.47382,-0.37952) -- (0.50190,-0.33751) -- cycle;
\fill[poly] (0.65407,-0.33474) -- (0.65365,-0.23053) -- (0.76142,-0.19808) -- cycle;
\fill[poly] (0.56174,-0.44836) -- (0.53345,-0.39321) -- (0.59443,-0.40943) -- cycle;
\fill[poly] (0.47865,-0.53860) -- (0.46094,-0.48954) -- (0.52161,-0.49355) -- cycle;
\fill[poly] (0.31693,-0.68836) -- (0.38435,-0.53706) -- (0.43736,-0.57843) -- cycle;
\fill[poly] (0.51395,-0.62958) -- (0.43736,-0.57843) -- (0.47865,-0.53860) -- cycle;
\fill[poly] (0.58420,-0.49501) -- (0.52161,-0.49355) -- (0.56174,-0.44836) -- cycle;
\fill[poly] (0.65002,-0.42663) -- (0.59443,-0.40943) -- (0.65407,-0.33474) -- cycle;
\fill[poly] (0.62515,-0.58366) -- (0.58420,-0.49501) -- (0.64223,-0.49304) -- cycle;
\fill[poly] (0.57833,-0.79174) -- (0.51395,-0.62958) -- (0.60261,-0.68265) -- cycle;
\fill[poly] (0.69867,-0.73797) -- (0.60261,-0.68265) -- (0.62515,-0.58366) -- cycle;
\fill[poly] (0.80668,-0.47940) -- (0.64223,-0.49304) -- (0.65002,-0.42663) -- cycle;
\fill[poly] (0.06471,-0.30968) -- (0.15908,-0.32512) -- (0.08983,-0.25918) -- cycle;
\fill[poly] (0.31886,-0.07563) -- (0.37431,-0.15353) -- (0.35378,-0.06014) -- cycle;
\fill[poly] (0.64994,-0.02318) -- (0.65165,-0.11879) -- (0.70538,-0.03969) -- cycle;
\fill[poly] (0.82049,-0.12289) -- (0.76142,-0.19808) -- (0.85298,-0.17052) -- cycle;
\fill[poly] (0.27800,-0.77570) -- (0.31693,-0.68836) -- (0.24630,-0.75282) -- cycle;
\fill[poly] (0.61361,-0.88061) -- (0.57833,-0.79174) -- (0.55756,-0.88508) -- cycle;
\fill[poly] (0.78153,-0.78569) -- (0.69867,-0.73797) -- (0.73980,-0.82429) -- cycle;
\fill[poly] (0.90197,-0.47150) -- (0.80668,-0.47940) -- (0.89730,-0.50993) -- cycle;
\fill[poly] (0.06185,-0.50645) -- (0.14449,-0.45835) -- (0.04930,-0.44933) -- cycle;
\draw[ext] (0.02550,-0.44707) -- (0.14449,-0.45835) -- (0.31476,-0.47449) -- (0.41050,-0.48487) -- (0.46094,-0.48954) -- (0.52161,-0.49355) -- (0.58420,-0.49501) -- (0.64223,-0.49304) -- (0.80668,-0.47940) -- (0.92580,-0.46952);
\draw[ext] (0.04112,-0.30582) -- (0.15908,-0.32512) -- (0.32895,-0.35292) -- (0.42342,-0.36940) -- (0.47382,-0.37952) -- (0.53345,-0.39321) -- (0.59443,-0.40943) -- (0.65002,-0.42663) -- (0.80668,-0.47940) -- (0.91995,-0.51756);
\draw[ext] (0.07252,-0.24270) -- (0.15908,-0.32512) -- (0.24038,-0.40254) -- (0.31476,-0.47449) -- (0.38435,-0.53706) -- (0.43736,-0.57843) -- (0.51395,-0.62958) -- (0.60261,-0.68265) -- (0.69867,-0.73797) -- (0.80224,-0.79762);
\draw[ext] (0.30499,-0.05616) -- (0.37431,-0.15353) -- (0.47000,-0.28795) -- (0.50190,-0.33751) -- (0.53345,-0.39321) -- (0.56174,-0.44836) -- (0.58420,-0.49501) -- (0.62515,-0.58366) -- (0.69867,-0.73797) -- (0.75008,-0.84587);
\draw[ext] (0.34865,-0.03679) -- (0.37431,-0.15353) -- (0.40934,-0.31290) -- (0.42342,-0.36940) -- (0.44147,-0.43082) -- (0.46094,-0.48954) -- (0.47865,-0.53860) -- (0.51395,-0.62958) -- (0.57833,-0.79174) -- (0.62243,-0.90283);
\draw[ext] (0.64951,0.00072) -- (0.65165,-0.11879) -- (0.65365,-0.23053) -- (0.65407,-0.33474) -- (0.65002,-0.42663) -- (0.64223,-0.49304) -- (0.62515,-0.58366) -- (0.60261,-0.68265) -- (0.57833,-0.79174) -- (0.55236,-0.90841);
\draw[ext] (0.71881,-0.01992) -- (0.65165,-0.11879) -- (0.55631,-0.25913) -- (0.50190,-0.33751) -- (0.47382,-0.37952) -- (0.44147,-0.43082) -- (0.41050,-0.48487) -- (0.38435,-0.53706) -- (0.31693,-0.68836) -- (0.26827,-0.79753);
\draw[ext] (0.83526,-0.10409) -- (0.76142,-0.19808) -- (0.65407,-0.33474) -- (0.59443,-0.40943) -- (0.56174,-0.44836) -- (0.52161,-0.49355) -- (0.47865,-0.53860) -- (0.43736,-0.57843) -- (0.31693,-0.68836) -- (0.22864,-0.76893);
\draw[ext] (0.87587,-0.16363) -- (0.76142,-0.19808) -- (0.65365,-0.23053) -- (0.55631,-0.25913) -- (0.47000,-0.28795) -- (0.40934,-0.31290) -- (0.32895,-0.35292) -- (0.24038,-0.40254) -- (0.14449,-0.45835) -- (0.04119,-0.51848);
\end{tikzpicture}
\end{subfigure}
\caption{A 2-defective arrangement for $n=7$ and a perfect arrangement for $n=9$.}
\label{fig:n7-n9-samples}
\end{figure}
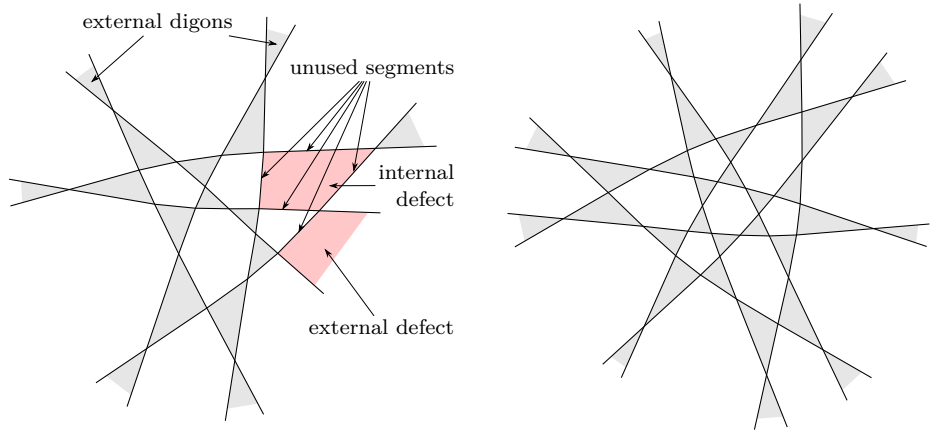

\subsection{Black-face bound and defects}
\label{subsec:black-face-bound}

There are two checkerboard colorings, differing by a swap of colors.
We choose the one in which the black faces are at least as numerous as
the white faces.  Then in a perfect arrangement the bounded triangles are
black and the adjacent non-triangular faces are white (\cref{fig:n7-n9-samples}).
Indeed, by \cref{lem:perfect-color} the triangles share one color, and there are
\(\tfrac13 n(n-2)\) of them among the \(\binom{n-1}{2}\) internal faces, that is,
more than half.  The external faces split evenly, so that color is black.

\begin{lemma}[External faces in perfect arrangements]
\label{lem:external-digons}
In a perfect arrangement, every external black face is an
external digon (a face between 2 unbounded segments).
\end{lemma}

\begin{proof}
Suppose an external black face contains a bounded segment.  This segment
lies between two non-triangular faces (a white face and the external
one) and is therefore unused, contradicting perfectness.  Hence the
external black face has only unbounded sides and is an external digon.
\end{proof}

To measure how a general arrangement deviates from this triangular
pattern, we count non-triangular black faces, treating internal and
external cases uniformly.  This is achieved by passing to the projective closure: bounded
triangles remain triangles, and external digons become
triangles whose third side lies on the line at infinity.  Being triangular
is then a property of the projective face, independent of which line plays
the role of the line at infinity.

\begin{definition}
\label{def:defect}
  A black face is a \emph{defect} if its image in the projective closure
  of the arrangement is not a triangle.  In affine terms, a defect is a
  black face which is neither a bounded triangle nor an external digon.
\end{definition}

Let \(b_{\mathrm{int}}(\mathcal{A})\) be the number of internal black faces.  Each
internal black face has at least three bounded sides, and each bounded
segment is incident to exactly one black face.  The same segment-count
argument as for triangles therefore gives
\begin{equation}
\label{eq:bint-bound}
  b_{\mathrm{int}}(\mathcal{A})\le \left\lfloor \frac{n(n-2)}{3}\right\rfloor .
\end{equation}

The \(2n\) external faces are colored alternately, so exactly \(n\) of
them are black.  The total number of black faces is
therefore
\begin{equation}
\label{eq:black-total}
  b(\mathcal{A}) = b_{\mathrm{int}}(\mathcal{A}) + n.
\end{equation}

\begin{lemma}[Defect-free odd arrangements are exactly the perfect ones]
\label{lem:defect-free-perfect}
For odd \(n\), a simple affine arrangement of \(n\) pseudolines is
perfect if and only if it is defect-free; in particular, defect-free
arrangements do not exist for \(n\equiv 1\pmod 6\).
\end{lemma}

\begin{proof}
By \cref{lem:perfect-color,lem:external-digons}, a perfect arrangement
has no defects.  Conversely, let
\(\mathcal{A}\) be defect-free.  Then every internal black face is
a triangle and every external black face is an external digon.  Each
bounded segment is incident to exactly one black face.  This face is a
triangle: an external digon has no bounded sides, so it is internal.
Each internal black face contributes exactly three
bounded sides, hence
\[
  3\, b_{\mathrm{int}}(\mathcal{A}) = n(n-2).
\]
For \(n\equiv 1\pmod 6\), \(n(n-2)\equiv 2\pmod 3\) and this equation has
no solution; so no defect-free arrangement exists.

For \(n\equiv 3,5\pmod 6\), the equation gives
\(b_{\mathrm{int}}(\mathcal{A}) = \lfloor n(n-2)/3\rfloor\).  Since every
internal black face is a triangle, \(a_3(\mathcal{A}) \ge
b_{\mathrm{int}}(\mathcal{A})\); together with the upper bound on \(a_3\),
this gives equality.  Hence \(u(\mathcal{A})=0\) and \(\mathcal{A}\) is
perfect.
\end{proof}

\begin{lemma}[No white triangles in optimal arrangements]
\label{lem:no-white-triangles}
If a simple arrangement \(\mathcal{A}\) attains the bound \(b_{\mathrm{int}}(\mathcal{A}) = \lfloor n(n-2)/3\rfloor\), every bounded
triangle of \(\mathcal{A}\) is black.
\end{lemma}

\begin{proof}
Suppose that \(\mathcal{A}\) has a bounded white triangle \(t\), and
flip it: push one side of \(t\) lying on a pseudoline
\(L\) across the opposite vertex, meeting no other element of
\(\mathcal{A}\).  The flip replaces \(t\) by a triangle across the vertex,
cut off by \(L\) from the face vertically opposite \(t\) and therefore white;
so the new triangle is black, while every other face keeps its color.
Thus \(b_{\mathrm{int}}\) increases by one, contradicting~\eqref{eq:bint-bound}.
\end{proof}

\subsection{Optimal arrangements for \texorpdfstring{$n\equiv 1\pmod 6$}{n = 1 (mod 6)}}
\label{subsec:max-1mod6}

In the absence of perfect arrangements, we establish the structural
properties of arrangements attaining the upper
bounds~\eqref{eq:a3-bound} and~\eqref{eq:bint-bound}.

\begin{lemma}[Defect profiles for $n\equiv 1\pmod 6$]
\label{lem:optimal-defects-2q1p}
For \(n\equiv 1\pmod 6\), an arrangement attaining the upper
bound~\eqref{eq:bint-bound} has, in its projective closure, either one
pentagonal defect or two quadrilateral defects.
\end{lemma}

\begin{proof}
Let \(s(f)\) denote the number of sides of the projective closure of a
black face \(f\).  Each external black face contributes three sides
beyond its bounded segments: two former unbounded rays and one side on
the line at infinity.  Hence
\[
  \sum\nolimits_{f \text{ black}} s(f) = n(n-2) + 3n .
\]
Subtracting 3 for each of the \(b(\mathcal{A})\) black faces and
using~\eqref{eq:black-total},
\[
  \sum\nolimits_{f \text{ black}} \bigl(s(f) - 3\bigr)
  = n(n+1) - 3\,b(\mathcal{A})
  = n(n-2) - 3\,b_{\mathrm{int}}(\mathcal{A}).
\]
If \(\mathcal{A}\) attains the bound~\eqref{eq:bint-bound}, the
right-hand side equals \(2\).  Split the sum:
\[
  \sum\nolimits_{f \text{ black}} \bigl(s(f) - 3\bigr)
  = \sum\nolimits_{s(f) > 3} \bigl(s(f) - 3\bigr)
  + \sum\nolimits_{s(f) = 3} \bigl(s(f) - 3\bigr)
  = 2.
\]
By \cref{def:defect}, the first sum runs over the defects; the second vanishes.
Each defect contributes at least \(1\), so either two defects are
projective quadrilaterals (\(s(f_1)=s(f_2)=4\)) or a single defect is a
projective pentagon (\(s(f_1)=5\)).
\end{proof}

\begin{definition}
\label{def:2-defective}
A simple arrangement of \(n\) pseudolines with \(n\equiv 1\pmod 6\) is
\emph{2-defective} if it attains the bound~\eqref{eq:bint-bound} on
internal black faces.
\end{definition}

\noindent
The name refers to the side excess \(\sum_{f}\bigl(s(f)-3\bigr)=2\),
not to the number of defects.

\begin{lemma}[Simultaneous attainment of the bounds]
\label{lem:a3-max-characterization}
For \(n\equiv 1\pmod 6\) and a simple affine arrangement \(\mathcal{A}\) of
\(n\) pseudolines,
\[
  a_3(\mathcal{A}) = \left\lfloor\frac{n(n-2)}{3}\right\rfloor
  \iff
  b_{\mathrm{int}}(\mathcal{A}) = \left\lfloor\frac{n(n-2)}{3}\right\rfloor
  \text{ and }\mathcal{A}\text{ has no internal defects}.
\]
\end{lemma}

\begin{proof}
\((\Leftarrow)\)\enspace
Assume \(b_{\mathrm{int}}(\mathcal{A}) = \lfloor n(n-2)/3\rfloor\) and
\(\mathcal{A}\) has no internal defects.  Then every internal black face
is a triangle, so \(a_3(\mathcal{A}) \ge b_{\mathrm{int}}(\mathcal{A}) =
\lfloor n(n-2)/3\rfloor\); combined with~\eqref{eq:a3-bound},
\(a_3(\mathcal{A}) = \lfloor n(n-2)/3\rfloor\).

\((\Rightarrow)\)\enspace
Assume \(a_3(\mathcal{A}) = \lfloor n(n-2)/3\rfloor\). Then exactly two
bounded segments are unused.  Call a pseudoline \emph{bad} if it contains
an unused bounded segment, and \emph{good} otherwise.
There are at most two bad pseudolines.

A good pseudoline has all its bounded segments used, so by the alternation
and propagation arguments of \cref{lem:perfect-color} the colors of all
triangles along all good pseudolines coincide; denote this color by \(c\).

Every bounded triangle has at most two of its three sides on bad
pseudolines and so is incident to at least one good pseudoline.  Hence
every bounded triangle has the common color \(c\).  Among the
\(\binom{n-1}{2}\) internal faces, \(\lfloor n(n-2)/3\rfloor\) are
triangles; this is more than half, so \(c\) is black.

Therefore \(a_3 = a_3^{\text{black}} \le b_{\mathrm{int}}(\mathcal{A})\), and
combined with~\eqref{eq:bint-bound} this gives
\(b_{\mathrm{int}}(\mathcal{A}) = \lfloor n(n-2)/3\rfloor = a_3(\mathcal{A})\),
so \(\mathcal{A}\) attains the bound and every internal black face is a
triangle.  Hence \(\mathcal{A}\) has no internal defects.
\end{proof}

\section{Problem statement, prior work, and our approach}
\label{sec:problem}

\subsection{Problem statement}
\label{subsec:problem-statement}

We study the following enumeration problem.

\begin{problem}
\label{prob:main}
For odd \(n\), enumerate all simple affine arrangements of \(n\)
pseudolines that attain the upper bound~\eqref{eq:bint-bound} on the
number of internal black faces.
\end{problem}

By \cref{lem:defect-free-perfect,lem:optimal-defects-2q1p}, the arrangements
in question split into two families: the perfect ones for
\(n\equiv 3,5\pmod 6\), and the 2-defective ones for
\(n\equiv 1\pmod 6\).  The first family solves both the Kobon and
Arnold problems described in the introduction.  The second family is
Arnold-optimal.  Some of its members are Kobon-optimal as well, possibly
after additional transformations (\cref{subsec:triangle-maximal-classes}).

The even case lies outside the scope of this paper.  Its optimal
arrangements involve deviations from simple arrangements: pairs of parallel
lines in the Arnold problem, and triple intersection points in the Kobon
problem.  Simple even arrangements, moreover, have many unavoidable
defects, which our algorithms do not handle.

\subsection{Prior enumeration algorithms}
\label{subsec:prior-algorithms}

The number of simple arrangements of \(n\) pseudolines grows as
\(2^{\Theta(n^2)}\)~\cite{knuth-1992}, so exhaustive search is feasible
only with substantial pruning.  We compare prior algorithms by how this
pruning is organized and by what structure the search is
built on: either the order in which the lines cross a given line
(the \(O\)-matrix, \cref{def:order-matrix}), or a word encoding
a wiring diagram.

Bokowski, Roudneff, and Strempel~\cite{bokowski-roudneff-strempel-1997}
enumerate perfect arrangements on the projective plane.  The
triangle condition propagates through the \(O\)-matrix as constraints
between rows, so non-perfect partials are ruled out during construction.

Bartholdi, Blanc, and Loisel~\cite{bartholdi-blanc-loisel-2007}
describe a depth-first search in the affine plane that grows a
wiring diagram one column of crossings at a time, with a budget on
unused segments for early pruning.

Wood~\cite{wood-2024} treats the perfect case with zero tolerance
for defects: a ``primary'' swap of two adjacent wires is admitted
only when the ``secondary'' swaps that resolve the enclosed
polygon into triangles are still available.  We arrived independently
at a similar idea of grouping elementary generators into composite
ones, and we choose this grouping to define a canonical form for the
enumerated words.

Savchuk~\cite{savchuk-2025} uses a table encoding similar
to the $O$-matrix and translates the optimality constraints
into clauses for a general-purpose SAT solver.

\subsection{Our approach}
\label{subsec:our-approach}

A straightforward pruning rule for the
depth-first search is a defect budget.  This was also our starting point.
It is fully local in the perfect case: the budget is zero,
so a branch dies as soon as a defect forms.
For 2-defective arrangements the same mechanism is much less local.
The budget allowing two internal quadrilateral defects cannot prune a branch
with extra trailing defects that become visible only after the word is built.
Instead, our approach either entirely disallows branches producing defects
(the perfect search of
\cref{sec:perfect-search}) or admits exactly the required defect profile
(the 2-defective search of \cref{sec:2-defective-search}), making the defect
budget unnecessary.

The search iterates over the
generators in a reduced word for the longest permutation \(w_0\).  Thus its natural
output is a stream of reduced words.  This representation is compact compared
to allowable sequences and $O$-matrices,
so we use it as the common input--output format of the enumeration
and classification programs.

\section{The perfect search}
\label{sec:perfect-search}

This section presents an exhaustive search for defect-free arrangements
of an odd number of pseudolines.  By \cref{lem:defect-free-perfect}, this
search enumerates exactly the perfect arrangements.  The search exploits the
local constraints of defect-freeness: every internal black face is a
triangle (\cref{lem:perfect-color}) and every external black face is a
digon (\cref{lem:external-digons}).  It also factors out
the commutation relation~\eqref{eq:commute}, enumerating one canonical
representative per commutation class of reduced words.

The search factors out only part of the geometric freedom of
\cref{obs:representation-multiplicity}.  In \cref{subsec:checkerboard} we
noted that the parity of a generator determines the color of the
face it opens in the wiring diagram, without fixing which parity
gives which color.  We now fix this convention: odd generators open black faces,
while even generators open white ones.  The reflection
$\sigma_i \mapsto \sigma_{n-2-i}$ inverts parity for odd
$n$, so this keeps one diagram from each reflection pair, leaving $2n$ of the
$4n$ wiring diagrams per arrangement.

\begin{figure}[ht]
\centering
\begin{subfigure}[b]{0.35\linewidth}
\centering
\begin{tikzpicture}[x=0.6cm,y=0.6cm]
  \PseudolineWiringDiagram[fill,noright]{9}{{7,5,3,1},dots,0,dots,{7,5,3,1}}{}
\end{tikzpicture}
\caption{}
\label{fig:perfect-schematic-digons}
\end{subfigure}\hfill
\begin{subfigure}[b]{0.27\linewidth}
\centering
\begin{tikzpicture}[x=0.6cm,y=0.6cm]
  \PseudolineWiringDiagram[fill,noright]{9}{{7,5,3,1},4,{3,5}}{}
\end{tikzpicture}
\caption{}
\label{fig:composite-K}
\end{subfigure}\hfill
\begin{subfigure}[b]{0.35\linewidth}
\centering
\begin{tikzpicture}[x=0.6cm,y=0.6cm]
  \PseudolineWiringDiagram[fill,noright,nolabels]{9}{{7,5,3,1},4,{3,5},4,{3,5}}{wire/3,wire/4,gen/2/4,gen/4/4}
\end{tikzpicture}
\caption{}
\label{fig:no-repeat}
\end{subfigure}
\caption{Wiring diagram patterns for $n=9$.  (a)~The forced odd prefix
$\sigma_1\sigma_3\sigma_5\sigma_7$, the single $\sigma_0$, and the
trailing generators.
(b)~A composite generator $K_4=\sigma_4\sigma_3\sigma_5$:
the white $\sigma_4$, with $\sigma_3$ and $\sigma_5$ closing two black
triangles.
(c)~Two consecutive $K_4$ form a white quadrilateral; its four black
neighbors cannot all be triangles, since otherwise the highlighted
wires would cross twice (\cref{lem:composite-no-repeat}).}
\label{fig:perfect-schematic}
\end{figure}
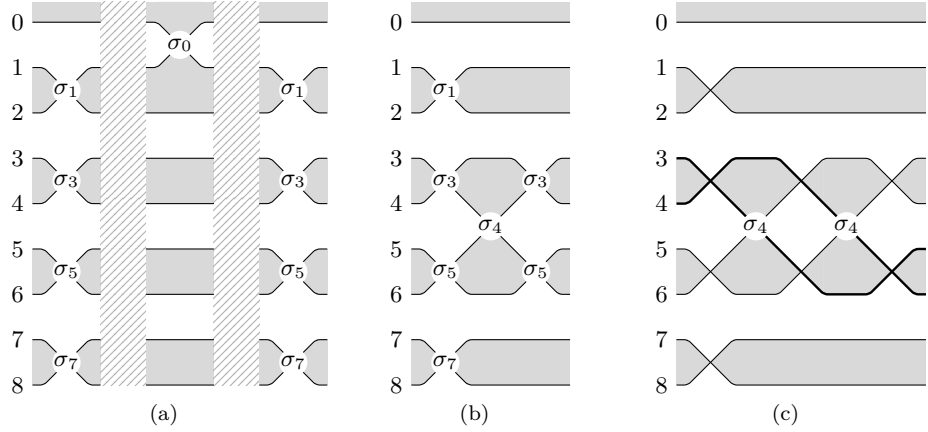

\subsection{Constraints from external black digons}
\label{subsec:external-digons}

In a perfect arrangement every external black face is a digon
(\cref{lem:external-digons}; see \cref{fig:perfect-schematic-digons}).
The digons are alike, but the sweep direction makes them
constrain the search in three different ways.  The initial digons fix a
starting prefix of the word, removing all branching at the top of the
search.  The digon associated with $\sigma_0$ forces that generator to occur
exactly once.  The trailing ones are formed at the end of the word,
but they still forbid certain pairs of wires from crossing in the interior.

\begin{lemma}[Forced prefix from initial digons]
\label{lem:initial-digons}
The initial external black digons of a perfect arrangement of $n$
pseudolines, $n$ odd, are bounded by the wire pairs
$$(1,2),\ (3,4),\ \ldots,\ (n-2,n-1).$$
They are closed by the letters of the prefix
$\sigma_1\sigma_3\cdots\sigma_{n-2}$ (up to commutations) of the word
encoding the wiring diagram.
\end{lemma}

\begin{proof}
Immediate from \cref{lem:external-digons} and the parity-color
convention.  The pairs are read off the start of the diagram; any
other letter at the start would add an extra vertex to one of the open
digons, making it a face with more than two sides.
\end{proof}

\begin{lemma}[The single $\sigma_0$]
\label{lem:sigma-zero}
In a perfect arrangement of an odd number of pseudolines, the generator
$\sigma_0$ is used exactly once in the encoding word, and it crosses the
two wires labeled $0$ and $n-1$.
\end{lemma}

\begin{proof}
Two external faces run along the whole wiring diagram, one black and one
white.  By our convention the black one is incident to $\sigma_0$
(\cref{fig:perfect-schematic-digons}).  By construction its only vertices
come from occurrences of $\sigma_0$.  In a perfect
arrangement it is a digon (\cref{lem:external-digons}), hence has a single
affine vertex.  Therefore $\sigma_0$ occurs exactly once, crossing the two
outermost wires, labeled $0$ and $n-1$.
\end{proof}

\begin{lemma}[Pruning from trailing digons]
\label{lem:trailing-digon-pruning}
The trailing external black digons of a perfect arrangement are bounded
by the wire pairs
\[
  (0,1),\ (2,3),\ \ldots,\ (n-3,n-2),
\]
and none of them is crossed in the interior of the wiring diagram.
\end{lemma}

\begin{proof}
Immediate from \cref{lem:external-digons}, the parity-color convention,
and reducedness.  The pairs are read off the wire reversal at the end of
the wiring diagram; reducedness forbids any other letter from crossing
them.
\end{proof}

\subsection{Even-only branching and composite generators}
\label{subsec:composite-generators}

After the initial odd prefix, no black generator can be the next
branching choice: every $\sigma_{2k-1}$ already appears in the prefix and
commutes with the others, so a second one would give two consecutive
occurrences of $\sigma_{2k-1}$, violating reducedness.  The next
branching is therefore on a white generator $\sigma_g$, with $g$ even.

An even generator $\sigma_g$ adds a vertex to each of the two adjacent black faces.
In a perfect arrangement these faces must become triangles
(\cref{lem:perfect-color}), so the closing
generators $\sigma_{g-1}$ and $\sigma_{g+1}$ have to occur before any later
generator adds another vertex to either face (see \cref{fig:composite-K}).
After these occurrences of $\sigma_{g-1}$ and $\sigma_{g+1}$,
every odd $\sigma_{2k-1}$ would again act on a pair of wires
that have already crossed, so branching requires an even generator.
We call this procedure \emph{even-only branching}.

We define the \emph{composite generators} as the following sequences
\begin{equation}
\label{eq:composite-K}
  K_g = \sigma_g\sigma_{g-1}\sigma_{g+1}\ \ (g=2,4,\ldots,n-3),
  \qquad
  K_0 = \sigma_0\sigma_1,
\end{equation}
and treat them as a single branching unit in the search.  We keep the bare
term \emph{generator}, or \emph{elementary generator} when contrast is
needed, for the individual $\sigma_i$.
In terms of composite generators, the even-only branching
introduced above builds a reduced word of the following
\emph{composite form}:
\begin{equation}
\label{eq:factorization}
  W=\sigma_1\sigma_3\cdots\sigma_{n-2}\,K_{g_1}K_{g_2}\cdots K_{g_{N_{\mathrm{perf}}}},
\end{equation}
where $K_0$ appears only once by \cref{lem:sigma-zero}.

Any reduced word for $w_0$ contains $\binom{n}{2}$ elementary generators:
$(n-1)/2$ in the prefix, $2$ in $K_0$, and $3$ in each of the
remaining $N_{\mathrm{perf}}-1$ composite generators. This gives
$N_{\mathrm{perf}} = (n^2 - 2n + 3)/6$ composite generators in \eqref{eq:factorization}.

The \emph{distant commutation} relation for composite generators follows
from~\eqref{eq:commute}:
\begin{equation}
\label{eq:distant-commute}
  K_gK_h = K_hK_g \qquad (|g-h|\ge 4).
\end{equation}

\begin{lemma}[Sufficiency of the composite form]
\label{lem:composite-sufficiency}
Every reduced word for $w_0$ of the form \eqref{eq:factorization}
encodes a perfect arrangement.
\end{lemma}

\begin{proof}
The odd prefix $\sigma_1\sigma_3\cdots\sigma_{n-2}$ of \eqref{eq:factorization}
opens $(n-1)/2$ black wedges.  By induction, this count is restored by each
composite generator $K_g$: it closes $2$ new black triangles (or $1$ if $g=0$)
and leaves $(n-1)/2$ open black wedges again.

The total number of internal black triangles is
$2N_{\mathrm{perf}}-1=n(n-2)/3$.
Their $n(n-2)$ sides exhaust the $n(n-2)$ bounded segments,
each incident to exactly one black face;
hence every bounded segment is used and the arrangement is perfect.
\end{proof}

\begin{lemma}[Necessity of the composite form]
\label{lem:composite-necessity}
A reduced word $W$ of odd-index majority encoding a perfect arrangement can be brought
to the form \eqref{eq:factorization} by commutations~\eqref{eq:commute} of elementary generators,
none of which transpose two white generators, that is, the indices of the composite
generators follow the order of the white generators in $W$.
\end{lemma}

\begin{proof}
Consider an occurrence of a white generator $\sigma_g$.  The two black
faces adjacent to it must be triangles.  Therefore the
generators that close them are $\sigma_{g-1}$ and $\sigma_{g+1}$, with the
obvious omission of $\sigma_{-1}$ when $g=0$.

Before both closing generators occur, none of $\sigma_{g-2}$, $\sigma_g$,
$\sigma_{g+2}$ can be applied.  Applying any of them would add another vertex to one of the two open
black faces, producing a black face with at least four vertices, which
cannot occur in a perfect arrangement.  Any letter occurring between $\sigma_g$
and the closing $\sigma_{g-1},\sigma_{g+1}$ is therefore none of $\sigma_{g-2},\sigma_g,\sigma_{g+2}$,
hence commutes with both $\sigma_{g-1}$ and $\sigma_{g+1}$.  Moving $\sigma_{g-1}$
and $\sigma_{g+1}$ next to $\sigma_g$ yields the composite generator $K_g$.

Applying this to every white generator, we obtain a word that starts
with the odd prefix and continues as a
sequence of composite generators, and the sequence of their indices coincides
with the order of the original white generators.
\end{proof}

\begin{lemma}[No consecutive repeat of $K_g$]
\label{lem:composite-no-repeat}
A reduced word for a perfect arrangement does not contain two consecutive
occurrences of the composite generator $K_g$.
\end{lemma}

\begin{proof}
For $g=0$, \cref{lem:sigma-zero} gives a single $\sigma_0$, hence a single
$K_0$, so there is nothing to show.  For $g\ge 2$, two consecutive
occurrences of $K_g$ contain the subword
$\sigma_g\sigma_{g-1}\sigma_{g+1}\sigma_g$.  The two occurrences of
$\sigma_g$ open and close a white face whose remaining two
vertices come from $\sigma_{g-1}$ and $\sigma_{g+1}$, so this face is a
quadrilateral.

In a perfect arrangement, each of the four black faces adjacent to this
quadrilateral is a triangle.  The two pseudolines containing a pair of
opposite sides of the quadrilateral cross twice, at the apex of
each of the two opposite triangles, contradicting
reducedness (see \cref{fig:no-repeat}).
\end{proof}

\begin{lemma}[Canonical composite form]
\label{lem:composite-canonical-form}
Among the words obtainable from the composite form \eqref{eq:factorization}
by distant commutations~\eqref{eq:distant-commute}, exactly one satisfies
$g-2\le h$ for every consecutive pair $K_gK_h$.
\end{lemma}

\begin{proof}
\emph{Existence.}  Repeatedly replace forbidden consecutive pairs $K_gK_h$,
$g\ge h+4$, by $K_hK_g$.  By~\eqref{eq:distant-commute}, each replacement
uses only commutations of independent elementary generators.
It remains to check that the process terminates.  Call a pair of composite
generators $K_g,K_h$ of $W$, with $K_g$ preceding $K_h$, a \emph{distant
inversion} if $g\ge h+4$.  A replacement
$K_gK_h \to K_hK_g$ reverses the order of exactly one pair, the swapped one,
which is a distant inversion before and not after; the relative order of every
other pair is unchanged.  Each replacement therefore decreases the number of
distant inversions by one, and the process ends at a word with no forbidden pair.

\emph{Uniqueness.}  The rule is locally confluent: two replacements at disjoint
position pairs commute, and the only overlap of two forbidden pairs is a
triple $K_aK_bK_c$ with
$a\ge b+4$ and $b\ge c+4$, whence $a\ge c+8$, so whichever swap is applied
first, the result is completed to $K_cK_bK_a$ by two further swaps.
Together with termination, Newman's lemma~\cite{newman-1942} makes the rule confluent, so
the normal form is unique.
\end{proof}

\begin{lemma}[Composite form up to distant commutation]
\label{lem:composite-distant-commutation}
Two reduced words of the form \eqref{eq:factorization} in the same
commutation class of elementary generators are related by distant commutations~\eqref{eq:distant-commute}.
\end{lemma}

\begin{proof}
By \cref{lem:composite-sufficiency}, the encoded arrangement is perfect.  Words in
one commutation class are related by the
commutations~\eqref{eq:commute}, which preserve every generator's index and the
relative order of non-commuting generators; so both words factor into the same
composite generators (\cref{lem:composite-necessity}), and only their order
may differ.

If two words of the class differed in the relative order of two white
generators $\sigma_g$ and $\sigma_h$ with $|g-h|<4$, they would be
consecutive in some word of the class. This is impossible for $\sigma_g\sigma_g$
by reducedness, and $\sigma_g\sigma_{g\pm 2}$ would add two vertices to the black face
between them, giving it at least four sides in the projective closure, contradicting perfectness
(\cref{lem:perfect-color,lem:external-digons}).  Hence the relative order of two
white generators at index distance at most $2$ is the same throughout the class.

By the above, the two words list the same composite generators in orders that
differ only in pairs at index distance at least $4$, which commute
by~\eqref{eq:distant-commute}; the two words are therefore related by
distant commutations~\cite[Prop.~2.1]{diekert-metivier-1997}.
\end{proof}

\subsection{The algorithm}
\label{subsec:algorithm}

The search enumerates the canonical words of the form \eqref{eq:factorization}
directly, adding the reducedness check and the trailing-digon
pruning of \cref{lem:trailing-digon-pruning}.

Define the search as the following recursive procedure:

\begingroup\ttfamily
\begin{tabbing}
xxxx\=xxxx\=\kill
Function Search($W$), where $W = \sigma_1\sigma_3\cdots\sigma_{n-2}\cdot K_{g_1}\cdots K_{g_k}$:\\
\>if $k = N_{\mathrm{perf}}$, output $W$;\\
\>else for each $g$ admissible at $W$, call Search($W\cdot K_g$).
\end{tabbing}
\endgroup

\noindent
The search starts at level $k=0$, with the \emph{partial word} $W$ equal to the
prefix from
\cref{lem:initial-digons}; the depth $N_{\mathrm{perf}}=(n^2-2n+3)/6$ is fixed by
\cref{lem:composite-necessity}.  Let $a_0,\ldots,a_{n-1}$ be the permutation of
wire labels produced by $W$ from the identity permutation; that is, $a_i$ is
the wire occupying slot $i$ after the crossings of $W$.

An index $g=0,2,4,\ldots,n-3$ is \emph{admissible} at $W$ when all of the following hold:

\begin{enumerate}[label=(\roman*)]
\item the wire pairs that $K_g$ would cross have not yet met
  (reducedness), checked by a comparison among $a_{g-1},a_g,a_{g+1},a_{g+2}$
  for $g\ge 2$ and by the special condition $a_0=0,\ a_1=n-1$ for $g=0$
  (\cref{lem:sigma-zero});
\item $g$ does not violate the trailing-digon constraint
  (\cref{lem:trailing-digon-pruning}, \cref{rem:wall-sealing});
\item $g\ne g_k$ (\cref{lem:composite-no-repeat});
\item $g\notin\{g_k-4,g_k-6,\ldots\}$ (canonical-form rule of
  \cref{lem:composite-canonical-form}).
\end{enumerate}

The implementation maintains $a$ incrementally: applying $K_g$ updates it
by adjacent transpositions, reverted on backtrack.  A transposition
$\sigma_i$ is reduced only when $a_i < a_{i+1}$; if $a_i > a_{i+1}$, the
wires have already crossed, and the branch is pruned.  Conditions
(iii) and (iv) reference the previous $K_{g_k}$ and are skipped at $k=0$.

Without~(i) the output contains words that describe no wiring diagram;
without~(iv) each wiring diagram is emitted many times, wasting work on the
redundant branches.
Conditions~(ii) and~(iii) are optional for correctness and only prune
empty branches early (those on which the search reaches no leaf).
The implementation includes several further prunings (\cref{app:prunings}).
They also do not affect correctness, but narrow the admissibility
conditions.

\begin{remark}[Branching order]
\label{rem:branching-order}
The branching order among admissible candidates affects only how quickly
a first arrangement is found; for an exhaustive run it changes the order of the
output, not its content.  We use two orders:

\begin{itemize}
\item[$\hookrightarrow$] the \emph{ascending} order
  $g_k-2$, then $g_k+2,\,g_k+4,\,\ldots,\,n-5,\,\,n-3$;
\item[$\hookrightarrow$] the \emph{descending} order
  $g_k-2$, then $n-3,\,n-5,\,\ldots,\,g_k+4,\,\,g_k+2$.
\end{itemize}

The ascending order makes every projective class (\cref{def:p-equivalence})
appear early in the output stream, which helps monitoring progress and
incremental classification of an exhaustive run.
For first-hit search the faster order is not known in advance and is chosen
empirically per~$n$ (\cref{subsec:first-hit}).
\end{remark}

\begin{theorem}[Completeness of the perfect search]
\label{thm:perfect-completeness}
The perfect search enumerates at least one reduced word
encoding each perfect arrangement.
\end{theorem}

\begin{proof}
Let $W$ be a reduced word encoding a perfect arrangement.  Without loss of
generality $W$ has odd-index majority; otherwise apply the parity-inverting
reflection $\sigma_i\mapsto\sigma_{n-2-i}$ ($n$ odd).  By
\cref{lem:composite-necessity}, commute $W$ into the form \eqref{eq:factorization}.
By \cref{lem:composite-canonical-form}, commute the composite part further into the
representative of its commutation class satisfying the canonical-form rule
of the search.

At each partial word along this representative, reducedness holds because the original
word is reduced and we used only commutations.  The trailing-digon and no-repeat
constraints hold because the representative still encodes the same perfect
arrangement.  Therefore the search accepts each of its composite generators
when the corresponding partial word is reached, and the final leaf is enumerated.
\end{proof}

\Cref{thm:perfect-completeness} guarantees that every perfect arrangement
is reached at least once.  The following corollary refines this to an exact
enumeration of commutation classes.

\begin{corollary}[One emitted word per commutation class]
\label{cor:perfect-class-output}
The perfect search emits exactly one reduced word in each odd-majority commutation class
of reduced words encoding a perfect arrangement, and none in any other class.
\end{corollary}

\begin{proof}
The search enumerates exactly the canonical words of the form
\eqref{eq:factorization}, all encoding perfect arrangements
(\cref{lem:composite-sufficiency}); so no emitted word lies outside the stated
classes.  Each such class contains an emitted word: the construction in the
proof of \cref{thm:perfect-completeness} turns any member into a canonical word
by commutations alone, so that word lies in the same class and is enumerated.  It is
unique: the words of the composite form in a class are related by distant
commutations (\cref{lem:composite-distant-commutation}), and exactly one of
them is canonical (\cref{lem:composite-canonical-form}).
\end{proof}

\Cref{cor:perfect-class-output} marks the limit of what the perfect search does
on its own: its output lists every commutation class of perfect arrangements
exactly once.
The only redundancy left is geometric.  By
\cref{obs:representation-multiplicity}, after the parity fix a single perfect
arrangement still corresponds to up to $2n$ wiring diagrams, one commutation
class each, so it is emitted up to $2n$ times.  These repeats are removed by
grouping the output into the equivalence classes of \cref{sec:classification}.

\section{Symmetries and classification of arrangements}
\label{sec:classification}

For a straight-line arrangement, the motions of the Euclidean plane act on it by
rotations and reflections (translations act trivially).  Projective
transformations relate it to other arrangements via projective closure
followed by a change of affine chart.  In this section, we describe how the
combinatorial type of a straight-line arrangement changes under the above
transformations and extend the resulting combinatorial operations to
pseudoline arrangements.

The enumeration output is a stream of reduced words.  By
\cref{obs:representation-multiplicity}, the same arrangement is enumerated
several times as several wiring diagrams.
We describe how to identify these duplicates.

This approach is also needed in
\cref{sec:2-defective-search}.  There the 2-defective search is complete in a
weaker sense: it does not enumerate every wiring diagram, but it reaches
every projective class, producing at least one representative of each.

\subsection{Encoding wiring diagrams as \texorpdfstring{$O$}{O}-matrices}
\label{subsec:order-matrices}

A reduced word reflects the order in which the crossings appear
during a particular sweep of a wiring diagram.
An alternative encoding of the same wiring diagram lists, for
each line, the order in which it crosses the others.  As we will see,
this encoding drops the inessential information about the order of
commuting generators in a reduced word.

\begin{definition}[$O$-matrix]
\label{def:order-matrix}
We refer to a tuple \(O=(O_0,\ldots,O_{n-1})\) whose row \(O_i\) is a
permutation of \(\{0,\ldots,n-1\}\setminus\{i\}\) as a label matrix.  Such a
label matrix is called an \emph{order matrix}, or \emph{$O$-matrix} for short, if
there is a simple pseudoline arrangement with a labeling of its lines at infinity
such that, for every \(i\), line \(i\) meets the others in the order
\(O_i\).\footnote{Each row \(O_i\) is the \emph{local sequence} of line \(i\)
\cite{goodman-pollack-1984}, equivalently the rank-3 case of
\emph{hyperline sequences} \cite{bokowski-roudneff-strempel-1997}; \(\{O_0,\ldots,O_{n-1}\}\)
encodes the rank-3 oriented matroid associated with a simple pseudoline
arrangement.  We write this as a matrix and develop
the row-level operations on this representation.  Rote, in his NumPSLA
program \cite{rote-2025-numpsla} for enumerating
all pseudoline arrangements, uses the same object to identify them via the
antipodal spherical model.}  Let \(\mathcal{O}_n\)
denote the set of all $O$-matrices on \(n\) labels.\footnote{Not every label
matrix is an $O$-matrix: some do not record the crossings of any
pseudoline arrangement.  The $O$-matrices among label matrices are characterized by a
consistency condition on index triples~\cite[Theorem~5.2.10]{felsner-goodman-2017}.
All $O$-matrices appearing below are guaranteed to lie in $\mathcal{O}_n$ by
construction: they arise either from
a reduced word by the procedure described next, or by applying to an existing
$O$-matrix an operation of \cref{subsec:order-matrix-operations}.  Each
operation re-reads the same arrangement (for \(\pi_p\), its projective closure
in a new affine chart) under a new labeling, so the resulting matrix is again
realized by that arrangement.  The procedure of reading the matrix off the relabeled
arrangement is explained in \cref{app:symmetries:row-formulas}.}
\end{definition}

Such a labeling is determined by three choices: which line is \(0\), an
end of it, and a direction around infinity (clockwise or
counterclockwise).  The remaining labels follow the ends of the lines
at infinity in that direction.  There are \(4n\) such labelings per arrangement.
Wiring diagrams and reduced words for \(w_0\) already have an intrinsic
labeling, induced by the sweep start and direction.

Given a reduced word \(W = \sigma_{g_1}\cdots\sigma_{g_m}\) for \(w_0\),
define \(O(W)\) by the following procedure.  Maintain the running wire
permutation \(a\), initially the identity.  For each letter \(\sigma_g\),
append \(a_{g+1}\) to row \(O_{a_g}\) and \(a_g\) to row \(O_{a_{g+1}}\),
then swap \(a_g\) and \(a_{g+1}\).  The wiring diagram of \(W\)
(\cref{subsec:wiring-diagrams}) is the arrangement for which \(O(W)\) records
the crossing orders, so \(O(W) \in \mathcal{O}_n\).

\begin{samepage}
For the word
\(
  \textcolor{red}{\sigma_1}\textcolor{orange}{\sigma_2}\textcolor{green!60!black}{\sigma_0}\textcolor{blue}{\sigma_1}\textcolor{violet}{\sigma_2}\textcolor{black}{\sigma_0}
\)
from \cref{fig:sweep-word-wiring}, this construction gives
\nopagebreak

\nopagebreak
\smallskip
\noindent\begin{minipage}[c]{0.45\linewidth}
\centering
\begin{tikzpicture}[x=0.5cm,y=0.5cm]
  \PseudolineWiringDiagram[fill,nolabels]{4}{1,2,0,1,2,0}{}
  \fill[red,opacity=0.7]            (0.75,1.5) circle (0.75mm);
  \fill[orange,opacity=0.7]         (1.75,0.5) circle (0.75mm);
  \fill[green!60!black,opacity=0.7] (2.75,2.5) circle (0.75mm);
  \fill[blue,opacity=0.7]           (3.75,1.5) circle (0.75mm);
  \fill[violet,opacity=0.7]         (4.75,0.5) circle (0.75mm);
  \fill[black,opacity=0.7]          (5.75,2.5) circle (0.75mm);
\end{tikzpicture}
\end{minipage}%
\begin{minipage}[c]{0.45\linewidth}
\centering
\(\renewcommand{\arraystretch}{1.15}
\begin{array}{c|ccc}
  & 1\text{st} & 2\text{nd} & 3\text{rd} \\
\hline
  0 & \textcolor{green!60!black}{2}   & \textcolor{blue}{3}   & \textcolor{violet}{1} \\
  1 & \textcolor{red}{2}    & \textcolor{orange}{3} & \textcolor{violet}{0} \\
  2 & \textcolor{red}{1}    & \textcolor{green!60!black}{0}   & \textcolor{black}{3}  \\
  3 & \textcolor{orange}{1} & \textcolor{blue}{0}   & \textcolor{black}{2}
\end{array}\)
\end{minipage}
\end{samepage}

\begin{lemma}[Bijection between commutation classes and \(\mathcal{O}_n\)]
\label{lem:o-matrix-commutation}
Words in the same commutation class of reduced words for \(w_0\) have the
same \(O\)-matrix, and the induced map from commutation classes to
\(\mathcal{O}_n\) is a bijection.\footnote{This equivalence is essentially
\mbox{(v)\,$\Leftrightarrow$\,(vi)} of \cite[Theorem~1.7]{goodman-pollack-1984},
stated there for periodic allowable sequences; we give a direct proof in the
present setting of reduced words with a fixed initial labeling.}
\end{lemma}

\begin{proof}
\emph{Well-defined.}  If \(|i-j|>1\), the generators
\(\sigma_i\sigma_j\) act on disjoint pairs of slots in \(a\), so
swapping them does not change which wire is appended to which row.

\emph{Surjective.}  By definition of \(\mathcal{O}_n\), any
\(O\in\mathcal{O}_n\) arises from a labeled arrangement.  Encoding it as
a reduced word \(W\) with the same labeling (\cref{subsec:wiring-diagrams})
gives \(O(W) = O\).

\emph{Injective.}  The matrix \(O\) gives the crossings (the
\(\binom{n}{2}\) unordered pairs of lines) a partial order
\(\prec\): row \(O_i\) lists the crossings on line \(i\) in order, and
\(\prec\) is the transitive closure of these per-line orders.  Crossings
sharing a line are thus comparable, so incomparable crossings have no
common line.  For two consecutive letters
\(\sigma_i\sigma_j\), a common line means \(|i-j|=1\) (no commutation)
and no common line means \(|i-j|>1\) (commutation~\eqref{eq:commute}).
A word with \(O\)-matrix \(O\) lists the crossings in a linear extension
of \(\prec\).

It therefore suffices to prove the following more general claim, by
induction on the number of crossings: any two words that list the same
crossings as linear extensions of a common partial order, in which
incomparable crossings share no line, are commutation-equivalent.
Since \(\prec\) has this property, injectivity follows.

Let \(W,W'\) be two such words and let \(c\) be the first crossing of
\(W\): \(W=c\,W_1\).  As nothing precedes it, \(c\) is \(\prec\)-minimal,
so every crossing before \(c\) in \(W'\) is incomparable to \(c\) (it cannot be
\(\prec\)-above \(c\), since \(W'\) is a linear extension) and hence
shares no line with \(c\).  Commuting \(c\) past each of them brings it to
the front: \(W'\sim c\,W_1'\).  The words
\(W_1,W_1'\) list the remaining crossings as linear extensions of the
restricted order, which inherits the property. Induction gives
\(W_1\sim W_1'\), and hence \(W\sim W'\).
\end{proof}

\begin{remark}[Bijection between \texorpdfstring{$O$}{O}-matrices and wiring diagrams]
\label{rem:o-matrix-wiring}
We use \(O\)-matrices as the formal combinatorial representation of wiring
diagrams. The \(O\)-matrix is recovered from a wiring diagram by definition.
In the other direction, an \(O\)-matrix encodes a rank-3 oriented
matroid, hence a pseudoline arrangement, from which a wiring diagram is built
by a sweep.  Neither this formalism nor a rigorous notion of wiring diagram is
needed for our purposes, so we treat wiring diagrams as an illustrative picture
and phrase formal statements in terms of \(O\)-matrices.
\end{remark}

The lemma gives a direct link between the commutation
classes of the words emitted by the search algorithms and their \(O\)-matrices,
letting us check that each class is encoded by a single canonical representative
and, where needed, deduplicate the output.  The Euclidean and projective
identifications below are therefore developed entirely on \(\mathcal{O}_n\),
without further reference to words.

\subsection{Operations on \texorpdfstring{$O$}{O}-matrices}
\label{subsec:order-matrix-operations}

The action of a Euclidean motion or projective transformation on an
arrangement can be viewed in two ways.  In the first, the sweep start
is fixed and the transformation moves the
arrangement; this view is geometric and, for a straight-line
arrangement, is realized by an actual motion of the plane.  In the
second, the arrangement is fixed and a change of labeling at infinity
yields the new \(O\)-matrix; this view is purely combinatorial and
applies to pseudoline arrangements as well.  Both give the same
row-level rules on \(\mathcal{O}_n\).  We formulate the rules
and record the realizing motion after each definition.
\Cref{app:symmetries:row-formulas} explains how each rule is read off the
relabeled arrangement.

\begin{definition}[Reflection]
\label{def:reflection}
The \emph{reflection} \(\mu\colon\mathcal{O}_n\to\mathcal{O}_n\) sends row
\(i\) to the old row \(n-1-i\) and replaces every entry \(x\) by
\(n-1-x\):
\[
  (\mu O)_{i,j} = n-1-(O_{n-1-i})_j .
\]
\end{definition}

The element \(\mu\) is an involution (\(\mu^2=\mathrm{id}\)).
Geometrically, for a straight-line arrangement, \(\mu\) is realized by a
Euclidean reflection of \(\mathbb{R}^2\) that reverses the order of the lines
at infinity (\(i \leftrightarrow n-1-i\)).

\begin{definition}[Euclidean rotation]
\label{def:euclidean-rotation}
The \emph{Euclidean rotation} \(\rho\colon\mathcal{O}_n\to\mathcal{O}_n\)
removes the top row, reverses it, appends it as the new bottom row, and
relabels lines so that the new top row carries label \(0\); see
\eqref{eq:rho-formula} in the appendix for the explicit row formula.
\end{definition}

Applying \(\rho\) to the matrix above yields:
\[
\begin{array}{c|ccc}
0&2&3&1\\
1&2&3&0\\
2&1&0&3\\
3&1&0&2
\end{array}
\xrightarrow[\scriptstyle\text{send to bottom}]{\scriptstyle\text{reverse row }0}
\begin{array}{c|ccc}
1&2&3&0\\
2&1&0&3\\
3&1&0&2\\
0&\mathbf{1}&\mathbf{3}&\mathbf{2}
\end{array}
\xrightarrow{\scriptstyle\text{relabel}}
\begin{array}{c|ccc}
0&1&2&3\\
1&0&3&2\\
2&0&3&1\\
3&0&2&1
\end{array}
\]

The operation \(\rho\) satisfies \(\rho^{2n}=\mathrm{id}\), corresponding to
the \(2n\) positions for the label \(0\) around infinity.  Geometrically,
for a straight-line arrangement, \(\rho\) is realized by a rotation of
\(\mathbb{R}^2\) past the next critical sweep direction, namely the one in
which the sweep is parallel to line \(0\).  At this critical direction,
line \(0\) becomes the new line \(n-1\), and the order of intersections
along it reverses.  Equivalently, \(\rho\) shifts the labeling at
infinity by one step.

Reflections and Euclidean rotations make the dihedral group \(D_{2n}\) of
order \(4n\) act on \(\mathcal{O}_n\):
\[
  D_{2n}=\langle\,\rho,\,\mu\,\mid\,\rho^{2n}=\mu^2=(\mu\rho)^2=e\,\rangle,
  \qquad |D_{2n}|=4n .
\]
The defining relations are verified for the operations \(\rho,\mu\) from the
row-level formulas in \cref{app:symmetries:group-relations}, and the action is
set up in \cref{def:dihedral-action}; it need not be faithful
(\cref{rem:action-not-faithful}).

\begin{definition}[Projective rotation]
\label{def:projective-rotation}
For each label \(p\in\{0,\ldots,n-1\}\), the \emph{projective rotation}
\(\pi_p\colon\mathcal{O}_n\to\mathcal{O}_n\) sends line \(p\) to infinity:
add the current line at infinity as a new line, reverse every row with
label less than \(p\), cyclically shift each row so that \(p\) becomes its
last entry, and relabel the remaining lines by their order along \(p\);
see \eqref{eq:pi-formula} in the appendix for the explicit row formula.
\end{definition}

Adding the current line at infinity as a new line materializes the
projective closure of \cref{subsec:checkerboard}, so the \(n+1\) lines
are now treated symmetrically and any of them can be chosen as the line
at infinity.

For example, with \(p=1\),
\[
\begin{aligned}
&
\begin{array}{c|ccc}
0&2&3&1\\
1&2&3&0\\
2&1&0&3\\
3&1&0&2
\end{array}
\xrightarrow{\scriptstyle\text{add }\infty}
\begin{array}{c|cccc}
0&2&3&1&\infty\\
1&2&3&0&\infty\\
2&1&0&3&\infty\\
3&1&0&2&\infty\\
\infty&0&1&2&3
\end{array}
\xrightarrow[{\begin{array}{c}\scriptstyle\text{send them}\\[-2pt]\scriptstyle\text{to bottom}\end{array}}]{\begin{array}{c}\scriptstyle\text{reverse rows}\\[-2pt]\scriptstyle\text{with label}<p\end{array}}
\begin{array}{c|cccc}
1&2&3&0&\infty\\
2&1&0&3&\infty\\
3&1&0&2&\infty\\
\infty&0&1&2&3\\
0&\boldsymbol{\infty}&\mathbf{1}&\mathbf{3}&\mathbf{2}
\end{array}
\\[6pt]
&\xrightarrow[\scriptstyle\text{put }p\text{ last}]{\scriptstyle\text{cyclic shift}}
\begin{array}{c|cccc}
1&2&3&0&\infty\\
2&0&3&\infty&\mathbf{1}\\
3&0&2&\infty&\mathbf{1}\\
\infty&2&3&0&\mathbf{1}\\
0&3&2&\infty&\mathbf{1}
\end{array}
\xrightarrow[{\begin{array}{c@{\;\;}c}\scriptstyle 2\to 0&\scriptstyle 3\to 1\\[-2pt]\scriptstyle 0\to 2&\scriptstyle \infty\to 3\end{array}}]{\begin{array}{c}\scriptstyle\text{relabel}\\[-2pt]\scriptstyle\text{by row }p\end{array}}
\begin{array}{c|ccc}
0&2&1&3\\
1&2&0&3\\
3&0&1&2\\
2&1&0&3
\end{array}
\xrightarrow{\scriptstyle\text{sort}}
\begin{array}{c|ccc}
0&2&1&3\\
1&2&0&3\\
2&1&0&3\\
3&0&1&2
\end{array}
\end{aligned}
\]
Geometrically, for a straight-line arrangement \(\pi_p\) is realized by
a rotation of the sphere that brings the great circle of line
\(p\) onto the equator (the new line at infinity).  The four substeps
trace this rotation: each line below \(p\) passes parallel to the sweep
(its row reverses, as in \cref{def:euclidean-rotation}), the cyclic
shift reads each row from its new crossing with the line at infinity,
and the \(\infty\) row enters with the closure and leaves once \(p\)
becomes the new infinity.  Equivalently, \(\pi_p\) reselects which line plays
the role of infinity: line \(p\) takes that role, and the remaining
labels are reassigned along it.

\subsection{Equivalence classes}
\label{subsec:equivalence-classes}

The equivalence relations on \(\mathcal{O}_n\) defined below are generated by
the operations of \cref{subsec:order-matrix-operations}.

\begin{definition}[Euclidean equivalence]
\label{def:e-equivalence}
The \emph{Euclidean equivalence relation} \(\sim_E\) on \(\mathcal{O}_n\) is
the orbit equivalence of the dihedral group \(D_{2n}\) acting on
\(\mathcal{O}_n\) (\cref{def:dihedral-action}):
\(O\sim_E O'\iff\exists\,g\in D_{2n}\colon O'=gO\).  An equivalence
class is called an \emph{\(E\)-class}.
\end{definition}

The orbit \(D_{2n}\cdot O\) is finite, so each \(E\)-class admits a
deterministic minimum under a fixed lexicographic order on \(\mathcal{O}_n\).
This minimum is the \emph{canonical representative} of the \(E\)-class,
computed in \cref{app:canonicalization:forms}.
Its short label, the \(E\)-class identifier (\emph{\(E\)-id}), is defined
in \cref{def:class-identifiers}.

An \(E\)-class forgets the labeling of lines at infinity within one
affine chart (\cref{def:order-matrix}).  It corresponds to an abstract
arrangement in that chart,
formalizing the intuitive notion of two arrangements being the same.

A broader equivalence relation identifies different
\(E\)-classes whose individual representatives are mapped to one another
by projective transformations:

\begin{definition}[Projective equivalence]
\label{def:p-equivalence}
The \emph{projective equivalence relation}\footnote{Throughout this paper,
``projective equivalence'' means combinatorial equivalence after projective
closure and change of affine chart.  It is not a claim of equivalence under
a geometric realization in \(\mathrm{PGL}(3)\).} \(\sim_P\) on
\(\mathcal{O}_n\) is the smallest equivalence relation containing
\(\sim_E\) and identifying \(O\) with \(\pi_p O\) for every
\(p\in\{0,\ldots,n-1\}\).  An equivalence class is called a
\emph{\(P\)-class}.
\end{definition}

In the spherical model, \(\sim_P\) corresponds to rotations of the
sphere: those preserving the equator (the dihedral part \(D_{2n}\)) and
those bringing another great circle onto the equator (the \(\pi_p\)).
Each \(P\)-class is a finite union
of \(E\)-classes; its \emph{canonical representative} is its
lexicographic minimum (\cref{app:canonicalization:forms}).  Its short label,
the \(P\)-class identifier (\emph{\(P\)-id}), is defined in
\cref{def:class-identifiers}.

\subsection{Symmetry groups and stabilizers}
\label{subsec:symmetries-stabilizers}

The operations \(\rho\), \(\mu\) and \(\pi_p\)
(\cref{def:euclidean-rotation,def:reflection,def:projective-rotation})
compose into finite \emph{operation sequences} and thus form a monoid
\(\mathcal{W}\). For \(g\in\mathcal{W}\) and \(O\in\mathcal{O}_n\) write
\(gO\) for the result of applying the sequence to \(O\).

Each operation also permutes the \(n+1\) projective lines
(ordinary labels \(0,\ldots,n-1\) and \(\infty\equiv n\)).  The permutation for \(\rho\) and \(\mu\) is determined
by the operation alone, but for \(\pi_p\) it additionally depends on row
\(p\) of the matrix acted on (the relabel step of \cref{def:projective-rotation}
reads it off).  Applying \(g\in\mathcal{W}\) to \(O\) therefore permutes
the lines by some \emph{induced line permutation} \(\tau(g;O)\in S_{n+1}\),
recording where each line ends up.  By construction, these permutations
satisfy the groupoid law
\begin{equation}
\label{eq:tau-composition}
  \tau(hg;\,O)=\tau(h;\,gO)\circ\tau(g;\,O).
\end{equation}

A sequence \(g\in\mathcal{W}\) with \(gM=M\) is a symmetry of the arrangement:
it returns the \(O\)-matrix to itself but relabels the lines by \(\tau(g;M)\).
For instance, rotations and reflections of a symmetric arrangement mix its lines,
but keep \(M\) fixed. Thus the relabeling depends on the sequence itself, not on
its action on the \(O\)-matrix. The next lemma gathers the relabelings of all
symmetries into a subgroup of \(S_{n+1}\).

\begin{lemma}[Stabilizer image is a subgroup]
\label{lem:gp-is-subgroup}
For every \(M\in\mathcal{O}_n\), the set
\begin{equation}
\label{eq:stabilizer-image}
  G(M):=\{\,\tau(g;M)\ :\ g\in\mathcal{W},\ gM=M\,\}
\end{equation}
of induced line permutations is a subgroup of \(S_{n+1}\).
\end{lemma}

\begin{proof}
The stabilizer of \(M\) in \(\mathcal{W}\) contains the empty sequence and
is closed under concatenation.  For \(g,h\) in it, the groupoid law
\eqref{eq:tau-composition} collapses, since \(gM=M\), to
\[
  \tau(hg;\,M)=\tau(h;M)\circ\tau(g;M),
\]
so \(G(M)\) is a nonempty subset of the finite group \(S_{n+1}\), closed
under composition, hence a subgroup.
\end{proof}

\begin{definition}[Symmetry group]
\label{def:p-symmetry-group}
The \emph{symmetry group} \(G_P\) of a \(P\)-class \(P\) is the subgroup
\(G(M)\) \eqref{eq:stabilizer-image}, taken for any representative
\(M\in P\).
Different representatives give conjugate subgroups of \(S_{n+1}\)
(\cref{lem:gp-conjugacy}), so as an abstract group (or as an action
on the \(n+1\) projective lines of \(P\)) \(G_P\) is independent of the choice.
\end{definition}

The group \(G_P\) acts on the \(n+1\) affine charts by permuting their lines
at infinity.  Fix a representative \(M\in P\).  These charts correspond to the
\(O\)-matrices \(\pi_i M\), one for each line \(i\) chosen as the line at
infinity (with \(\pi_\infty=\mathrm{id}\)).  Under \(\sim_E\) they split into
\(E\)-classes contained in \(P\).  Write \(m(E)\) for the number of charts
falling in an \(E\)-class \(E\) (a different \(M\in P\) relabels the same
\(n+1\) charts, so this count depends only on \(P\)); then
\(\sum_{E\subseteq P}m(E)=n+1\).
Each \(E\)-class also has internal symmetries within a fixed chart: the
elements of \(D_{2n}\) fixing its \(O\)-matrix.

\begin{definition}[Multiplicity and Euclidean stabilizer]
\label{def:multiplicity}%
\label{def:euclidean-stabilizer}
For an \(E\)-class \(E\) contained in a \(P\)-class \(P\), the number
\(m(E)\) is its \emph{multiplicity}, and the \emph{Euclidean stabilizer}
\(H_E\) is the stabilizer in \(D_{2n}\) of a representative \(M\in E\).
\end{definition}

\begin{lemma}[Orbit-stabilizer for charts]
\label{lem:orbit-stabilizer-charts}
For every \(E\)-class \(E\) contained in a \(P\)-class \(P\),
\[
  m(E)\cdot|H_E|=|G_P|.
\]
\end{lemma}

\begin{proof}
Since \(|G_P|\) does not depend on the chosen representative
(\cref{lem:gp-conjugacy}), compute it from a representative \(M\) of
\(E\), and write \(M_s=\pi_s M\) (\(\pi_\infty=\mathrm{id}\)) for the
\(O\)-matrix in chart \(s\).  By \cref{lem:chart-reduction}, every
\(g\in\mathcal{W}\) acts on \(M\) as \(gM=h M_s\) for some
\(h\in D_{2n}\), where \(s=\tau(g;M)^{-1}(n)\) is the line \(g\) sends to
infinity.  Hence some stabilizer element sends line \(s\) to infinity if and
only if \(h M_s=M\) for some \(h\in D_{2n}\), that
is, if and only if \(M_s\sim_E M\).  (When it does, \(g=h\pi_s\)
is such an element.)  The charts so reached are therefore the
\(m(E)\) charts of \(E\) (\cref{def:multiplicity}).

Fix one such \(s\).  The \(h\in D_{2n}\) with \(h M_s=M\) form a coset of
\(\operatorname{Stab}_{D_{2n}}(M_s)\); as \(M_s\) is a representative of
\(E\), this stabilizer has order \(|H_E|\) (\cref{def:euclidean-stabilizer}).
Distinct \(h\) in the coset induce distinct line permutations
\(\tau(h;M_s)\).  If two agreed, they would differ by a stabilizer element
inducing the trivial permutation.  But the only nonidentity element of
\(D_{2n}\) with trivial induced permutation is \(\rho^{n}\)
(\cref{lem:tau-kernel}), which lies outside any stabilizer
(\cref{cor:rho-n-no-fixed}).  Hence each
chart \(s\) contributes exactly \(|H_E|\) distinct elements
\(\tau(g;M)=\tau(h;M_s)\circ\tau(\pi_s;M)\) of \(G_P\), and elements from
different \(s\) differ since \(s=\tau(g;M)^{-1}(n)\).  Therefore
\(m(E)\cdot|H_E|=|G_P|\).
\end{proof}

It remains to count how many times a single \(E\)-class \(E\) is
enumerated.  This means counting how many commutation classes of reduced
words correspond to \(E\) and obey the coloring convention (i.e., the
chosen majority parity of generators).  By \cref{lem:o-matrix-commutation} these are the
\(O\)-matrices of \(E\) that survive the odd-majority filter; since the
perfect search emits one word per commutation class
(\cref{cor:perfect-class-output}), their count is exactly how many times
\(E\) is duplicated in its output.

\begin{lemma}[Parity-fixed representatives]
\label{lem:parity-fixed-representatives}
Let \(E\) be an \(E\)-class with a parity majority: in each of its
reduced words, odd- and even-indexed generators occur in unequal numbers
(in particular for perfect and 2-defective arrangements,
\cref{subsec:checkerboard,subsec:black-face-bound}).
Then the number \(\operatorname{rep}(E)\) of its \(O\)-matrices
that survive the odd-majority filter is
\[
  \operatorname{rep}(E)=\frac{2n}{|H_E|}.
\]
\end{lemma}

\begin{proof}
The \(D_{2n}\)-orbit of a representative of \(E\) has \(4n/|H_E|\)
$O$-matrices (the orbit--stabilizer theorem with stabilizer \(H_E\)).
The number of generators of each index parity is invariant under commutation,
hence is a
function of the $O$-matrix (\cref{lem:o-matrix-commutation}). By hypothesis
these two numbers never coincide, so each $O$-matrix of \(E\) has a strict majority
parity (one checkerboard color strictly dominates, \cref{subsec:checkerboard}).
The reflection \(\mu\) interchanges the
parities (\(n\) odd) and so flips the majority, pairing each odd-majority
$O$-matrix with an even-majority one; the filter keeps one per pair:
\[
  \operatorname{rep}(E) = \frac{1}{2}\cdot\frac{4n}{|H_E|} = \frac{2n}{|H_E|}.\qedhere
\]
\end{proof}

This specifies the ``up to \(4n\)'' of \cref{obs:representation-multiplicity}.
The \(4n\) is the full orbit size \(|D_{2n}|\), reached when \(H_E\) is trivial.
The parity fix halves it to \(2n\), and the arrangement's own symmetries
divide this by a further \(|H_E|\).

The algorithm reconstructing \(G_P\) and the \(H_E\) from a canonical
representative is described in \cref{app:canonicalization:groups}.  The
per-\(n\) profiles it produces are tabulated in
\cref{tab:sym-combined}; the isomorphism types of \(G_P\) observed
there fall within \(C_k\), \(D_k\), \(A_4\), \(S_4\), \(A_5\), and the
trivial group \(\{e\}\).\footnote{For symmetric straight-line arrangements, \(G_P\) is
conjugate to a finite subgroup of \(\mathrm{PGL}(3,\mathbb{R})\); since
any finite subgroup of a connected Lie group lies in a maximal compact
subgroup, this gives \(G_P\subseteq\mathrm{SO}(3)\). The finite
subgroups of \(\mathrm{SO}(3)\) are exactly the cyclic groups \(C_n\),
the dihedral groups \(D_n\), and the rotation groups of the regular
polyhedra \(A_4\), \(S_4\), \(A_5\) (tetrahedron, cube, icosahedron).
For pseudoline
arrangements the bound is a priori weaker, yet no symmetry type outside
this list occurs in our enumeration (\cref{tab:sym-combined}).}

\section{The 2-defective search}
\label{sec:2-defective-search}

For $n\equiv 1\pmod 6$, no perfect arrangement exists.  By
\cref{lem:optimal-defects-2q1p}, an optimal arrangement contains either two
quadrilateral defects or one pentagonal defect, with all other
internal black faces triangles.

This section extends the perfect search of \cref{sec:perfect-search} to
the 2-defective case, keeping a small per-step state,
composite generators, and the trailing-digon pruning of
\cref{lem:trailing-digon-pruning}.
The search enumerates the 2-defective arrangements by modified
composite generators, so the search space stays comparable to the
defect-free one.

\subsection{How missing black generators yield defects}
\label{subsec:special-idea}

Removing a black elementary generator from a word encoding a perfect
triangular pattern merges the two black faces it would have separated
(\cref{fig:merge-patterns}).

\begin{figure}[ht]
\centering
\begin{subfigure}[b]{0.3\linewidth}
\centering
\begin{tikzpicture}[x=0.5cm,y=0.5cm]
  \PseudolineWiringDiagram[fill,nolabels,noright]{7}{{5,1},4,3,2,1}{%
    0/3/F, 1/3/F, 2/3/F, 2/4/T, 3/3/L,
    gen/1/1, gen/1/5,
    spec/1/3%
  }
\end{tikzpicture}
\caption{}
\label{fig:merge-external}
\end{subfigure}\hfill
\begin{subfigure}[b]{0.3\linewidth}
\centering
\begin{tikzpicture}[x=0.5cm,y=0.5cm]
  \PseudolineWiringDiagram[fill,nolabels,noright]{7}{{1,3,5},{2},{1},{4},{3,5}}{%
    1/3/R, 2/3/F, 3/3/F, 4/3/F,
    2/2/B, 5/3/L, 4/4/T,
    gen/2/2, gen/3/1, gen/4/4, gen/5/3, gen/5/5,
    spec/3/3%
  }
\end{tikzpicture}
\caption{}
\label{fig:merge-internal}
\end{subfigure}\hfill
\begin{subfigure}[b]{0.3\linewidth}
\centering
\begin{tikzpicture}[x=0.5cm,y=0.5cm]
  \PseudolineWiringDiagram[fill,nolabels,noright]{7}{{1,5},{2},{1},{4},{3,5}}{%
    0/3/F, 1/3/F, 2/3/F, 3/3/F, 4/3/F,
    2/2/B, 5/3/L, 4/4/T,
    gen/2/2, gen/3/1, gen/4/4, gen/5/3, gen/5/5,
    spec/1/3, spec/3/3%
  }
\end{tikzpicture}
\caption{}
\label{fig:merge-pentagon}
\end{subfigure}
\caption{Examples of local merge patterns producing defects.
(a)~An external digon merges with an adjacent bounded triangle into a
single external face, a quadrilateral defect in the projective closure
($\sigma_3$ skipped).
(b)~Two adjacent bounded triangles merge into an internal quadrilateral
defect ($K_2^{-}K_4$).
(c)~Three black faces in a row merge into a pentagonal defect when both
crossings separating them are missing ($\sigma_3$ skipped, $K_2^{-}K_4$).}
\label{fig:merge-patterns}
\end{figure}
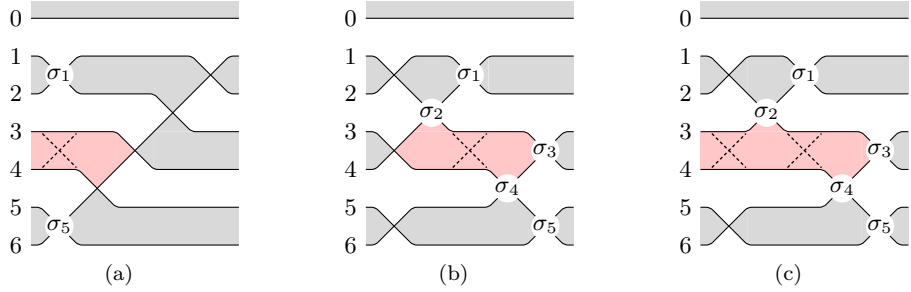

A 2-defective arrangement (\cref{def:2-defective}) is encoded by a word for a perfect
triangular pattern with exactly two black elementary generators
omitted.  We constrain the search to fix one omission in advance and
enumerate only the other.

For the first (fixed) omission we drop $\sigma_X$ from the forced
odd prefix of \cref{lem:initial-digons}:
\[
  \sigma_1\sigma_3\cdots\widehat{\sigma_X}\cdots\sigma_{n-2},
\]
where $X\in\{1,3,\ldots,n-2\}$ is a search parameter.  The missing crossing
turns the corresponding initial digon into a defect on the
initial external boundary between wires $X$ and $X+1$.  We call this
defect the \emph{anchor}.  Pinning it partially fixes the $E$-class
and $P$-class representative
(\cref{def:e-equivalence,def:p-equivalence}), removing part of the
redundancy of how different wiring diagrams represent the same arrangement
(see \cref{obs:representation-multiplicity}).

The second omission is left free: the search enumerates its position,
covering the 2-defective arrangements compatible with the chosen $X$.
Internal defects cannot be anchored to the external boundary,
so arrangements with no external defects are not emitted directly.
By \cref{cor:chart-exhaustion}, all of them can be reconstructed via the
$O$-matrix operations of \cref{subsec:order-matrix-operations} from the
other $P$-class representatives, emitted by the search.

\subsection{Special composite generators and correctness}
\label{subsec:special-blocks}

The second omission is realized at the level of the composite-generator
alphabet \eqref{eq:composite-K}: exactly one composite
generator $K_g=\sigma_g\sigma_{g-1}\sigma_{g+1}$ in the composite
form \eqref{eq:factorization} is replaced by a shorter composite generator containing the white
$\sigma_g$ and only one of its two closing black generators
$\sigma_{g-1}$ or $\sigma_{g+1}$.

\begin{definition}[Special composite generators]
\label{def:special-generators}
For an even index $g\in\{2,4,\ldots,n-3\}$, the two \emph{special
composite generators} at index $g$ are
\begin{equation}
\label{eq:special-generators}
  K_g^{+} = \sigma_g\sigma_{g+1},
  \qquad
  K_g^{-} = \sigma_g\sigma_{g-1}.
\end{equation}
The superscript records the surviving black generator.  The composite
generators \eqref{eq:composite-K} are called \emph{ordinary} when contrast
is needed.  The \emph{special-mode alphabet} of the 2-defective search is
\begin{equation}
\label{eq:special-alphabet}
  \{K_g : g\in\{0,2,\ldots,n-3\}\}
  \;\cup\;
  \{K_g^{+},K_g^{-} : g\in\{2,4,\ldots,n-3\}\}.
\end{equation}
\end{definition}

\begin{lemma}[Composite-generator factorization for quadrilateral defects]
\label{lem:special-factorization}
Fix $n\equiv 1\pmod 6$ and $X\in\{1,3,\ldots,n-2\}$.  Let $\mathcal{A}$ be a
simple arrangement of $n$ pseudolines with two quadrilateral defects,
in a wiring-diagram representation $D$ with exactly one of the defects (the
anchor) on the initial external boundary between wires $X$ and
$X+1$.  Then $D$ is encoded by a reduced word
\[
  W \;=\; \sigma_1\sigma_3\cdots\widehat{\sigma_X}\cdots\sigma_{n-2}\,
  C_1 C_2 \cdots C_{N},
\]
with each letter $C_i$ taken from the special-mode alphabet
\eqref{eq:special-alphabet}, and each quadrilateral defect $F$
corresponds in $W$ to
one of the following cases:
\begin{enumerate}
\item the prefix skip $\widehat{\sigma_X}$ --- anchor defect between
  wires $X$ and $X+1$;
\item two occurrences of $K_0$ --- external defect between wires
  $0$ and $n-1$;
\item two non-consecutive letters $K_h^{\pm}\,\ldots\,K_h$ at the
  same even index $h > 0$, with no $K_{h\mp 2}$ between them ---
  $K^{\pm}$-trapezoid;
\item two letters $K_h^{-}\,\ldots\,K_{h+2}$,
  $K_{h+2}^{+}\,\ldots\,K_h$ ($h > 0$), or $K_2^{+}\,\ldots\,K_0$,
  with no letter of index $h$ or $h+2$ between them --- rhombus;
\item a single special $K_h^{\pm}$, $h > 0$, with no later $C_i$
  equal to $K_h$ or $K_{h\mp 2}$ closing the defect --- trailing
  external defect.
\end{enumerate}
\end{lemma}

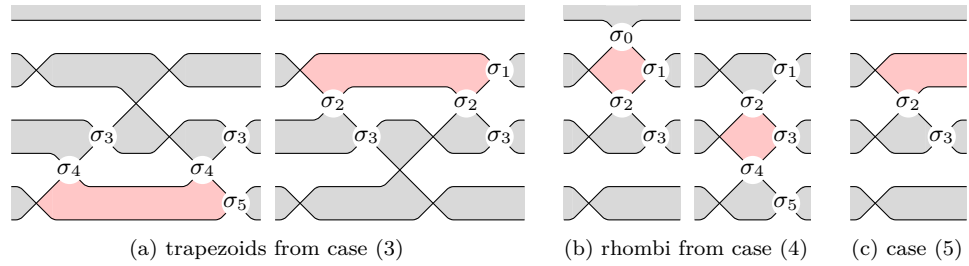
\begin{figure}[ht]
\centering
\subcaptionbox{trapezoids from case~(3)\label{fig:defect-class-trapezoid}}{%
\begin{tikzpicture}[x=0.44cm,y=0.44cm]
  \PseudolineWiringDiagram[fill,nolabels,noleft,noright]{7}{{5,1},4,3,2,{1,3},4,{3,5}}{%
    1/5/R, 2/4/B, 2/5/F, 3/4/F, 3/5/F, 4/4/F, 4/5/F, 5/4/F, 5/5/F, 6/4/B, 6/5/F, 7/5/L,
    gen/2/4, gen/3/3, gen/6/4, gen/7/3, gen/7/5%
  }
\end{tikzpicture}\hspace{0.5em}%
\begin{tikzpicture}[x=0.44cm,y=0.44cm]
  \PseudolineWiringDiagram[fill,nolabels,noleft,noright]{7}{{1,5},2,3,4,{3,5},2,{1,3}}{%
    1/1/R, 2/1/F, 2/2/T, 3/1/F, 3/2/F, 4/1/F, 4/2/F, 5/1/F, 5/2/F, 6/1/F, 6/2/T, 7/1/L, 7/2/F,
    gen/2/2, gen/3/3, gen/6/2, gen/7/1, gen/7/3%
  }
\end{tikzpicture}%
}\hfill
\subcaptionbox{rhombi from case~(4)\label{fig:defect-class-rhombus}}{%
\begin{tikzpicture}[x=0.44cm,y=0.44cm]
  \PseudolineWiringDiagram[fill,nolabels,noleft,noright]{7}{{1,3,5},{2,0},{1,3}}{%
    1/1/R, 2/0/B, 2/1/F, 2/2/T, 3/0/F, 3/1/L, 3/2/F,
    gen/2/0, gen/2/2, gen/3/1, gen/3/3%
  }
\end{tikzpicture}\hspace{0.5em}%
\begin{tikzpicture}[x=0.44cm,y=0.44cm]
  \PseudolineWiringDiagram[fill,nolabels,noleft,noright]{7}{{1,3,5},{2,4},{1,3,5}}{%
    1/3/R, 2/2/B, 2/3/F, 2/4/T, 3/2/F, 3/3/L, 3/4/F,
    gen/2/2, gen/2/4, gen/3/1, gen/3/3, gen/3/5%
  }
\end{tikzpicture}%
}\hfill
\subcaptionbox{case~(5)\label{fig:defect-class-trailing}}{%
\begin{tikzpicture}[x=0.44cm,y=0.44cm]
  \PseudolineWiringDiagram[fill,nolabels,noleft,noright]{7}{{1,3,5},{2},{3}}{%
    1/1/R, 2/1/F, 2/2/T, 3/1/F, end/1/F,
    gen/2/2, gen/3/3%
  }
\end{tikzpicture}%
}
\caption{Local quadrilateral defect patterns:
(a)~$K^{-}$-trapezoid $K_4^{-}\,K_2\,K_4$ and $K^{+}$-trapezoid $K_2^{+}\,K_4\,K_2$;
(b)~rhombi $K_2^{+}\,K_0$ and $K_2^{-}\,K_4=K_4^{+}\,K_2$;
(c)~trailing external defect $K_2^{+}$.}
\label{fig:quadrilateral-defect-classification}
\end{figure}

\begin{proof}
Let $W$ be any reduced word encoding the wiring diagram $D$.
By the argument of \cref{lem:initial-digons}, commutations of individual
$\sigma_i$~\eqref{eq:commute} bring it into the form
\[
  W \;=\; \sigma_1\sigma_3\cdots\widehat{\sigma_X}\cdots\sigma_{n-2}\;
  R,
\]
where the prefix generators close the initial external digons other
than the missing digon at $(X, X+1)$, and $R$ encodes the rest of the
diagram. We
construct the factorization $C_1\cdots C_N$ by scanning $R$ and
applying the local argument of \cref{lem:composite-necessity} to each
occurrence of a white $\sigma_h$.

\emph{Composition rules.} An application of the white generator
$\sigma_h$ gives one vertex to each of two adjacent black faces.
On any side where the adjacent black face is a triangle, the
corresponding closing $\sigma_{h\pm 1}$ commutes into a position
consecutive with $\sigma_h$, since no $\sigma_{h\pm 2}$ or $\sigma_h$
can appear between them in $W$ without adding a vertex to the open
triangle.

(i) If both adjacent faces of $\sigma_h$ are triangles, both
$\sigma_{h-1}$ and $\sigma_{h+1}$ commute in and $\sigma_h$ is packaged
into an ordinary $K_h$ (or, at $h=0$, into $K_0$ with only $\sigma_1$
involved).

(ii) If the adjacent face on one side of $\sigma_h$ is a defect $F$
and $\sigma_h$ is the last letter with a vertex on $F$ before the closing
$\sigma_{h\pm 1}$ in $W$, the same argument applies:
$\sigma_{h\pm 1}$ still commutes in.  When the other side
contributes its black generator as well (by rule (i) or by this
rule), $\sigma_h$ is packaged into an ordinary $K_h$ ($K_4$ in
\cref{fig:defect-class-trapezoid}, left).

(iii) Otherwise, either $F$ has no closing $\sigma_{h\pm 1}$, or another
letter with a vertex on $F$ follows $\sigma_h$ before that closing $\sigma_{h\pm 1}$.
Either way, the defect-facing side of $\sigma_h$ is left without
a black $\sigma_{h\pm 1}$, and $\sigma_h$ is packaged into the special $K_h^{\mp}$
using only the other side ($K_4^{-}$ in
\cref{fig:defect-class-trapezoid}, left).

Rule~(iii) cannot hold on both sides of a single $\sigma_h$.  If it
did, both faces adjacent to $\sigma_h$ would be defects.  Since the
arrangement has exactly two defects, one of them would be the anchor
at $(X, X{+}1)$.  But $\sigma_h$ is the only white letter with a
vertex on the anchor, hence trivially the last before its closing
$\sigma_{h\pm 1}$, so rule~(ii) applies on the anchor side.

\emph{Defect classification.}
Consider the places in wiring diagram $D$ where a defect can appear.
An external defect can be placed on the initial external boundary (case~(1)),
the trailing external boundary (case~(5)), or between wires $0$ and $n-1$
(case~(2)).
Otherwise $F$ is internal, covered by cases~(3) and~(4).

\emph{Case (1).}
The anchor occupies the initial boundary at the wire pair $(X, X+1)$.
Its two internal vertices come from a white $\sigma_h$,
$h=X\pm 1$, and the black $\sigma_X$ closing the defect.
Rule~(ii) applies on the anchor side of $\sigma_h$, so $\sigma_h$ is
packaged into either $K_{X\pm 1}$ or $K_{X\pm 1}^{\mp}$,
depending on the composition rule applying on the other side.

\emph{Case~(2).}
$F$ has both of its internal vertices from applications of $\sigma_0$,
on distinct wire pairs.  The adjacent black face of $\sigma_0$ on
the side opposite $F$ is either a triangle or the anchor defect
(no other defects exist), so composition rule (i) or (ii) applies.
Thus each $\sigma_0$ is packaged into a $K_0$, yielding two occurrences of $K_0$
in $W$ (\cref{fig:special-zero-example}).

\emph{Cases~(3),~(4).}
$F$ has four internal vertices, coming from four letters of $W$.
By the parity coloring the opening
$\sigma_g$ and closing $\sigma_g$ at the slot pair $(g, g+1)$ are odd,
and the other two are white $\sigma_h, \sigma_{h'}$ at adjacent
slot pairs, so $h, h'\in\{g-1, g+1\}$. One of them, say $\sigma_h$,
obeys composition rule (ii) and is packaged into $K_h$, while the other
obeys composition rule (iii) and is packaged into $K_{h'}^{\pm}$.

For $h' = h$ (trapezoid), placing $K_h^{\pm}$ and $K_h$ as
consecutive letters would force a wire pair to cross twice.
The two letters are therefore separated.
No $K_{h\mp 2}$ may occur between them.
Such a letter would already close the defect as a rhombus (the next
case), leaving the later $K_h$ outside the defect.  For $h' = h\pm 2$ (rhombus),
the earlier of the two white generators takes rule~(iii), the later
rule~(ii).  Both orders occur among the words encoding $D$:
$\sigma_h$ and $\sigma_{h'}$ commute~\eqref{eq:commute}, and a chain of
noncommuting letters leading from one to the other would have to pass
through an intermediate occurrence of $\sigma_g$ (either $h=g-1$ and
$h'=g+1$ or vice versa), which would close $F$ before its
fourth vertex.  No letter of index $h$ or $h'$ occurs between the two
composite letters: its white generator would give the quadrilateral
$F$ a fifth vertex.  The two variants $K_h^{-}\,\ldots\,K_{h+2}$ and
$K_{h+2}^{+}\,\ldots\,K_h$ thus encode the same rhombus defect; with
the two letters consecutive, they are both equal to
\(\sigma_h\,\sigma_{h+2}\,\sigma_{h-1}\,\sigma_{h+1}\,\sigma_{h+3}\)
and form the two sides of the additional relation
\begin{equation}
\label{eq:special-commutation}
  K_h^{-}\, K_{h+2} \;=\; K_{h+2}^{+}\, K_h.
\end{equation}
A special variant at index $0$ is not defined (no
$K_0^{-}=\sigma_0$), so for the rhombus closed by $\sigma_1$ only the
order with $\sigma_0$ later applies, yielding $K_2^{+}\,\ldots\,K_0$
with no paired form.

\emph{Case~(5).}
$F$ has one internal black opening vertex from $\sigma_g$ for some odd
$g > 0$ and one internal white vertex from $\sigma_h$, $h\in\{g-1,
g+1\}$. The composition rule for $\sigma_h$ is (iii), so it is packaged into
$K_h^{\pm}$ with no further $K_{h}$ or $K_{h\mp 2}$.

Scanning $R$, the composition rules package every white $\sigma_h$
with its closing black generators into a single $K_h$ or $K_h^{\pm}$.
This exhausts $R$ and produces the factorization $C_1\cdots C_N$,
with every letter in the special-mode alphabet.
\end{proof}

Both $K_h^{+}$ and $K_h^{-}$ are required in the alphabet: each
encodes one of the two trapezoids of case~(3).  In case~(4) a
rhombus admits two reduced-word encodings, related by
\eqref{eq:special-commutation} and commutations of independent
elementary generators; this is the principal source of
duplicates in the search output.

\begin{figure}[ht]
\centering
\begin{tikzpicture}[x=0.5cm,y=0.5cm]
  \PseudolineWiringDiagram[fill,nolabels]{7}{{3,1},2,{1,3},{0,4},{1,3,5},2,{1,3},4,{3,5},2,{1,3},0,1}{top, 0/5/F, 1/5/F, 2/5/F, 3/5/F, 4/5/F, 4/4/B, 5/5/L, 4/0/T, 12/0/T, gen/4/0, gen/5/1, gen/12/0, gen/13/1}
\end{tikzpicture}
\caption{An arrangement for $n=7$ with two applications of $K_0$.}
\label{fig:special-zero-example}
\end{figure}
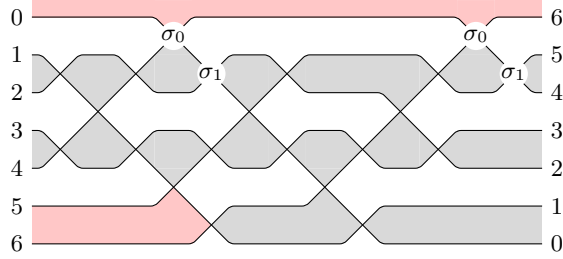

The composite generator $K_0=\sigma_0\sigma_1$
has only one black generator and admits no
shortened special variant.  When the external face between wires
$0$ and $n-1$ is a quadrilateral defect, the single-use rule of
\cref{lem:sigma-zero} relaxes instead:

\begin{corollary}[$K_0$ for a quadrilateral defect at $(0, n-1)$]
\label{cor:special-sigma-zero}
In case~(2) of \cref{lem:special-factorization}, the two $K_0$
encoding the quadrilateral defect at $(0, n-1)$ cross, in order, wire
$0$ with some wire $k\ne n-1$ and wire $k$ with wire $n-1$
(\cref{fig:special-zero-example}).
\end{corollary}

\begin{proof}
By \cref{lem:special-factorization}, case~(2), exactly two of the
$C_i$ equal $K_0$, and among all letters only $K_0=\sigma_0\sigma_1$
moves a wire occupying slot~$0$.  Hence wire~$0$ holds slot~$0$ until
the first $K_0$, which crosses it with some wire $k$ occupying slot~$1$.
Afterwards $k$ holds slot~$0$ until the second $K_0$.
This is the last occurrence of $K_0$, moving the wire $n-1$ to its final slot~$0$.
\end{proof}

\begin{lemma}[Coverage with initial-digon anchor]
\label{lem:special-coverage-internal}
Let $n\equiv 1\pmod 6$, $X\in\{1,3,\ldots,n-2\}$, and let $\mathcal{A}$ be a
2-defective arrangement of $n$ pseudolines that admits a
wiring-diagram representation in which exactly one defect lies on
the initial external boundary, between wires $X$ and $X+1$.  Then
this representation is encoded by a reduced word
\begin{equation}
\label{eq:special-word-form}
  \sigma_1\sigma_3\cdots\widehat{\sigma_X}\cdots\sigma_{n-2}\,
  C_1 C_2 \cdots C_{N_{\mathrm{sp}}},
\end{equation}
where each $C_i$ is taken from the special-mode alphabet
\eqref{eq:special-alphabet}, the sequence $C_1\cdots C_{N_{\mathrm{sp}}}$
falls into one of two
subtypes:
\begin{enumerate}
\item \textls[100]{special subtype}: exactly one $C_i$ equals
$K_g^{\pm}$ and exactly one $C_i$ equals $K_0$;
\item \textls[100]{double-zero subtype}: no $C_i$ equals
$K_g^{\pm}$ and exactly two $C_i$ equal $K_0$.
\end{enumerate}
The number of letters $C_i$ is
\begin{equation}
\label{eq:special-block-count}
  N_{\mathrm{sp}} = \frac{n^2 - 2n + 7}{6}.
\end{equation}
\end{lemma}

\begin{proof}
Apply \cref{lem:initial-digons} to each initial wire pair
$(2k-1, 2k)$ with $2k-1\ne X$: each is a digon contributing
$\sigma_{2k-1}$ to the prefix, up to commutations of independent
letters.  The pair $(X, X+1)$ is the anchor defect, not a digon, so
$\sigma_X$ is absent.

\emph{Pentagonal anchor.}  The pentagon at $(X, X+1)$ has three
internal vertices: two from white generators $\sigma_{h}$ and
$\sigma_{h'}$, $h, h'\in\{X-1, X+1\}$, and one from the black closing
$\sigma_X$ (not to be confused with the skipped one).
The local composition rules of \cref{lem:special-factorization}
apply verbatim (the pentagon is the sole defect, so neither $\sigma_h$
nor $\sigma_{h'}$ has a defect adjacent on both sides, and rule~(iii)
never applies on both sides).  Thus $\sigma_X$ is packaged with the
later $\sigma_{h'}$ into an ordinary $K_{h'}$ by rule~(ii), and the
earlier $\sigma_h$ into a special $K_{h}^{\pm}$ (either $K_{X-1}^{-}$
or $K_{X+1}^{+}$) by rule~(iii).  This is the special subtype
(there is no external defect between wires $0$ and $n-1$).
Every other white $\sigma_g$ has both adjacent faces triangles and is
packaged into an ordinary $K_g$ by rule~(i).

If $X=n-2$, there is no generator $\sigma_{n-1}$, so $h=h'=n-3$.
If $X=1$, $\sigma_0$ cannot precede the closing $\sigma_1$, the first
occurrence of $\sigma_1$ in \eqref{eq:special-word-form}: $\sigma_0$
crosses wires $0$ and $n-1$ (by \cref{lem:sigma-zero}, whose proof uses
only that this face is a digon), so wire $n-1$ must already have been
moved to slot $1$ by an occurrence of $\sigma_1$.  Therefore, $h=h'=2$.

\emph{Two quadrilateral defects.}  Otherwise $\mathcal{A}$ has a second
quadrilateral defect.  By \cref{lem:special-factorization}, that
second defect is one of cases~(2)--(5).  Case~(2) gives the
double-zero subtype (two $K_0$, no special); cases~(3)--(5) give
the special subtype (one special $K_g^{\pm}$ and one $K_0$ from
the $(0, n-1)$ digon by \cref{lem:sigma-zero}).

\emph{Length.}  Total length is $\binom{n}{2}$, the prefix
contributes $(n-3)/2$, and the $C_i$ contribute
$3 N_{\mathrm{sp}} - 2$ in both cases (two length-$2$ letters and
$N_{\mathrm{sp}} - 2$ length-$3$ ordinary composite generators).  Solving
gives \eqref{eq:special-block-count}.
\end{proof}

\Cref{lem:special-coverage-internal} covers wiring diagrams with
exactly one defect on the initial boundary, at the wire pair $(X, X+1)$.
Two cases lie out of scope.
Wiring diagrams with both defects on the initial boundary could in
principle be encoded by skipping two letters of the prefix, but this
would expand the search space.  The defect jump of
\cref{subsec:special-jump} below reaches such diagrams as images of
in-scope ones instead.
Wiring diagrams with no defect on the initial boundary are not
produced by the search directly.  The $P$-class of such an
arrangement contains a representative with a defect on the initial
boundary, which is enumerated (\cref{thm:2-defective-completeness}).

\begin{lemma}[Sufficiency of the composite form]
\label{lem:special-sufficiency}
Every reduced word for $w_0$ of the form \eqref{eq:special-word-form}, with
each $C_i$ in the special-mode alphabet and $C_1\cdots
C_{N_{\mathrm{sp}}}$ of the special or double-zero subtype of
\cref{lem:special-coverage-internal}, encodes a 2-defective arrangement
with one defect on the initial external boundary
between wires $X$ and $X+1$.
\end{lemma}

\begin{proof}
As in \cref{lem:composite-sufficiency}, each black generator of a composite
letter closes a distinct internal black face.  Two black generators per
letter would give $2N_{\mathrm{sp}}$; three are lost, as $K_0$ and the
special $K_h^{\pm}$ (or the second $K_0$) carry one black generator each,
and one black generator closes the anchor rather than an internal face.
Hence $b_{\mathrm{int}} \ge 2N_{\mathrm{sp}} - 3 = \left\lfloor n(n-2)/3\right\rfloor$,
giving with \eqref{eq:bint-bound} an equality, so $\mathcal{A}$ is 2-defective.

The anchor face between wires $X$ and $X+1$ is a defect rather than a
digon: the first $\sigma_X$ sits in a
letter $C_i$ behind its white generator $\sigma_{X\pm 1}$, which has
already placed a vertex on the face; together they give it at least two
affine vertices.
\end{proof}

\begin{lemma}[Consecutive same-index composite generators]
\label{lem:special-no-repeat}
In the encoding word \eqref{eq:special-word-form} of
\cref{lem:special-coverage-internal}, with $\sigma_X$ the skipped
initial generator and the only defect on the initial boundary between
wires $X$ and $X+1$, no two consecutive composite generators share
the same index $g$, except in one of the forms
\[
  K_{X+1}\, K_{X+1}, \quad
  K_{X+1}\, K_{X+1}^{+},
  \qquad
  K_{X-1}\, K_{X-1}, \quad
  K_{X-1}\, K_{X-1}^{-}.
\]
In each exception, $K_{X\pm 1}$ is the first $C_i$ involving
wires $X$ or $X+1$.
\end{lemma}

\begin{proof}
The pair $K_0 K_0$ expands to $\sigma_0\sigma_1\sigma_0\sigma_1$,
which is not reduced (a braid move~\eqref{eq:braid} and a cancellation
give $\sigma_1\sigma_0$).  A pair $K_g^{\pm} K_g^{\pm}$ or $K_g^{\pm} K_g$
contains the subword $\sigma_g\sigma_{g\pm 1}\sigma_g$ encoding a
white triangle, ruled out by \cref{lem:no-white-triangles}.

A pair of consecutive composite generators $K_g K_g^{\pm}$ or $K_g K_g$
produces a white quadrilateral $Q$ whose two adjacent black faces, sharing a
common vertex $v$, must both be non-triangles: otherwise the argument of
\cref{lem:composite-no-repeat} forces some wire pair to cross twice.
This can take place in a 2-defective arrangement only if both non-triangle
black faces are quadrilateral defects
(\cref{fig:special-no-repeat-example,fig:special-no-repeat-example-2}).
One of them is the anchor defect between wires $X$ and $X+1$; sharing a
side with $Q$ forces $g\in\{X-1,X+1\}$.  The two vertices the anchor
shares with $Q$ come from letters both packaged into the first $K_g$:
$\sigma_g$ opens $Q$ and $\sigma_X$ closes the anchor, and the common
vertex $v$ is from $\sigma_X$ or from $\sigma_g$.  If $v$ is from
$\sigma_X$, then $\sigma_X$ opens a second defect
(\cref{fig:special-no-repeat-example}) and the next composite generator
is the special $K_g^{\pm}$, lacking a further $\sigma_X$.  If $v$ is from
$\sigma_g$, the other black generator of $K_g$ closes the internal defect
(\cref{fig:special-no-repeat-example-2}), produced by a special generator,
hence only ordinary $K_g$ may repeat.

For the anchor to reach this $\sigma_X$, wires $X$ and $X+1$ must
occupy their initial slots when it is applied: no earlier $C_i$ may
have touched them.
\end{proof}

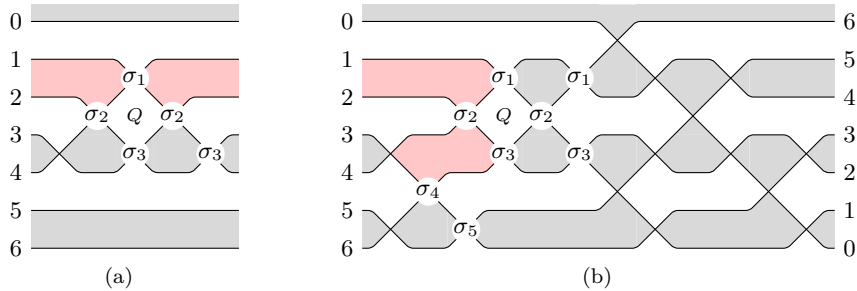
\begin{figure}[ht]
\centering
\begin{subfigure}[b]{0.3\linewidth}
\centering
\begin{tikzpicture}[x=0.5cm,y=0.5cm]
  \PseudolineWiringDiagram[fill,nolabels,noright]{7}{3,2,{1,3},2,3}{%
    0/1, 1/1, 2/1, 3/1, 4/1, 5/1, end/1/F, 2/2/T, 4/2/T,
    gen/2/2, gen/3/1, gen/3/3, gen/4/2, gen/5/3%
  }
  \node[font=\scriptsize] at (2.75,3.5) {$Q$};
\end{tikzpicture}
\caption{}
\label{fig:special-no-repeat-example}
\end{subfigure}%
\begin{subfigure}[b]{0.7\linewidth}
\centering
\begin{tikzpicture}[x=0.5cm,y=0.5cm]
  \PseudolineWiringDiagram[fill,nolabels]{7}{{3,5},4,{2,5},{1,3},2,{1,3},{0,4},{1,3,5},2,{1,3},4,{3,5}}{%
    0/1/F, 1/1/F, 2/1/F, 3/1/F, 3/2/T, 4/1/L,
    1/3/R, 2/3/F, 2/4/T, 3/3/F, 3/2/B, 4/3/L,
    gen/2/4, gen/3/5, gen/3/2, gen/4/1, gen/4/3, gen/5/2, gen/6/1, gen/6/3%
  }
  \node[font=\scriptsize] at (3.75,3.5) {$Q$};
\end{tikzpicture}
\caption{}
\label{fig:special-no-repeat-example-2}
\end{subfigure}
\caption{The exceptions allowed by \cref{lem:special-no-repeat}, for $X=1$:
(a)~$K_{X+1}\,K_{X+1}^{+}$ as a local pattern;
(b)~$K_{X+1}\,K_{X+1}$ in the word $K_4^{+}K_2K_2K_0K_4K_2K_4$ emitted
at $n=7$.}
\label{fig:special-no-repeat-patterns}
\end{figure}

\subsection{Defect jump}
\label{subsec:special-jump}

The forced initial-digon prefix and the trailing-digon pruning of
\cref{lem:initial-digons,lem:trailing-digon-pruning} were derived for
perfect arrangements, where they force the boundary wire pairs into
digons.  They therefore conflict with 2-defective arrangements with a defect
on the boundary.  Rather than drop the pruning, we keep it
and reach these arrangements by a local operation that moves a defect: the search
enumerates arrangements with the defect off the boundary and emits
the images in which the defect has moved onto it.

\begin{definition}[Defect jump]
\label{def:defect-jump}
Let $F$ be a defect of a 2-defective arrangement, and let
$L_1,L_2$ be two wires carrying non-adjacent sides of $F$.
Then $L_1$ and $L_2$ cross at a unique point $p$, lying
outside $F$.  The \emph{defect jump along the strip between
$L_1$ and $L_2$} uncrosses $L_1,L_2$ at $p$ and
crosses them inside $F$ instead.
\end{definition}

\begin{lemma}[Correctness of the defect jump]
\label{lem:defect-jump-correct}
The defect jump produces a simple arrangement $\mathcal{A}'$,
which is again 2-defective; moreover the two
faces of $\mathcal{A}$ at $p$ that lie along the strip between $L_1$
and $L_2$ are black.
\end{lemma}

\begin{proof}
\emph{Simplicity.}  The jump exchanges the parts of $L_1$ and $L_2$
inside the closed region $R$ bounded by the arc of $L_1$ between $p$
and $F$, the arc of $L_2$ between $p$ and $F$, and a path joining their
endpoints through the interior of $F$ (\cref{fig:defect-jump-schema}).
Besides those two arcs, the
boundary $\partial R$ consists only of the path through the interior
of the face $F$, which no wire crosses.  Every other wire $N$ is an
unbounded curve with both ends outside $R$, so it meets the closed
boundary $\partial R$ an even number of times.  As $N$ crosses each of
$L_1$ and $L_2$ exactly once in $\mathcal{A}$, it therefore meets the
two arcs either both once or both not at all.  Exchanging the arcs then
keeps each crossing of $N$ with $L_1$ and $L_2$ (only swapping which
arc carries it), leaves all crossings not involving $L_1,L_2$ in
place, and replaces the crossing at $p$ by a single crossing inside
$F$.  Every pair of wires still crosses exactly once, so
$\mathcal{A}'$ is a simple arrangement.

\emph{2-defectiveness.}  The jump changes $\mathcal{A}$ only inside
$R$, at its two ends.  Crossing $L_1,L_2$ inside the black defect $F$
splits it and raises the number of internal black faces by one.  At
$p$, the two faces along the strip merge into one; they are opposite
at $p$ and therefore share the same color.  If this merge removed
no internal black face, then
$b_{\mathrm{int}}(\mathcal{A}') = b_{\mathrm{int}}(\mathcal{A}) + 1 =
\lfloor n(n-2)/3\rfloor + 1$, contradicting~\eqref{eq:bint-bound} for
the simple arrangement $\mathcal{A}'$.  Hence the merge removes exactly
one internal black face; the two merging faces are then black, and
$b_{\mathrm{int}}(\mathcal{A}') = b_{\mathrm{int}}(\mathcal{A})$, so
$\mathcal{A}'$ is again 2-defective.
\end{proof}

\begin{figure}[ht]
\centering
\begin{subfigure}[b]{0.64\linewidth}
\centering
\begin{tikzpicture}[scale=0.8, every node/.style={font=\small},
    wire/.style={line width=0.4pt},
    boundary/.style={line width=1.2pt}]
\coordinate (p) at (4.6,0.25);
\def\vext{1.5}
\def\Lone{plot[hobby] coordinates {(-1.7,0.86) (0,0.9) (2.4,0.74) (4.6,0.25) (6.4,-0.40)}}
\def\Ltwo{plot[hobby] coordinates {(-1.7,-0.66) (0,-0.7) (2.4,-0.48) (4.6,0.25) (6.4,1.25)}}
\def\Nline{(3.35,\vext) .. controls (3.25,0.95) and (3.25,-0.35) .. (3.35,-\vext)}
\def\Npline{(5.75,\vext) .. controls (5.65,0.9) and (5.65,-0.05) .. (5.75,-\vext)}
\def\Fleftc{(-1.15,\vext) .. controls (-1.10,0.7) and (-1.10,-0.3) .. (-1.15,-\vext)}
\def\Frightc{(0.15,\vext) .. controls (0.10,0.7) and (0.10,-0.3) .. (0.15,-\vext)}
\path[name path=Lone] \Lone;
\path[name path=Ltwo] \Ltwo;
\path[name path=Fleft] \Fleftc;
\path[name path=Fright] \Frightc;
\path[name intersections={of=Lone and Fleft, by=f1}];
\path[name intersections={of=Lone and Fright, by=f2}];
\path[name intersections={of=Ltwo and Fright, by=f3}];
\path[name intersections={of=Ltwo and Fleft, by=f4}];
\fill[red!22] (f1) -- (f2) -- (f3) -- (f4) -- cycle;          
\begin{scope}[overlay]
  \clip \Nline -- (7,-1.5) -- (7,1.5) -- cycle;               
  \clip (-3,-3) rectangle (4.6,3);                            
  \clip \Lone -- (6.6,-3) -- (-3,-3) -- cycle;                
  \clip \Ltwo -- (6.6,3) -- (-3,3) -- cycle;                  
  \fill[black!15] (-3,-3) rectangle (7,3);
\end{scope}
\begin{scope}[overlay]
  \clip \Npline -- (2,-1.5) -- (2,1.5) -- cycle;              
  \clip (4.6,-3) rectangle (7,3);                             
  \clip \Lone -- (6.6,3) -- (-3,3) -- cycle;                  
  \clip \Ltwo -- (6.6,-3) -- (-3,-3) -- cycle;                
  \fill[black!15] (2,-3) rectangle (7,3);
\end{scope}
\draw[wire] \Lone node[right,inner sep=2pt] {$L_1$};
\draw[wire] \Ltwo node[right,inner sep=2pt] {$L_2$};
\draw[wire] \Fleftc node[above right,inner sep=2pt] {$N_1$};
\draw[wire] \Frightc node[above right,inner sep=2pt] {$N_2$};
\draw[wire] \Nline node[above right,inner sep=2pt] {$N_3$};
\draw[wire] \Npline node[above right,inner sep=2pt] {$N_4$};
\begin{scope}
  \clip \Frightc -- (4.6,-1.5) -- (4.6,1.5) -- cycle;
  \draw[boundary] \Lone;
  \draw[boundary] \Ltwo;
\end{scope}
\draw[boundary, dashed] (f2) .. controls (-0.35,0.25) and (-0.35,-0.25) .. (f3);
\fill (p) circle (1.4pt) node[below] {$p$};
\node[circle, fill=white, inner sep=0.5pt] at (-0.65,0.1) {$F$};
\node[boundary,black] at (1.6,1.2) {$\partial R$};
\node[left,inner sep=2pt] at (-1.7,0.86) {$L_1$};
\node[left,inner sep=2pt] at (-1.7,-0.66) {$L_2$};
\end{tikzpicture}
\caption{}
\label{fig:defect-jump-schema}
\end{subfigure}\hfill
\begin{subfigure}[b]{0.35\linewidth}
\centering
\begin{tikzpicture}[x=0.55cm,y=0.55cm]
  \PseudolineWiringDiagram[fill,nolabels,noleft,noright]{5}{1,2,3,2,1}{%
    1/1/R, 2/1/F, 2/2/T, 3/1/F, 4/1/F, 4/2/T, 5/1/L,
    wire/2, wire/3%
  }
  \node[pl/generator even] at (1.75,1.5) {$\sigma_g$};
  \node[pl/generator even] at (3.75,1.5) {$\sigma_g$};
  \node[pl/generator odd]  at (4.75,2.5) {$\sigma_h$};
  \node[font=\scriptsize, inner sep=1pt] (qlab) at (0.5,1.5) {$q$};
  \draw[thin] (qlab) -- (1.4,1.5);
  \node[circle, fill=white, inner sep=0pt] at (2.75,2.5) {$F$};
  \node[font=\scriptsize] at (2.75,3.35) {$L_2$};
  \node[font=\scriptsize] at (2.75,1.7) {$L_1$};
  \node[left, font=\footnotesize] at (0,3) {$h$};
  \node[left, font=\footnotesize] at (0,2) {$h{+}1=g$};
  \node[left, font=\footnotesize] at (0,1) {$g{+}1$};
\end{tikzpicture}
\caption{}
\label{fig:special-jump-pair-fragment}
\end{subfigure}
\caption{(a)~The region $R$ in the proof of \cref{lem:defect-jump-correct}.
(b)~A fragment for \cref{lem:special-jump-pair} with a $K_g^{+}$ whose
omitted $\sigma_h$ ($h=g-1$) would cross $L_1,L_2$ held by the slots
$h,h+1$: the wires carry non-adjacent sides of $F$.  Here the $\sigma_g$
of the later $K_g$ moves $L_1$ out of the pair, ending its side before
the $\sigma_h$ of that $K_g$ closes $F$.}
\label{fig:defect-jump-and-pair}
\end{figure}
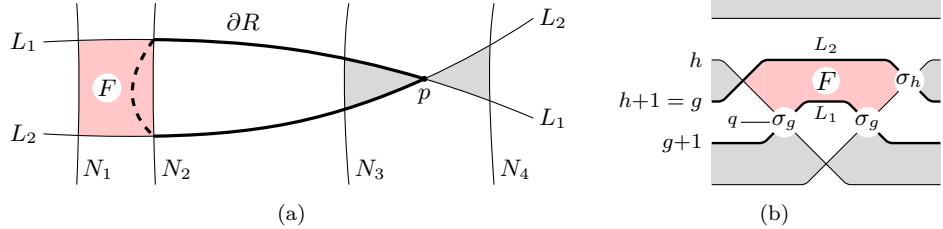

%
\newcommand\jface[2]{\begin{scope}[overlay]#2\fill[#1] (-4,-6) rectangle (12,6);\end{scope}}
\newcommand\jstripL{\clip \Wone -- (12,-6) -- (-4,-6) -- cycle;
                    \clip \Wtwo -- (12,6) -- (-4,6) -- cycle;}
\newcommand\jstripR{\clip \Wone -- (12,6) -- (-4,6) -- cycle;
                    \clip \Wtwo -- (12,-6) -- (-4,-6) -- cycle;}
\newcommand\jleftof[1]{\clip #1 -- (-4,-6) -- (-4,6) -- cycle;}
\newcommand\jrightof[1]{\clip #1 -- (12,-6) -- (12,6) -- cycle;}
%
\newcommand\jvext{1.4}
\newcommand\jtrans[3]{%
  ({#1+(#2)+(#3)},\jvext)
  .. controls ({#1+(#2)*0.5/\jvext-0.333*(#3)},0.5)
         and ({#1-(#2)*0.5/\jvext-0.333*(#3)},-0.5) ..
  ({#1-(#2)+(#3)},-\jvext)}
%
\newcommand\jBl{\jtrans{0.6}{0.85}{-0.09}}
\newcommand\jCl{\jtrans{0.6}{-0.85}{-0.09}}
\newcommand\jAl{\jtrans{3.4}{0}{0.0}}
\newcommand\jApl{\jtrans{5}{0}{0.03}}
\newcommand\jDl{\jtrans{7.9}{0}{0.06}}
\newcommand\jwires{\draw[wire] \Wone; \draw[wire] \Wtwo;
  \draw[wire] \jBl; \draw[wire] \jCl;
  \draw[wire] \jAl; \draw[wire] \jApl; \draw[wire] \jDl;}

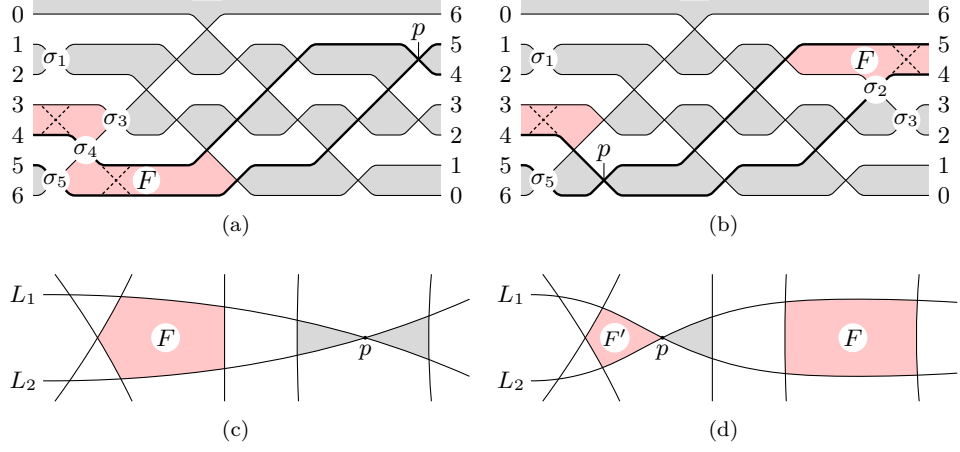
\begin{figure}[ht]
\centering
\begin{subfigure}[b]{0.49\linewidth}
\centering
\begin{tikzpicture}[x=0.40cm,y=0.40cm]
  \PseudolineWiringDiagram[fill,nolabels]{7}{{5,1},4,3,2,{1,3},{0,4},{1,3,5},2,{1,3},4,{3,5},2,{1,3}}{%
    0/3/F, 1/3/F, 2/3/F, 2/4/T, 3/3/L,
    1/5/R, 2/4/B, 2/5/F, 3/5/F, 4/5/F, 5/5/F, 6/4/B, 6/5/F, 7/5/L,
    wire/4, wire/5,
    gen/1/1, gen/1/5, gen/2/4, gen/3/3,
    spec/1/3, spec/3/5%
  }
  \node[circle, fill=white, inner sep=0.0 pt] at (3.75, 0.5) {$F$};
  \node[inner sep=0.4 pt] (plabel-a) at (12.75, 5.4) {$p$};
  \draw[thin] (plabel-a) -- ++(0, -0.85);
\end{tikzpicture}
\caption{}
\label{fig:defect-jump-internal}
\end{subfigure}\hfill
\begin{subfigure}[b]{0.49\linewidth}
\centering
\begin{tikzpicture}[x=0.40cm,y=0.40cm]
  \PseudolineWiringDiagram[fill,nolabels]{7}{{5,1},4,{3,5},2,{1,3},{0,4},{1,3,5},2,{1,3},4,{3,5},2,3}{%
    0/3/F, 1/3/F, 2/3/F, 2/4/T, 3/3/L,
    9/1/R, 10/1/F, 11/1/F, 12/1/F, 12/2/T, 13/1/F, end/1/F,
    wire/4, wire/5,
    gen/1/1, gen/1/5, gen/12/2, gen/13/3,
    spec/1/3, spec/13/1%
  }
  \node[circle, fill=white, inner sep=0pt] at (11.37, 4.47) {$F$};
  \node[inner sep=0.4 pt] (plabel-b) at (2.75, 1.4) {$p$};
  \draw[thin] (plabel-b) -- ++(0, -0.85);
\end{tikzpicture}
\caption{}
\label{fig:defect-jump-trailing}
\end{subfigure}\\[12pt]
\begin{subfigure}[b]{0.49\linewidth}
\centering
\begin{tikzpicture}[scale=0.6, every node/.style={font=\small},
    wire/.style={line width=0.4pt}]
\def\Wone{plot[hobby] coordinates {(-0.6,0.95) (1.4,0.89) (3.6,0.65) (6.5,0) (8.8,-0.85)}}
\def\Wtwo{plot[hobby] coordinates {(-0.6,-0.95) (1.4,-0.89) (3.6,-0.65) (6.5,0) (8.8,0.85)}}
\coordinate (p) at (6.5,0);
\jface{red!22}{\jstripL \jrightof\jBl \jrightof\jCl \jleftof\jAl}
\jface{black!15}{\jstripL \jrightof\jApl}
\jface{black!15}{\jstripR \jleftof\jDl}
\jwires
\fill (p) circle (1.2pt) node[below,inner sep=2.5pt] {$p$};
\node[circle, fill=white, inner sep=0.5pt] at (2.1,0) {$F$};
\node[left,inner sep=2pt] at (-0.6,0.95) {$L_1$};
\node[left,inner sep=2pt] at (-0.6,-0.95) {$L_2$};
\end{tikzpicture}\unskip
\caption{}
\label{fig:defect-jump-pentagon}
\end{subfigure}\hfill
\begin{subfigure}[b]{0.49\linewidth}
\centering
\begin{tikzpicture}[scale=0.6, every node/.style={font=\small},
    wire/.style={line width=0.4pt}]
\def\Wone{plot[hobby] coordinates {(-0.6,0.95) (1.0,0.65) (2.3,0) (4,-0.63) (6.4,-0.85) (8.8,-0.78)}}
\def\Wtwo{plot[hobby] coordinates {(-0.6,-0.95) (1.0,-0.65) (2.3,0) (4,0.63) (6.4,0.85) (8.8,0.78)}}
\coordinate (x) at (2.3,0);
\jface{red!22}{\jstripL \jrightof\jBl \jrightof\jCl}
\jface{black!15}{\jstripR \jleftof\jAl}
\jface{red!22}{\jstripR \jrightof\jApl \jleftof\jDl}
\jwires
\fill (x) circle (1.2pt) node[below,inner sep=2.5pt] {$p$};
\node[circle, fill=white, inner sep=0.5pt] at (6.5,0) {$F$};
\node[circle, fill=white, inner sep=0pt, font=\footnotesize] at (1.25,0) {$F'$};
\node[left,inner sep=2pt] at (-0.6,0.95) {$L_1$};
\node[left,inner sep=2pt] at (-0.6,-0.95) {$L_2$};
\end{tikzpicture}\unskip
\caption{}
\label{fig:defect-jump-pentagon-after}
\end{subfigure}
\captionsetup{margin=28pt}
\caption{\looseness=-1
Two 2-defective arrangements differing by a defect jump on the strip
between wires $4$ and $5$: (a)~an internal defect
($K_4^{-}$ at the start) becomes (b)~a trailing external one
($K_2^{+}$ at the end).  A pentagonal defect~(c) and the two
quadrilateral defects~(d) after a defect jump.}
\label{fig:defect-jump-examples}
\end{figure}

By \cref{lem:defect-jump-correct} the two faces of $\mathcal{A}$ at
$p$ along the strip are black, which leaves three cases.
If $F$ is a quadrilateral and $p$ is the shared vertex of two
triangles (in the projective closure), the defect moves: $F$ becomes
a pair of triangles, and a new quadrilateral defect appears at $p$
(\cref{fig:defect-jump-internal,fig:defect-jump-trailing}).
If $F$ is a pentagon (necessarily the sole defect, so $p$ is again
the shared vertex of two triangles), $F$ splits into a triangle
and a quadrilateral defect, and a second quadrilateral defect
appears at $p$
(\cref{fig:defect-jump-pentagon,fig:defect-jump-pentagon-after}).
Conversely, if $F$ is a quadrilateral and $p$ is a vertex of the other
quadrilateral defect $F'$, the two defects become a pentagon at $p$.

By construction, the defect jump is an involution: the jump of the new
defect at $p$, along the same strip between $L_1$ and $L_2$, restores
the original arrangement.

\begin{lemma}[Jump pair at a special composite generator]
\label{lem:special-jump-pair}
Let $W$ be a reduced word of the form \eqref{eq:special-word-form} of
the special subtype, let $F$ be the black face that the omitted generator
$\sigma_{g\mp 1}$ of $K_g^{\pm}$ would have closed, and let $L_1,L_2$ be the
wires that $\sigma_{g\mp 1}$ would cross.  Then $L_1$ and $L_2$ carry
non-adjacent sides of $F$ and $F$ is a defect.
\end{lemma}

\begin{proof}
Let $h=g\mp 1$ be the index of the omitted elementary generator and
$(h,h+1)$ the slot pair whose wires it would cross.  Right after the
$\sigma_g$ of $K_g^{\pm}$ these slots hold $L_1$ and $L_2$, the face $F$ lies
between them, and $\sigma_g$ gives $F$ a vertex $q$
(\cref{fig:special-jump-pair-fragment}).  $\sigma_g$ acts on the
slots $g$ and $g+1$, shifted by one relative to $(h,h+1)$, so it crosses one
wire of the pair, say $L_1$, with a third wire.  Hence the side of $F$
starting at $q$ lies on $L_1$, whereas the side carried by $L_2$
started earlier; the two can be adjacent only at their common end, the
closing vertex of $F$, and only if that vertex is the crossing of $L_1$ with
$L_2$, produced by an occurrence of $\sigma_h$.  A bare $\sigma_h$ is not a
letter of the alphabet; the prefix precedes $\sigma_g$ entirely; $W$ has no
second special composite generator; and in $K_{h\pm 1}$ the generator
$\sigma_h$ is preceded by the white $\sigma_{h\pm 1}$, acting on
the slots $(h-1,h)$ or $(h+1,h+2)$, which moves one of $L_1,L_2$ out of
the pair and thereby ends the side that wire carries before $\sigma_h$
is applied.  The crossing of $L_1$ with $L_2$ is therefore not the closing
vertex of $F$: the two sides are not adjacent, so $F$ is not a triangle in the
projective closure and is therefore a defect.
\end{proof}

\begin{lemma}[Coverage with anchored initial defect]
\label{lem:special-coverage-boundary}
Let $\mathcal{A}$ be a 2-defective arrangement of $n$ pseudolines that
admits a wiring-diagram representation $D$ with a defect on the
initial external boundary between wires $X$ and $X+1$.  Then either
$D$ is encoded by a word of the form \eqref{eq:special-word-form}
whose special composite generator $K_g^{\pm}$ (if any) does not
create a defect on the trailing external boundary, or $D$ is the
image of such a wiring diagram under a single defect jump
(\cref{def:defect-jump}).
\end{lemma}

\begin{proof}
If $\mathcal{A}$ has no defect on the trailing external boundary and
exactly one defect on the initial external boundary (the anchor),
the second defect is either external, between wires
$0$ and $n-1$, or internal,
and $D$ is covered by \cref{lem:special-coverage-internal} directly.

Otherwise the non-anchor defect $F_{\mathrm{ext}}$ of $\mathcal{A}$ lies
on an external boundary away from the anchor; write $D=D_{\mathrm{ext}}$.
Let $L_1,L_2$ be the two wires carrying non-adjacent
sides of $F_{\mathrm{ext}}$, crossing once at $p\notin F_{\mathrm{ext}}$
(\cref{def:defect-jump}).
By \cref{lem:defect-jump-correct} the two faces meeting at $p$ along
the strip between $L_1$ and $L_2$ are black.
Suppose one of these two faces is external.  Then it is an external digon
whose only vertex is $p$: a wire splitting its arc at infinity would
split the arc of $F_{\mathrm{ext}}$ as well, since the sweep reverses
the wire order.  This digon and $F_{\mathrm{ext}}$ are then opposite external
faces of the same color.
This contradicts the checkerboard coloring since $n$ is odd.

Both faces at $p$ are therefore internal black triangles, and the
jump of $F_{\mathrm{ext}}$ along the strip merges them into an internal
quadrilateral $F_{\mathrm{int}}$. The resulting diagram $D_{\mathrm{int}}$ with
the non-anchor defect $F_{\mathrm{int}}$ is covered by
\cref{lem:special-coverage-internal}.  Since the jump is an involution,
$D_{\mathrm{ext}}$ is the image of $D_{\mathrm{int}}$ under the
jump along the same strip.
\end{proof}

\subsection{The algorithm}
\label{subsec:2-defective-algorithm}

The 2-defective search extends the depth-first enumeration of composite
generators of \cref{subsec:algorithm}: it runs over the special-mode
alphabet \eqref{eq:special-alphabet} on the prefix with one skipped
generator $\sigma_X$.

Define the search as the following recursive procedure:

\begingroup\ttfamily
\begin{tabbing}
xxxx\=xxxx\=\kill
Function SpecialSearch($W$), where $W = \sigma_1\sigma_3\cdots\widehat{\sigma_X}\cdots\sigma_{n-2}\cdot C_1\cdots C_k$:\\
\>if $k = N_{\mathrm{sp}}$, output $W$;\\
\>else for each $C$ admissible at $W$, call SpecialSearch($W\cdot C$).
\end{tabbing}
\endgroup

\noindent
The search starts at level $k=0$ with $W$ equal to the prefix of
\cref{lem:special-coverage-internal}; the depth $N_{\mathrm{sp}}$ is
given by \eqref{eq:special-block-count}.  The letters $C_i$ are subject
to the structural constraint of
\cref{lem:special-coverage-internal}: either exactly one $K_g^{\pm}$ and
exactly one $K_0$ among them (special subtype), or no $K_g^{\pm}$ and exactly
two $K_0$ (doub\-le-zero subtype).

A pair consisting of an index $g\in\{0,2,\ldots,n-3\}$ and a
composite-generator variant (ordinary, $K_g^{+}$, or $K_g^{-}$) is
\emph{admissible} at $W$ when all of the following hold:

\begin{enumerate}[label=(\roman*)]
\item the wire pairs that the composite generator would cross have not yet met
  (reducedness): the comparison of
  \cref{subsec:algorithm}, with the pair on the omitted side ignored
  for $K_g^{\pm}$;
\item for $g\ge 2$, $g$ does not violate the trailing-digon constraint
  (\cref{lem:trailing-digon-pruning}, checked as in \cref{rem:wall-sealing});
\item $K_0$ usage matches \cref{lem:sigma-zero} ($a_0=0,\ a_1=n-1$) or
  \cref{cor:special-sigma-zero} ($a_0=0$ for the first $K_0$,
  $a_1=n-1$ for the second);
\item $g\ne g_k$, the index of the last letter $C_k$, with the exceptions
  of \cref{lem:special-no-repeat};
\item $g\notin\{g_k-4,g_k-6,\ldots\}$ (canonical-form rule of
  \cref{lem:composite-canonical-form}).
\end{enumerate}

The implementation extends the state of \cref{subsec:algorithm}: each
stack entry additionally records the variant of $K_g$ applied; the
single-special constraint of
\cref{lem:special-coverage-internal} is enforced by skipping the
$K_g^{\pm}$ branches once one has been used, or once one of the two
$K_0$ of \cref{cor:special-sigma-zero} has been applied.

As in the perfect search, the implementation includes the further
prunings of \cref{app:prunings} (see \cref{rem:2-defective-forced-descent}
for the restricted use of the forced-descent pruning).

When the search reaches depth $N_{\mathrm{sp}}$, the accepted leaf is
emitted as-is: by \cref{lem:special-sufficiency} it encodes a 2-defective
arrangement.
When the search applies a special $K_g^{\pm}$ at a position where the
ordinary $K_g$ is rejected by a reducedness conflict with a prefix
elementary generator or by the trailing-digon pruning
(\cref{lem:trailing-digon-pruning}) and reaches a leaf, the defect jump
(\cref{def:defect-jump}) is applied to the resulting defect along the
strip between the wires that the omitted generator of $K_g^{\pm}$ would
have crossed.  This jump is well defined by \cref{lem:special-jump-pair}.
The image is emitted as a second output,
recovering a boundary defect: on the initial external boundary in
case of the prefix reducedness conflict, and on the trailing external
boundary in case of the trailing-digon pruning.

The defect jump lets the trailing-digon pruning~(ii) stay in force for
$g\ge 2$, keeping the search tree small.
For $g=0$ it is not imposed: reaching the defect between wires $0$ and $1$ on
the trailing external boundary by a defect jump would require defining a
$K_0^{-}=\sigma_0$.  Extending~(ii) to $g=0$ would prune more
for $X=1$ (the only case in which an emitted word crosses the wire pair
$(0,1)$ in the interior) but add an extra branch at $g=0$ for all $X$.
We prefer not to do this.

\begin{proposition}[Coverage of the 2-defective search]
\label{prop:2-defective-coverage}
Fix $n\equiv 1\pmod 6$ and $X\in\{1,3,\ldots,n-2\}$.  Let $D$ be a
wiring-diagram representation of a 2-defective arrangement of $n$
pseudolines with one defect on the initial external boundary between
wires $X$ and $X+1$.  Then, up to commutations of independent
elementary generators, $D$ appears in the output of the 2-defective search at parameter
$X$, either as a directly emitted leaf or as the image of one under
a defect jump.
\end{proposition}

\begin{proof}
\emph{Two cases: direct leaf and jump image.}
By \cref{lem:special-coverage-boundary}, $D$ is either encoded by a word of
the form \eqref{eq:special-word-form} (with the special composite
generator, if any, not creating a defect on the trailing external
boundary) or equals $D_{\mathrm{ext}}$, the image of such a diagram
$D_{\mathrm{int}}$ under the jump of its internal defect
$F_{\mathrm{int}}$ along the strip between two wires $L_1,L_2$.
In the latter case $L_1,L_2$ carry two non-adjacent sides of the
non-anchor external defect $F_{\mathrm{ext}}$ of $D_{\mathrm{ext}}$, and
the inverse jump, of $F_{\mathrm{ext}}$,
merges two internal black triangles $t_1,t_2$ of $D_{\mathrm{ext}}$
(the two faces along the strip at $p=L_1\cap L_2$, where $t_1$ ends
and $t_2$ starts) into $F_{\mathrm{int}}$
(see the proof of \cref{lem:special-coverage-boundary} and
\cref{fig:defect-jump-internal,fig:defect-jump-trailing}).  It suffices to show that
$D_{\mathrm{int}}$ (in the direct case, $D$ itself) is enumerated
as a directly emitted leaf and, in the defect-jump case, that its leaf
triggers the second output with image $D_{\mathrm{ext}}$.

\emph{The word to be enumerated.}
By \cref{lem:special-coverage-internal}, $D_{\mathrm{int}}$ admits an
encoding word of the form \eqref{eq:special-word-form} with $C_1\cdots
C_{N_{\mathrm{sp}}}$ of the special or double-zero subtype.  In the
defect-jump case the jump moves a single crossing, so the other letters keep
the order they take in $D_{\mathrm{ext}}$, where $t_1$ ends at $p$ and $t_2$
starts there.  The white vertex of $t_1$ is thus the earlier of the two white
vertices of $F_{\mathrm{int}}$, and the letter
giving it is packaged into the special composite generator (cases~(3) and~(4) of
\cref{lem:special-factorization}).  The black generator it omits is then the
one that would create $p$, absent from $D_{\mathrm{int}}$.
This fails only for $\sigma_0$, where no $K_0^{-}$ exists; see the exception
at the end of the proof.

\emph{Its canonical representative.}
The structural properties of
\cref{lem:special-coverage-internal,lem:special-no-repeat} hold for
every word in the commutation class of $D_{\mathrm{int}}$.  Distant
commutations~\eqref{eq:distant-commute} bring the word above to a
representative $W$ with $C_1\cdots C_{N_{\mathrm{sp}}}$ satisfying the
canonical-form rule of \cref{lem:composite-canonical-form} (with $K_g^{\pm}$
treated as having index $g$: its elementary support is a subset of
that of $K_g$, so the commutation relations this rule relies on still hold).
They swap composite generators acting on disjoint slot ranges, so they do
not change the wire pair on the omitted side of the special composite
generator.

\emph{The leaf is enumerated.}
At each partial word along $W$, reducedness holds because $W$ is reduced;
the trailing-digon condition holds because $D_{\mathrm{int}}$ has no defect on the
trailing boundary; the canonical-form rule of \cref{lem:composite-canonical-form} holds by the
choice of $W$; the subtype constraint on the $K_0$ and $K_g^{\pm}$ letters
holds by \cref{lem:special-coverage-internal} and the $K_0$ slot conditions
by \cref{lem:sigma-zero,cor:special-sigma-zero}; no-repeat holds by
\cref{lem:special-no-repeat}.  Every letter $C_i$ is therefore
admissible at the corresponding partial word and the leaf is enumerated.

\emph{The jump image is emitted.}
In the defect-jump case, the closing generator omitted by the
special composite generator would restore the removed crossing $p$,
so the pair on its omitted side is exactly $(L_1,L_2)$, along which the jump
of $F_{\mathrm{int}}$ is defined (\cref{lem:special-jump-pair}).
In $D_{\mathrm{int}}$ the
wires $L_1,L_2$ cross inside the region that $F_{\mathrm{ext}}$ occupies in
$D_{\mathrm{ext}}$ (\cref{def:defect-jump}); write $q$ for this crossing.
It splits the region in two, and the part along the external boundary is an
external digon whose only vertex is $q$.  If this digon lies on the initial
boundary, it is closed by a prefix letter; if it lies on the trailing
boundary, it is a trailing digon of a wire pair
of \cref{lem:trailing-digon-pruning}.  At the position of the special composite
generator the ordinary $K_g$ is therefore rejected either by a reducedness
conflict with a prefix elementary generator or by the trailing-digon
pruning, so the leaf triggers the second output; the defect jump
being an involution along the same strip, the emitted image is $D_{\mathrm{ext}}$.

\emph{The $\sigma_0$ exception.}
Consider the case where the white vertex of $t_1$ comes from $\sigma_0$.
No $K_0^{-}$ is defined, so no factorization packages $\sigma_0$ into the
special composite generator (case~(4) of \cref{lem:special-factorization}),
and the route through $D_{\mathrm{int}}$ is unavailable.  Here $\sigma_0$ is
packaged into a $K_0$, and $p$ is produced by its black generator, so $L_1$ is the
wire $0$ and $F_{\mathrm{ext}}$ is the trailing defect at
$(L_1,L_2)=(0,1)$ (an external defect between wires $0$ and $n-1$ falls
under the first branch of \cref{lem:special-coverage-boundary}, and $(0,1)$
is the only trailing pair with wire $0$).
Then $D$ itself is enumerated
directly: it has exactly one defect on the initial boundary, hence is
encoded by a word of the form \eqref{eq:special-word-form}, and the checks
above apply verbatim except for the trailing-digon condition.  The only
trailing pair of \cref{lem:trailing-digon-pruning} that is not a digon of
$D$ is $(0,1)$, crossed by the black generator of $K_0$, where
condition~(ii) is not imposed.
\end{proof}

\begin{remark}[Duplicate words are unavoidable]
\label{rem:2-defective-duplicate-words}
The special composite generators bring an additional
relation~\eqref{eq:special-commutation} with their ordinary neighbors
(illustrated in \cref{fig:defect-class-rhombus}), so beyond the multiplicity of
\cref{obs:representation-multiplicity} the search
may emit several reduced words encoding the same wiring diagram.
These duplicates are removed afterward by grouping the output into the
equivalence classes of \cref{sec:classification}.

The search itself, however, cannot drop either rhombus variant.  Both
encode the same internal rhombus, so as directly emitted leaves they are
duplicates; but each may also trigger a second output
(\cref{prop:2-defective-coverage}), and since the two variants each package
into the special composite generator the letter giving a different white
vertex of the rhombus, their defect jumps remove different crossings and
may yield two distinct boundary-defect diagrams.
\end{remark}

\subsection{Completeness on the level of \texorpdfstring{$P$}{P}-classes}
\label{subsec:special-coverage}

The perfect search was complete at the level of commutation classes,
emitting one word per class (\cref{cor:perfect-class-output}).  At
parameter $X$ the 2-defective search visits only diagrams with a defect on
the initial external boundary between wires $X$ and $X+1$, so no
per-diagram statement holds; completeness is stated on the $P$-classes
(\cref{def:p-equivalence}).

\begin{theorem}[Completeness of the 2-defective search]
\label{thm:2-defective-completeness}
Fix $n\equiv 1\pmod 6$.  For every
$X\in\{1,3,\ldots,n-2\}$, the 2-defective search at parameter $X$
emits at least one wiring-diagram representative of every $P$-class
of 2-defective arrangements of $n$ pseudolines.
\end{theorem}

\begin{proof}
Let $P$ be a $P$-class of 2-defective arrangements.  Pick any
representative $\mathcal{A}$ of $P$ and any defect $F$, and choose a
side of $F$ on the spherical model.  Send the pseudoline carrying this
side to infinity via the projective rotation $\pi_p$
(\cref{def:projective-rotation}); $F$ then becomes an external defect.
The lines at infinity may be labeled freely.
Two of these labelings, differing by direction, place the chosen
side of $F$ on the initial external boundary between wires $X$ and
$X+1$.  Fix either.  By \cref{prop:2-defective-coverage} the
resulting wiring diagram appears in the output of the 2-defective
search at parameter $X$, directly or as the image of a leaf under a
defect jump.
\end{proof}

\begin{remark}[Representatives per \texorpdfstring{$P$}{P}-class realized by the 2-defective search]
\label{rem:2-defective-representatives}
In the proof of \cref{thm:2-defective-completeness}, the choice of
a defect, a side of it, and a direction along the side forms a
\emph{flag} of the representative $\mathcal{A}$ of $P$.

The symmetry group $G_P$ acts on the $n+1$ projective lines (\cref{def:p-symmetry-group}),
hence on the defects and their flags.  We show this action is free.
Suppose $t\in G_P$ fixes a flag.  Choose a line at infinity
and the labeling start and direction along it according to the flag;
they are also fixed by $t$. By \cref{lem:adjacency-preservation},
this determines the labeling itself through the crossing-adjacency, so $t=\mathrm{id}$.
Two flags produce the same wiring diagram exactly when their labelings
differ by an element of $G_P$.
The representatives of $P$ are therefore in bijection with the $G_P$-orbits
of flags, each of size $|G_P|$ by freeness:
\[
  \operatorname{rep}_{\text{2-defective}}(P) =
  \begin{cases}
    16/|G_P| & \text{for two quadrilateral defects }(2\cdot 4\cdot 2),\\
    10/|G_P| & \text{for a single pentagonal defect }(1\cdot 5\cdot 2).
  \end{cases}
\]
The result is independent of $X$. The groups $G_P$ that occur are listed in
\cref{tab:sym-combined}.
\end{remark}

By \cref{thm:2-defective-completeness}, a single $X$ suffices for
completeness.  Running the search at multiple values of $X$ provides
a self-consistency check on the implementation: the $P$-class set
after the deduplication of \cref{sec:classification} must be
invariant under the choice of $X$.  Numerical results of this
consistency check can be found in \cref{app:special-self-consistency}.

\section{Computational results}
\label{sec:computational-results}

Our implementations print one word per arrangement, annotated with the two face
counts maximized in the Kobon and Arnold problems: the triangle number
\(K = a_3(\mathcal{A})\) and the black excess \(A=b-w\).  The full outputs
are stored in the \texttt{data/} directory of the public
repository~\cite{parpalak-utkin-pseudoline-algorithms}, along with
the search implementations (\texttt{search/}) and the classification
tools (\texttt{combinatorics/}).

\subsection{Extracting triangle-maximal arrangements}
\label{subsec:triangle-maximal-classes}

The two searches enumerate the black-ma\-xi\-mal family of
\cref{prob:main}.  For $n\equiv 3,5\pmod 6$, by
\cref{lem:defect-free-perfect}, every black-maximal
arrangement is perfect, so it attains the affine triangle
bound \eqref{eq:a3-bound}, $a_3(\mathcal{A}) = \lfloor n(n-2)/3\rfloor$;
the perfect search of \cref{sec:perfect-search} thereby enumerates
tri\-angle-ma\-xi\-mal arrangements directly.
For $n\equiv 1\pmod 6$, triangle-maximal arrangements are extracted
from the 2-defective family of \cref{sec:2-defective-search} by the
projective and affine refinements.

\emph{Projective.}  By \cref{lem:optimal-defects-2q1p} a 2-defective
arrangement carries either one pentagonal defect or two quadrilateral
defects, and the pentagon case has one more projective triangle.
Pentagonal-defect arrangements are therefore projectively
triangle-maximal.  This matches \cite[Lemma~3.1]{blanc-2011}, where the
same statement is proved directly in the projective formulation.
Pentagonal-defect arrangements first appear at $n=25$.  For
$n=7,13,19$, the entire 2-defective family is two-quadrilateral and
projectively triangle-maximal as a whole.

\emph{Affine.}  By \cref{lem:a3-max-characterization}, at
$n\equiv 1\pmod 6$ an arrangement attains
\(a_3(\mathcal{A}) = \lfloor n(n-2)/3\rfloor\) if and only if it is
2-defective with all defects external.  To enumerate all $a_3$-maximal
arrangements one expands each \(P\)-class of the search output into its
constituent \(E\)-classes (one per affine chart) and keeps those
\(E\)-classes whose defects are all external.
When a single example arrangement is available rather than an exhaustive
enumeration, its internal defects can often be made external by a defect jump
(\cref{def:defect-jump}) or a change of affine chart.  The first-hit
output shows this (files \texttt{data/partial/*.words-part.txt}): a leaf
and its defect-jump image occupy consecutive lines, differing by one in
the triangle number \(K\).

\subsection{Exhaustive enumerations}
\label{subsec:exhaustive-enumerations}

\Cref{tab:complete-counts} gives the summary of the exhaustive
enumerations we were able to carry out.
\(\#D\) is the number of wiring diagrams (equivalently the number of
reduced words modulo commutations~\eqref{eq:commute} or of distinct
\(O\)-matrices, see \cref{lem:o-matrix-commutation,rem:o-matrix-wiring}).
\(\#E\) and \(\#P\) are the numbers of \(E\)- and \(P\)-classes
(\cref{def:e-equivalence,def:p-equivalence}), which identify wiring
diagrams under Euclidean rotations and reflections, and changes of
affine chart.

\begin{table}[ht]
\centering
\captionsetup{width=0.68\linewidth}
\caption{Exhaustive enumeration counts. \emph{$^\dagger$The column for $n=25$ records the
single-pentagon defect family only.}}
\label{tab:complete-counts}
\setlength{\tabcolsep}{2.93pt}
\begin{tabular*}{\linewidth}{l@{\hspace{1em}}rrrrrrrrrrrrr}
\toprule
\(n\) & 3 & 5 & 7 & 9 & 11 & 13 & 15 & 17 & 19 & 21 & 23 & \textit{25 (5-gon)}$^\dagger\!\!$ & 27 \\
\midrule
\#\(D\) & 1 & 1 & 28 & 3 & 0 & 1\,456 & 96 & 255 & 991\,040 & 9\,548 & 104\,512 & \textit{3\,021\,200} & 85\,562\,064 \\
\#\(E\) & 1 & 1 & 3 & 1 & 0 & 60 & 4 & 10 & 26\,084 & 236 & 2\,272 & \textit{60\,424} & 1\,584\,540 \\
\#\(P\) & 1 & 1 & 1 & 1 & 0 & 6 & 1 & 3 & 1\,312 & 18 & 112 & \textit{2\,324} & 56\,646 \\
\bottomrule
\end{tabular*}
\end{table}

The zero entry at $n=11$ shows the known non-existence of perfect
arrangements in this case~\cite{blanc-2011}.  Comparisons with the counts
published in prior work are collected in \cref{subsec:published-enumerations}.

The counts in \cref{tab:complete-counts} are obtained by different methods for the
two searches, because each covers a different level of the class
hierarchy.  The perfect search emits one word per commutation class
(\cref{cor:perfect-class-output}), hence one wiring diagram, so the
output size is already \(\#D\).  Deduplication under \(E\)- and
\(P\)-equivalence (\cref{def:e-equivalence,def:p-equivalence}),
implemented by computing a canonical representative per class
(\cref{app:canonicalization:forms}), gives \(\#E\) and \(\#P\).
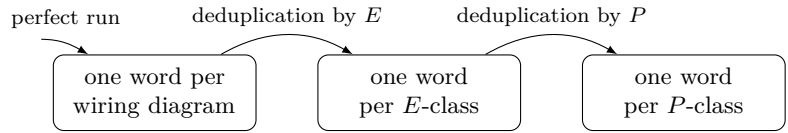
\begin{center}
\begin{tikzpicture}[node distance=3.5cm,
  block/.style={draw, rounded corners, text width=2.4cm, align=center,
                font=\small, minimum height=1cm, inner sep=4pt},
  >=latex
]
\node[block] (a) {one word per wiring diagram};
\node[block, right of=a] (b) {one word per $E$-class};
\node[block, right of=b] (c) {one word per $P$-class};
\draw[->] ([xshift=-1.5cm,yshift=0.2cm]a.north) to[out=0, in=150] node[above, yshift=2.2pt] {\footnotesize perfect run} (a);
\draw[->] (a) to[out=30, in=150] node[above] {\footnotesize deduplication by $E$} (b);
\draw[->] (b) to[out=30, in=150] node[above] {\footnotesize deduplication by $P$} (c);
\end{tikzpicture}
\end{center}

The 2-defective search emits several words per \(P\)-class
(\cref{rem:2-defective-representatives}), so its output is deduplicated to
\(P\)-classes first, giving \(\#P\).  The constituent \(E\)-classes and
wiring diagrams can be recovered by the reverse process: each
\(P\)-class representative is expanded by applying the projective
rotations \(\pi_p\) and Euclidean rotations \(\rho^k\)
(\cref{def:projective-rotation,def:euclidean-rotation}) to its
\(O\)-matrix, and \(\#E\) and \(\#D\) are obtained by aggregating the
distinct results.
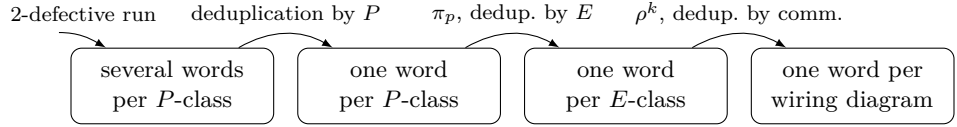
\begin{center}
\begin{tikzpicture}[node distance=3cm,
  block/.style={draw, rounded corners, text width=2.4cm, align=center,
                font=\small, minimum height=1cm, inner sep=4pt},
  >=latex
]
\node[block] (a) {several words per $P$-class};
\node[block, right of=a] (b) {one word per $P$-class};
\node[block, right of=b] (c) {one word per $E$-class};
\node[block, right of=c] (d) {one word per wiring diagram};
\draw[->] ([xshift=-1.5cm,yshift=0.2cm]a.north) to[out=0, in=150] node[above, yshift=2.2pt] {\footnotesize 2-defective run} (a);
\draw[->] (a) to[out=30, in=150] node[above] {\footnotesize deduplication by $P$} (b);
\draw[->] (b) to[out=30, in=150] node[above, align=center] {\footnotesize $\pi_p$, dedup.\ by $E$} (c);
\draw[->] (c) to[out=30, in=150] node[above, align=center] {\footnotesize $\rho^k$, dedup.\ by comm.} (d);
\end{tikzpicture}
\end{center}

The enumerations were performed on personal computers, not in a
controlled environment, so the single-core CPU times below are approximate:
well under a second for $n \le 17$; for the perfect search,
${\sim}\,3$\,s at $n=21$ and ${\sim}\,1.8$\,min at $n=23$;
for the 2-defective search, ${\sim}\,8$\,s at $n=19$ and
${\sim}\,3.3$\,min at $n=25$ (5-gon).
The $n=27$ case would require ${\sim}\,7$ single-core days and
was produced by parallel computation.
Deduplicating its output to the $56\,646$ $P$-classes requires storing
one canonical $O$-matrix per $E$-class in memory (about $1.6$ million
representatives, or ${\sim}\,1.3$\,GB).

The $25$ (5-gon) column of \cref{tab:complete-counts} records only the
single-pentagon defect family.
After projective closure, it gives the triangle-maximal arrangements
of $n=26$ projective pseudolines.
We tested this family for stretchability~\cite{parpalak-utkin-2026}:
some representatives are stretchable, but none of them admit the
iterative constructions
of~\cite{forge-ramirez-1998,bartholdi-blanc-loisel-2007}.

The two-quadrilateral case at $n = 25$ is also algorithmically reachable,
but we chose not to run it, given the size of the expected
output.  The search would emit ${\sim}\,200$\,GB
of words and yield an estimated ${\sim}\,10^7$ projective classes.
The output is too large to deduplicate in memory, and we saw no reason
to produce that many representatives.

\subsection{Symmetry profiles}
\label{subsec:symmetry-profiles}

For every \(P\)-class in the exhaustive enumerations of
\cref{tab:complete-counts} we record its symmetry profile: the
\(P\)-symmetry group \(G_P\) (\cref{def:p-symmetry-group}) together with
the orbit--stabilizer data of its constituent \(E\)-classes, computed
as in \cref{app:canonicalization:groups}.
\Cref{tab:sym-combined} collects these profiles; drawings illustrating
them are given in \cref{app:symmetry-drawings}.  The caption of
\cref{fig:symmetry-n15-n17-a} explains one profile in detail.

Column legend:
\begin{itemize}[label=$\circ$, itemsep=0pt]
\item \(G_P\), \(|G_P|\) --- type and order of the \(P\)-symmetry group.
\item \(\#E/P\) --- number of \(E\)-classes of the given type inside
  one \(P\)-class.
\item \(m(E)\) --- multiplicity (\cref{def:multiplicity}).
\item \(\operatorname{rep}(E)\) --- number of wiring diagrams
  (\(O\)-matrices, \cref{rem:o-matrix-wiring}) in one \(E\)-class, equal
  to \(2n/|H_E|\) (\cref{lem:parity-fixed-representatives}).
\item \(H_E\) --- Euclidean stabilizer (\cref{def:euclidean-stabilizer}),
  the symmetry group of an \(E\)-class of this row: cyclic \(C_k\) or
  dihedral \(D_k\), \(D_1\) a single mirror reflection.
\item \(\operatorname{rep}(P)\) --- all wiring diagrams in one \(P\)-class
of this profile.
\item \(\#P\) --- number of \(P\)-classes sharing this profile.
\item \(\#E\), \(\#D\) --- total \(E\)-classes and wiring diagrams
  contributed by this profile (\(\#P\cdot\sum_{\text{rows}}(\#E/P)\) and
  \(\#P\cdot\operatorname{rep}(P)\) respectively).
\end{itemize}

\begin{table}[p]
\centering
\captionsetup{width=\linewidth}
\caption{Symmetry profiles of the \(P\)-classes.
\(G_P\): \(P\)-symmetry group;
\(\#E/P\): \(E\)-classes of this type in one \(P\)-class;
\(m(E)\): multiplicity;
\(\operatorname{rep}(\cdot)\): all wiring diagrams in one class;
\(H_E\): symmetry group of an \(E\)-class.
\emph{$^\dagger$The row for \(n=25\) records the single-pentagon defect family,
two-quadrilateral arrangements are skipped.}}
\label{tab:sym-combined}
\setlength{\tabcolsep}{4pt}
\renewcommand{\arraystretch}{0.96}
\setlength{\aboverulesep}{0.05ex}
\setlength{\belowrulesep}{0.355ex}
\begin{tabular*}{\linewidth}{@{}l@{\hspace{0.7em}}l@{\extracolsep{\fill}}rrrrlrrrr@{}}
\toprule
 & & & \multicolumn{4}{c}{$E$-classes} & & \multicolumn{3}{c}{counts} \\
\cmidrule(lr){4-7}\cmidrule(lr){9-11}
$n$ & $G_P$ & $|G_P|$ & \#$E/P$ & $m(E)$ & $\operatorname{rep}(E)$ & $H_E$ & $\operatorname{rep}(P)$ & \#$P$ & \#$E$ & \#$D$ \\
\midrule
$3$ & $S_{4}$ & 24 & 1 & 4 & 1 & $D_{3}$ & 1 & 1 & 1 & 1 \\
\midrule
$5$ & $A_{5}$ & 60 & 1 & 6 & 1 & $D_{5}$ & 1 & 1 & 1 & 1 \\
\midrule
\multirow{2}{*}{$7$} & \multirow{2}{*}{$D_{2}$} & \multirow{2}{*}{4} & 2 & 2 & 7 & $D_{1}$ & \multirow{2}{*}{28} & \multirow{2}{*}{1} & \multirow{2}{*}{3} & \multirow{2}{*}{28} \\
 &  &  & 1 & 4 & 14 & $C_{1}$ &  &  &  &  \\
\midrule
$9$ & $A_{5}$ & 60 & 1 & 10 & 3 & $D_{3}$ & 3 & 1 & 1 & 3 \\
\midrule
\multirow{4}{*}{$13$} & \multirow{2}{*}{$C_{2}$} & \multirow{2}{*}{2} & 2 & 1 & 13 & $D_{1}$ & \multirow{2}{*}{182} & \multirow{2}{*}{4} & \multirow{2}{*}{32} & \multirow{2}{*}{728} \\
 &  &  & 6 & 2 & 26 & $C_{1}$ &  &  &  &  \\
\cmidrule{2-11}
 & $C_{1}$ & 1 & 14 & 1 & 26 & $C_{1}$ & 364 & 2 & 28 & 728 \\
\cmidrule{2-11}
 & \multicolumn{7}{@{}l}{Totals for $n=13$} & 6 & 60 & 1\,456 \\
\midrule
\multirow{2}{*}{$15$} & \multirow{2}{*}{$C_{5}$} & \multirow{2}{*}{5} & 1 & 1 & 6 & $C_{5}$ & \multirow{2}{*}{96} & \multirow{2}{*}{1} & \multirow{2}{*}{4} & \multirow{2}{*}{96} \\
 &  &  & 3 & 5 & 30 & $C_{1}$ &  &  &  &  \\
\midrule
\multirow{5}{*}{$17$} & \multirow{2}{*}{$D_{3}$} & \multirow{2}{*}{6} & 2 & 3 & 17 & $D_{1}$ & \multirow{2}{*}{102} & \multirow{2}{*}{2} & \multirow{2}{*}{8} & \multirow{2}{*}{204} \\
 &  &  & 2 & 6 & 34 & $C_{1}$ &  &  &  &  \\
\cmidrule{2-11}
 & \multirow{2}{*}{$A_{4}$} & \multirow{2}{*}{12} & 1 & 6 & 17 & $D_{1}$ & \multirow{2}{*}{51} & \multirow{2}{*}{1} & \multirow{2}{*}{2} & \multirow{2}{*}{51} \\
 &  &  & 1 & 12 & 34 & $C_{1}$ &  &  &  &  \\
\cmidrule{2-11}
 & \multicolumn{7}{@{}l}{Totals for $n=17$} & 3 & 10 & 255 \\
\midrule
\multirow{5}{*}{$19$} & $C_{1}$ & 1 & 20 & 1 & 38 & $C_{1}$ & 760 & 1\,296 & 25\,920 & 984\,960 \\
\cmidrule{2-11}
 & $C_{2}$ & 2 & 10 & 2 & 38 & $C_{1}$ & 380 & 12 & 120 & 4\,560 \\
\cmidrule{2-11}
 & \multirow{2}{*}{$C_{2}$} & \multirow{2}{*}{2} & 2 & 1 & 19 & $D_{1}$ & \multirow{2}{*}{380} & \multirow{2}{*}{4} & \multirow{2}{*}{44} & \multirow{2}{*}{1\,520} \\
 &  &  & 9 & 2 & 38 & $C_{1}$ &  &  &  &  \\
\cmidrule{2-11}
 & \multicolumn{7}{@{}l}{Totals for $n=19$} & 1\,312 & 26\,084 & 991\,040 \\
\midrule
\multirow{9}{*}{$21$} & \multirow{2}{*}{$C_{3}$} & \multirow{2}{*}{3} & 1 & 1 & 14 & $C_{3}$ & \multirow{2}{*}{308} & \multirow{2}{*}{8} & \multirow{2}{*}{64} & \multirow{2}{*}{2\,464} \\
 &  &  & 7 & 3 & 42 & $C_{1}$ &  &  &  &  \\
\cmidrule{2-11}
 & $C_{1}$ & 1 & 22 & 1 & 42 & $C_{1}$ & 924 & 7 & 154 & 6\,468 \\
\cmidrule{2-11}
 & \multirow{3}{*}{$A_{4}$} & \multirow{3}{*}{12} & 1 & 4 & 14 & $C_{3}$ & \multirow{3}{*}{77} & \multirow{3}{*}{2} & \multirow{3}{*}{6} & \multirow{3}{*}{154} \\
 &  &  & 1 & 6 & 21 & $D_{1}$ &  &  &  &  \\
 &  &  & 1 & 12 & 42 & $C_{1}$ &  &  &  &  \\
\cmidrule{2-11}
 & \multirow{2}{*}{$C_{2}$} & \multirow{2}{*}{2} & 2 & 1 & 21 & $D_{1}$ & \multirow{2}{*}{462} & \multirow{2}{*}{1} & \multirow{2}{*}{12} & \multirow{2}{*}{462} \\
 &  &  & 10 & 2 & 42 & $C_{1}$ &  &  &  &  \\
\cmidrule{2-11}
 & \multicolumn{7}{@{}l}{Totals for $n=21$} & 18 & 236 & 9\,548 \\
\midrule
\multirow{3}{*}{$23$} & $C_{1}$ & 1 & 24 & 1 & 46 & $C_{1}$ & 1\,104 & 86 & 2\,064 & 94\,944 \\
\cmidrule{2-11}
 & $C_{3}$ & 3 & 8 & 3 & 46 & $C_{1}$ & 368 & 26 & 208 & 9\,568 \\
\cmidrule{2-11}
 & \multicolumn{7}{@{}l}{Totals for $n=23$} & 112 & 2\,272 & 104\,512 \\
\midrule
$\mathit{25^{\dagger}}$ & \textit{$C_{1}$} & \textit{1} & \textit{26} & \textit{1} & \textit{50} & \textit{$C_{1}$} & \textit{1\,300} & \textit{2\,324} & \textit{60\,424} & \textit{3\,021\,200} \\
\midrule
\multirow{4}{*}{$27$} & $C_{1}$ & 1 & 28 & 1 & 54 & $C_{1}$ & 1\,512 & 56\,560 & 1\,583\,680 & 85\,518\,720 \\
\cmidrule{2-11}
 & \multirow{2}{*}{$C_{3}$} & \multirow{2}{*}{3} & 1 & 1 & 18 & $C_{3}$ & \multirow{2}{*}{504} & \multirow{2}{*}{86} & \multirow{2}{*}{860} & \multirow{2}{*}{43\,344} \\
 &  &  & 9 & 3 & 54 & $C_{1}$ &  &  &  &  \\
\cmidrule{2-11}
 & \multicolumn{7}{@{}l}{Totals for $n=27$} & 56\,646 & 1\,584\,540 & 85\,562\,064 \\
\bottomrule

\end{tabular*}
\end{table}

Within each profile, \(\sum_{\text{rows}} (\#E/P)\cdot m(E) = n+1\), the number of
affine charts of a \(P\)-class.  When an \(n\) has several profiles, the
table's total row reproduces the \(\#P\), \(\#E\), \(\#D\) columns of
\cref{tab:complete-counts}.  For 2-defective \(n\) this is tautological:
those counts were filled by the same \(P\)-class expansion.  Here
\(\operatorname{rep}(P)\) counts \emph{all} wiring diagrams of a \(P\)-class;
the 2-defective search realizes only the subset
\(\operatorname{rep}_{\text{2-defective}}(P)\) of \cref{rem:2-defective-representatives}.
For perfect \(n\) the agreement is an independent
check: there \(\#D\) is the raw search output
(\cref{cor:perfect-class-output}), so reproducing it from the
\(P\)-classes confirms that the search loses no wiring diagram, that the
deduplication is correct, and that arrangements are correctly
reconstructed from their \(P\)-classes.

\subsection{Comparison with prior enumerations}
\label{subsec:published-enumerations}

The counts of \cref{tab:complete-counts,tab:sym-combined} admit
independent cross-checks against two prior enumerations, collected in
\cref{tab:published-comparison}.  Wherever a comparison is defined, the
published counts agree with ours.

Bokowski, Roudneff, and Strempel (BRS) work in the projective plane
\cite{bokowski-roudneff-strempel-1997}, with one more line than our
affine \(n\).  A triangle-maximal arrangement in their sense is a
combinatorial type of projective arrangement, that is, a \(P\)-class
(\cref{def:p-equivalence}), so their counts match our \(\#P\).  Savchuk \cite{savchuk-2025} counts
affine tables, one per \(E\)-class (\cref{def:e-equivalence}), so his
totals match our \(\#E\), and, when a symmetry is imposed, the number of
\(E\)-classes whose stabilizer \(H_E\) contains it.

\begin{table}[H]
\centering
\caption{Published enumeration counts compared with our results.
Parenthetical tags are each source's symmetry option: BRS's ``order 3'';
Savchuk's rotational \texttt{-R}\,\(k\) and mirror \texttt{-M}
(our \(C_k\) and \(D_1\)).
}
\label{tab:published-comparison}
\setlength{\tabcolsep}{4pt}
\renewcommand{\arraystretch}{1.02}
\begin{tabular*}{\linewidth}{@{\extracolsep{\fill}}llll@{}}
\toprule
\(n\), sym.\ & BRS (\(n{+}1\)) & Savchuk & counts in \cref{tab:complete-counts,tab:sym-combined} \\
\midrule
\(n=3,5,9\) & 1 & 1  & \(\#P=\#E=1\) \\
\addlinespace
\(n=11\) & 0 & 0 & \(\#P=\#E=0\) \\
\addlinespace
\(n=15\) & 1 & 4 & \(\#P=1\), \(\#E=4\) \\
\addlinespace
\(n=15\), \(C_3\) & -- & 0\,(\texttt{-R}\,3) & no \(C_3\subseteq H_E\) \\
\addlinespace
\(n=15\), \(C_5\) & -- & 1\,(\texttt{-R}\,5) & 1 \(P\)-class: \(G_P=C_5\), \(H_E=C_5\), \(\#E/P=1\) \\
\addlinespace
\(n=17\) & 3 & 10 & \(\#P=3\), \(\#E=10\) \\
\addlinespace
\(n=21\) & -- & 236 & \(\#E=236\) \\
\addlinespace
\(n=21\), \(D_1\) & -- & 4\,(\texttt{-M})
  & \(\!\!\left[\begin{array}{@{}l@{}}
      2\ P\text{-classes:}\ G_P=A_4,\ H_E=D_1,\ \#E/P=1\\
      1\ P\text{-class:}\ G_P=C_2,\ H_E=D_1,\ \#E/P=2
    \end{array}\right.\) \\
\addlinespace
\(n=21\), \(C_3\) & 10\,(order 3) & 10\,(\texttt{-R}\,3)
  & \(\!\!\left[\begin{array}{@{}l@{}}
      8\ P\text{-classes:}\ G_P=C_3,\ H_E=C_3,\ \#E/P=1\\
      2\ P\text{-classes:}\ G_P=A_4,\ H_E=C_3,\ \#E/P=1
    \end{array}\right.\) \\
\addlinespace
\(n=21\), \(C_7\) & -- & 0\,(\texttt{-R}\,7) & no \(C_7\subseteq H_E\) \\
\addlinespace
\(n=23\), \(D_1\) & -- & 0\,(\texttt{-M}) & no \(D_1\subseteq H_E\) \\
\addlinespace
\(n=27\), \(C_3\) & 86\,(order 3) & 86\,(\texttt{-R}\,3) & 86 \(P\)-classes: \(G_P=C_3\), \(H_E=C_3\), \(\#E/P=1\) \\
\addlinespace
\(n=27\), \(D_1\) & -- & 0\,(\texttt{-M}) & no \(D_1\subseteq H_E\) \\
\addlinespace
\(n=27\), \(C_9\) & -- & 0\,(\texttt{-R}\,9) & no \(C_9\subseteq H_E\) \\
\bottomrule
\end{tabular*}
\end{table}

Rows without symmetry are read from \cref{tab:complete-counts}; rows
with a symmetry are read from \cref{tab:sym-combined}.  A zero
entry means no \(E\)-class has an \(H_E\) containing the listed subgroup.
For the 2-defective cases (\(n\equiv1\pmod6\)) Savchuk reports the
number of arrangements with a given placement of the unused segments.
This is neither a total count nor a count with an imposed symmetry, so
it matches none of our counts, and we did not compare it.

The two enumerations produce their counts by different means.  BRS's
algorithm constructs arrangements as hyperline sequences by constraint
propagation, one per wiring diagram; the
projective-class counts they publish therefore rest on a deduplication
to combinatorial types that their paper does not describe.  Savchuk's
SAT-based search reaches the \(E\)-class level directly: each
arrangement it finds is excluded, together with its full orbit under the
Euclidean rotations and reflections
(\cref{def:euclidean-rotation,def:reflection}), before the next is
generated.

We also reconstructed the algorithm of BRS.  Its completed hyperline
sequences carry the same data as an \(O\)-matrix, which we convert to
reduced words (\cref{app:word-reconstruction}).
The reconstruction enumerates all perfect
arrangements through \(n=23\) in about \(1.3\) hours on a single core,
and its
output coincides with that of the perfect search as a set of wiring
diagrams, up to a global reflection.  This is a check at the level of
the individual diagrams, not only of the counts.  Unlike our search,
whose enumeration tree is fixed, the reconstruction forms its tree at
runtime from the current partial data structure; these runtime choices
govern its total running time, not only the time to the first
arrangement (cf.\ \cref{rem:branching-order}).

These cross-checks rest on methods of an entirely different nature
from our word-based search, so their agreement is an independent check
of our implementation.

\subsection{Partial and first-hit searches}
\label{subsec:first-hit}

For $n > 27$, exhaustive enumeration becomes impractical on personal
computers in both time and output size, so we instead run the
same algorithms only until they emit one example arrangement.

The resulting leaves are distributed irregularly across the search tree, so
the time to the first one depends strongly on the order in which the tree is
traversed.  We observed that, for perfect arrangements, choosing
the ascending order (\cref{rem:branching-order}) and the first composite
generator $K_{n-3}$ leads to fast runs.  The results are summarized in
\cref{tab:first-hit-perfect}.
\(T\) is not monotone in~$n$, illustrating the irregularity.

\begin{table}[H]
\centering
\setlength{\tabcolsep}{3pt}%
\caption{First-hit times for perfect arrangements found by the implemented
algorithm of \cref{subsec:algorithm}.
$T$ is a normalized single-core time derived from iteration counts.\protect\footnotemark}
\label{tab:first-hit-perfect}
\newcommand{\tu}[1]{{\footnotesize\,#1}}
\begin{tabular*}{0.9\linewidth}{@{\extracolsep{\fill}}l*{11}{p{2.2em}}@{}}
\toprule
\(n\)\, & 29 & 33 & 35 & 39 & 41 & 45 & 47 & 51 & 53 & 57 & 59 \\
\(T\)\, & $25$\tu{$\mu$s} & $11$\tu{ms} & $81$\tu{$\mu$s} & $19$\tu{ms} & $32$\tu{ms} & $5$\tu{ms} & $80$\tu{ms} & $10$\tu{ms} & $26$\tu{ms} & $5.5$\tu{s} & $4.9$\tu{s} \\
\midrule
\(n\)\, & 63 & 65 & 69 & 71 & 75 & 77 & 81 & 83 & 87 & 89 & 93 \\
\(T\)\, & $31$\tu{s} & $1.7$\tu{m} & $3.3$\tu{m} & $5.0$\tu{h} & $12$\tu{s} & $43$\tu{m} & $3.0$\tu{m} & $26$\tu{m} & $2.5$\tu{h} & $41$\tu{h} & $16$\tu{h} \\
\bottomrule
\end{tabular*}
\end{table}
\footnotetext{Directly measured wall-clock times are noisy and
machine-dependent, whereas the
DFS iteration count $N_{\mathrm{iter}}$ to the first hit is deterministic.
We report a normalized time as an intuitive indication of how far the search must
walk the tree to reach an arrangement, obtained from the iteration count as
$T=N_{\mathrm{iter}}/R$.  The single-core rate $R\approx3.1\times10^{8}$
iterations/s is measured across the perfect $n=39$--$81$ runs on one machine,
an AMD Ryzen AI 9 HX 370 (stable to within ${\pm}5\%$, no decay with~$n$); the 2-defective rate is less
stable, varying by about ${\pm}30\%$ across different runs in the reported time range.
$N_{\mathrm{iter}}$ is specific to our DFS implementation, which can be found in the
repository~\cite{parpalak-utkin-pseudoline-algorithms} together with the
arrangements and their iteration counts.}%

For the 2-defective search no fixed parameters work uniformly
across~$n$. Instead, for each $n$ we run a grid over the branch order, the
first composite generator $K_{g_1}$, and the skipped odd generator $X$
(\cref{lem:special-coverage-internal}) under a fixed per-run budget of iterations
(an iteration is one candidate generator tried at the current level), keeping the
arrangement reached in the fewest iterations.  If no run succeeds within the
budget, we raise it and repeat. \Cref{tab:first-hit-2defective}
lists the winning parameters and the resulting single-core time for each
$n$.

\begin{table}[H]\centering
\setlength{\tabcolsep}{3.6pt}%
\caption{First-hit times for 2-defective arrangements found by the implemented
algorithm of \cref{subsec:2-defective-algorithm}, with $X$, $g_1$, and
branch order chosen per instance.
$T$ is a normalized single-core time derived from iteration counts
(as in \cref{tab:first-hit-perfect}).}
\label{tab:first-hit-2defective}
\newcommand{\tu}[1]{{\footnotesize\,#1}}
\begin{tabular*}{0.9\linewidth}{@{\extracolsep{\fill}}l*{9}{p{2.6em}}@{}}
\toprule
\(n\)\,      & 31 & 37 & 43 & 49 & 55 & 61 & 67 & 73 & 79 \\
\midrule
\(X\)\,      & 5 & 11 & 7 & 3 & 27 & 15 & 65 & 11 & 11 \\
\(g_1\)\,    & 6 & 16 & 6 & 20 & 20 & 20 & 38 & 32 & 44 \\
branch       & desc & asc & desc & asc & asc & asc & asc & asc & asc \\
\midrule
\(T\)\,      & $4.7$\tu{$\mu$s} & $17$\tu{$\mu$s} & $0.68$\tu{ms} & $3.1$\tu{ms} & $5.5$\tu{ms} & $0.14$\tu{s} & $0.17$\tu{s} & $1.4$\tu{s} & $2.9$\tu{s} \\
\bottomrule
\end{tabular*}
\end{table}

To reach a single arrangement we relied only on varying the parameters
described above.  More
advanced methods (randomized restarts or heuristic estimates of where the
solutions lie) could reach the first arrangement sooner, but we leave these to
future work.

The examples reported above can be used in the known iterative constructions.
These constructions build a triangle-maximal arrangement from a smaller one:
doubling
$n\mapsto 2n-1$~\cite{roudneff-1986,forge-ramirez-1998,bartholdi-blanc-loisel-2007},
tripling $n\mapsto 3n$~\cite{harborth-1985,roudneff-1986}, and the
constructions $n\mapsto 7n$ and
$n\mapsto 8n-1$~\cite{bokowski-roudneff-strempel-1997}.
All of them apply to perfect arrangements; the doubling
of~\cite{bartholdi-blanc-loisel-2007} applies to 2-defective arrangements
as well (thus these constructions do not mix the perfect and 2-defective families).

An odd size $n$ is \emph{reducible} if some construction builds an
arrangement of size $n$ from one of smaller size, and \emph{irreducible}
otherwise.  The irreducible sizes are
$3, 7, 11, 19, 31, 43, 47, 55, 59, 67, 75, 79, 83, 91 \dots$
(no arrangement exists at the irreducible size $n=11$, the zero entry of
\cref{tab:complete-counts}).

Not every size in
\cref{tab:first-hit-perfect,tab:first-hit-2defective} requires a direct
example: at a reducible size an arrangement can be built from a smaller
one by the constructions mentioned above.  Among the remaining irreducible
sizes, $n=43,47,55,59,67,75,79,83$ have not been reported in the literature.
The words themselves are published in the
repository~\cite{parpalak-utkin-pseudoline-algorithms}.

\begin{corollary}[Existence range]
\label{cor:existence-range}
A triangle-maximal simple pseudoline arrangement, attaining the
bound~\eqref{eq:a3-bound}, exists for every odd $n\le 89$
except $n=11$, for which none exists.  The smallest odd size at which
existence remains open is $n=91$.
\end{corollary}

\begin{proof}
For odd $n\le 27$ the statement follows from the exhaustive enumeration,
reported in \cref{tab:complete-counts}. Among the larger sizes the irreducible
ones are $31$, $43$, $47$, $55$, $59$, $67$, $75$, $79$, $83$
(then $91$, $95$, $\dots$); for each, a run reaching an arrangement
that attains the bound~\eqref{eq:a3-bound} is reported in
\cref{tab:first-hit-perfect,tab:first-hit-2defective}, and the arrangement
is stored in the repository~\cite{parpalak-utkin-pseudoline-algorithms}.
Every remaining odd $n\le 89$ is reducible and inherits
existence from a smaller size.  The next irreducible size,
$n=91$, has no direct example, and no construction produces
an arrangement of this size: it is 2-defective ($91\equiv 1\pmod 6$),
so only doubling could produce one, and $91=2m-1$ forces the even source $m=46$.
\end{proof}

\appendix
\numberwithin{table}{section}
\numberwithin{figure}{section}
\section{Drawings illustrating the symmetry profiles}
\label{app:symmetry-drawings}

\begingroup
\captionsetup{skip=4pt}
\setlength{\floatsep}{8pt plus 2pt minus 2pt}
\setlength{\textfloatsep}{10pt plus 2pt minus 2pt}
\setlength{\intextsep}{8pt plus 2pt minus 2pt}

We computed the symmetry profile of every \(P\)-class in our
exhaustive enumerations; the combined profile across all \(n\) is
collected in \cref{tab:sym-combined}.  Some
arrangements are drawn below on the sphere in
the antipodal model, with \(n+1\)
lines (the affine arrangement completed by the line at infinity, as in
\cref{subsec:checkerboard}).  White faces are colored by side count:
quadrilaterals yellow, pentagons red, hexagons green, heptagons
purple, octagons blue.  For small \(n\) we use stretched drawings; for
larger \(n\) the unstretched spherical drawings make the incidence
structure more visible.

Some profiles are highlighted in the captions of the figures below.  Two
remarks help in matching the \(G_P\) column of \cref{tab:sym-combined}
against the spherical drawings.

\emph{\(C_2\) and \(D_1\) coincide under antipodal identification.}
The \(G_P\) column reports the abstract isomorphism type, and we
write \(C_2\) for groups of order~2.  On the antipodal sphere the
same involution is realized by two transformations of \(S^2\).
In coordinates with the symmetry axis along \(z\), the
rotation by \(\pi\) about this axis sends \((x,y,z)\mapsto(-x,-y,z)\), the
reflection across the plane \(z=0\) sends
\((x,y,z)\mapsto(x,y,-z)\), and these two images are antipodal.
Both names therefore describe the same element of \(G_P\), and every
such figure shows both patterns at once: a \(C_2\) rotation about the
polar axis and a \(D_1\) reflection across the equatorial plane.

\emph{Polyhedral symmetry.} The polyhedral
groups \(S_4\), \(A_5\), \(A_4\) cannot be realized within the dihedral
group \(D_{2n}\) of a single affine chart, so wherever they appear as
\(G_P\), the symmetry is realized only in the projective plane.  All remaining
\(G_P\) groups in the enumeration are cyclic \(C_k\) (including the trivial
\(C_1\)) or dihedral \(D_k\) (these too may be visible only on the sphere,
e.g., \(D_2\) at \(n=7\), \(D_3\) at \(n=17\)).

The two largest, \(S_4\) and \(A_5\), occur only at the small \(n\)
and can be read directly off the sphere (the antipodal model shows
each independent direction twice).  Looking into a white face
center, one sees each of their arrangements repeat the same local picture from
several directions: the \(S_4\) arrangement (\(n=3\)) from six directions,
each carrying \(D_4\) symmetry, just as a cube has six faces
(\(6\cdot|D_4|=48=2|S_4|\)); the \(A_5\) arrangement (\(n=5\)) from
twelve directions, each \(D_5\), just as an icosahedron has twelve
vertices (\(12\cdot|D_5|=120=2|A_5|\)).  Looking into a black face
center, one sees \(D_3\) instead: eight
directions at \(n=3\) and twenty at \(n=5\), the same group decomposed
the other way (\(8\cdot|D_3|=48\), \(20\cdot|D_3|=120\)).

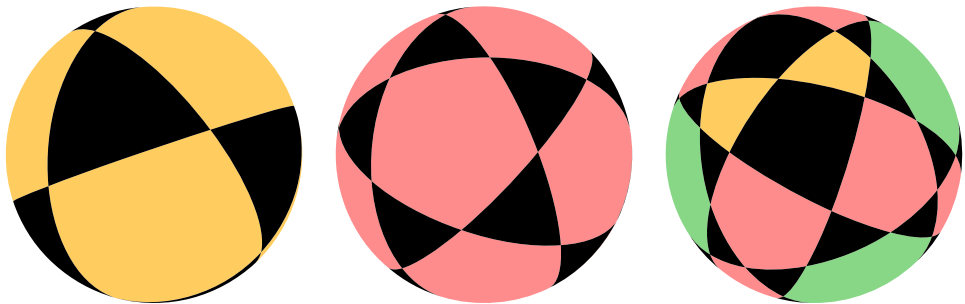
\begin{figure}[H]
\centering
\begin{minipage}{0.31\linewidth}
\centering
\begin{tikzpicture}[x=\linewidth,y=\linewidth]
\useasboundingbox (0,0) rectangle (1,1);
\clip (0,0) rectangle (1,1);
\definecolor{cell4}{RGB}{255,205,95}
\fill[cell4](0.461,0.998)..controls(0.407,0.994)and(0.351,0.967)..(0.301,0.917)..controls(0.35,0.903)and(0.414,0.865)..(0.484,0.806)..controls(0.554,0.748)and(0.628,0.668)..(0.69,0.584)..controls(0.779,0.613)and(0.853,0.635)..(0.9,0.649)..controls(0.949,0.662)and(0.972,0.667)..(0.972,0.664)..controls(0.937,0.768)and(0.866,0.859)..(0.774,0.919)..controls(0.682,0.979)and(0.57,1.007)..(0.461,0.998)--cycle;
\fill[cell4](0.984,0.627)..controls(1.003,0.553)and(1.003,0.47)..(0.982,0.39)..controls(0.961,0.311)and(0.919,0.234)..(0.861,0.171)..controls(0.857,0.155)and(0.85,0.143)..(0.841,0.134)..controls(0.906,0.195)and(0.955,0.274)..(0.98,0.361)..controls(1.005,0.447)and(1.007,0.54)..(0.984,0.627)--cycle;
\fill[cell4](0.37,0.017)..controls(0.374,0.016)and(0.377,0.015)..(0.381,0.014)..controls(0.365,0.018)and(0.35,0.023)..(0.336,0.028)..controls(0.347,0.024)and(0.359,0.02)..(0.37,0.017)--cycle;
\fill[cell4](0.301,0.917)..controls(0.248,0.864)and(0.203,0.785)..(0.174,0.692)..controls(0.145,0.599)and(0.135,0.493)..(0.144,0.394)..controls(0.062,0.363)and(0.024,0.346)..(0.025,0.342)..controls(-0.009,0.444)and(-0.008,0.557)..(0.026,0.659)..controls(0.06,0.761)and(0.127,0.851)..(0.216,0.912)..controls(0.235,0.925)and(0.264,0.927)..(0.301,0.917)--cycle;
\fill[cell4](0.861,0.171)..controls(0.798,0.102)and(0.716,0.051)..(0.629,0.025)..controls(0.543,-0.002)and(0.452,-0.004)..(0.37,0.017)..controls(0.316,0.034)and(0.263,0.075)..(0.223,0.141)..controls(0.182,0.207)and(0.153,0.296)..(0.144,0.394)..controls(0.207,0.419)and(0.296,0.451)..(0.394,0.485)..controls(0.493,0.519)and(0.599,0.555)..(0.69,0.584)..controls(0.753,0.5)and(0.803,0.413)..(0.832,0.34)..controls(0.862,0.267)and(0.87,0.208)..(0.861,0.171)--cycle;
\fill[black](0.972,0.664)..controls(0.972,0.667)and(0.949,0.662)..(0.9,0.649)..controls(0.853,0.635)and(0.779,0.613)..(0.69,0.584)..controls(0.753,0.5)and(0.803,0.413)..(0.832,0.34)..controls(0.862,0.267)and(0.87,0.208)..(0.861,0.171)..controls(0.919,0.234)and(0.961,0.311)..(0.982,0.39)..controls(1.003,0.47)and(1.003,0.553)..(0.984,0.627)..controls(0.98,0.64)and(0.977,0.652)..(0.972,0.664)--cycle;
\fill[black](0.861,0.171)..controls(0.798,0.102)and(0.716,0.051)..(0.629,0.025)..controls(0.543,-0.002)and(0.452,-0.004)..(0.37,0.017)..controls(0.374,0.016)and(0.377,0.015)..(0.381,0.014)..controls(0.461,-0.005)and(0.546,-0.005)..(0.626,0.016)..controls(0.706,0.037)and(0.78,0.078)..(0.841,0.134)..controls(0.85,0.143)and(0.857,0.155)..(0.861,0.171)--cycle;
\fill[black](0.144,0.394)..controls(0.153,0.296)and(0.182,0.207)..(0.223,0.141)..controls(0.263,0.075)and(0.316,0.034)..(0.37,0.017)..controls(0.359,0.02)and(0.347,0.024)..(0.336,0.028)..controls(0.264,0.053)and(0.198,0.095)..(0.144,0.149)..controls(0.09,0.203)and(0.049,0.27)..(0.025,0.342)..controls(0.024,0.346)and(0.062,0.363)..(0.144,0.394)--cycle;
\fill[black](0.69,0.584)..controls(0.599,0.555)and(0.493,0.519)..(0.394,0.485)..controls(0.296,0.451)and(0.207,0.419)..(0.144,0.394)..controls(0.135,0.493)and(0.145,0.599)..(0.174,0.692)..controls(0.203,0.785)and(0.248,0.864)..(0.301,0.917)..controls(0.35,0.903)and(0.414,0.865)..(0.484,0.806)..controls(0.554,0.748)and(0.628,0.668)..(0.69,0.584)--cycle;
\fill[black](0.461,0.998)..controls(0.407,0.994)and(0.351,0.967)..(0.301,0.917)..controls(0.264,0.927)and(0.235,0.925)..(0.216,0.912)..controls(0.288,0.961)and(0.373,0.992)..(0.461,0.998)--cycle;
\end{tikzpicture}
\end{minipage}\hfill
\begin{minipage}{0.31\linewidth}
\centering
\begin{tikzpicture}[x=\linewidth,y=\linewidth]
\useasboundingbox (0,0) rectangle (1,1);
\clip (0,0) rectangle (1,1);
\definecolor{cell5}{RGB}{255,140,140}
\fill[cell5](0.978,0.645)..controls(0.984,0.626)and(0.989,0.607)..(0.992,0.587)..controls(0.993,0.585)and(0.993,0.583)..(0.993,0.581)..controls(0.99,0.602)and(0.985,0.624)..(0.978,0.645)--cycle;
\fill[cell5](0.682,0.509)..controls(0.727,0.391)and(0.754,0.276)..(0.759,0.195)..controls(0.853,0.201)and(0.917,0.232)..(0.945,0.273)..controls(0.994,0.37)and(1.011,0.482)..(0.992,0.587)..controls(0.98,0.647)and(0.931,0.708)..(0.846,0.753)..controls(0.82,0.693)and(0.761,0.604)..(0.682,0.509)--cycle;
\fill[cell5](0.37,0.972)..controls(0.408,0.961)and(0.462,0.912)..(0.52,0.828)..controls(0.644,0.827)and(0.763,0.799)..(0.846,0.753)..controls(0.869,0.804)and(0.871,0.836)..(0.857,0.85)..controls(0.756,0.954)and(0.608,1.01)..(0.464,0.999)..controls(0.433,0.996)and(0.401,0.988)..(0.37,0.972)--cycle;
\fill[cell5](0.011,0.591)..controls(0.031,0.652)and(0.09,0.713)..(0.182,0.758)..controls(0.228,0.86)and(0.298,0.936)..(0.37,0.972)..controls(0.357,0.976)and(0.346,0.976)..(0.336,0.972)..controls(0.257,0.945)and(0.186,0.898)..(0.13,0.836)..controls(0.073,0.774)and(0.033,0.698)..(0.014,0.617)..controls(0.012,0.609)and(0.011,0.6)..(0.011,0.591)--cycle;
\fill[cell5](0.424,0.246)..controls(0.508,0.317)and(0.604,0.414)..(0.682,0.509)..controls(0.638,0.627)and(0.578,0.745)..(0.52,0.828)..controls(0.397,0.829)and(0.273,0.803)..(0.182,0.758)..controls(0.135,0.656)and(0.112,0.529)..(0.121,0.411)..controls(0.194,0.347)and(0.305,0.286)..(0.424,0.246)--cycle;
\fill[cell5](0.003,0.557)..controls(0.005,0.568)and(0.007,0.58)..(0.011,0.591)..controls(0.011,0.54)and(0.047,0.475)..(0.121,0.411)..controls(0.129,0.293)and(0.168,0.187)..(0.224,0.117)..controls(0.201,0.11)and(0.185,0.112)..(0.176,0.12)..controls(0.113,0.172)and(0.065,0.241)..(0.035,0.317)..controls(0.005,0.392)and(-0.006,0.476)..(0.003,0.557)--cycle;
\fill[cell5](0.759,0.195)..controls(0.764,0.113)and(0.748,0.065)..(0.72,0.051)..controls(0.719,0.05)and(0.718,0.05)..(0.717,0.049)..controls(0.607,-0.004)and(0.477,-0.015)..(0.36,0.02)..controls(0.311,0.034)and(0.263,0.067)..(0.224,0.117)..controls(0.267,0.129)and(0.339,0.174)..(0.424,0.246)..controls(0.542,0.205)and(0.665,0.187)..(0.759,0.195)--cycle;
\fill[cell5](0.72,0.051)..controls(0.72,0.051)and(0.72,0.051)..(0.72,0.051)..controls(0.815,0.097)and(0.896,0.177)..(0.945,0.273)..controls(0.948,0.278)and(0.95,0.283)..(0.953,0.287)..controls(0.905,0.185)and(0.821,0.1)..(0.72,0.051)--cycle;
\fill[black](0.992,0.587)..controls(1.011,0.482)and(0.994,0.37)..(0.945,0.273)..controls(0.948,0.278)and(0.95,0.283)..(0.953,0.287)..controls(0.995,0.378)and(1.01,0.482)..(0.993,0.581)..controls(0.993,0.583)and(0.993,0.585)..(0.992,0.587)--cycle;
\fill[black](0.846,0.753)..controls(0.931,0.708)and(0.98,0.647)..(0.992,0.587)..controls(0.989,0.607)and(0.984,0.626)..(0.978,0.645)..controls(0.955,0.722)and(0.913,0.793)..(0.857,0.85)..controls(0.871,0.836)and(0.869,0.804)..(0.846,0.753)--cycle;
\fill[black](0.52,0.828)..controls(0.578,0.745)and(0.638,0.627)..(0.682,0.509)..controls(0.761,0.604)and(0.82,0.693)..(0.846,0.753)..controls(0.763,0.799)and(0.644,0.827)..(0.52,0.828)--cycle;
\fill[black](0.182,0.758)..controls(0.273,0.803)and(0.397,0.829)..(0.52,0.828)..controls(0.462,0.912)and(0.408,0.961)..(0.37,0.972)..controls(0.298,0.936)and(0.228,0.86)..(0.182,0.758)--cycle;
\fill[black](0.336,0.972)..controls(0.346,0.976)and(0.357,0.976)..(0.37,0.972)..controls(0.401,0.988)and(0.433,0.996)..(0.464,0.999)..controls(0.42,0.996)and(0.377,0.987)..(0.336,0.972)--cycle;
\fill[black](0.121,0.411)..controls(0.112,0.529)and(0.135,0.656)..(0.182,0.758)..controls(0.09,0.713)and(0.031,0.652)..(0.011,0.591)..controls(0.011,0.54)and(0.047,0.475)..(0.121,0.411)--cycle;
\fill[black](0.224,0.117)..controls(0.168,0.187)and(0.129,0.293)..(0.121,0.411)..controls(0.194,0.347)and(0.305,0.286)..(0.424,0.246)..controls(0.339,0.174)and(0.267,0.129)..(0.224,0.117)--cycle;
\fill[black](0.014,0.617)..controls(0.012,0.609)and(0.011,0.6)..(0.011,0.591)..controls(0.007,0.58)and(0.005,0.568)..(0.003,0.557)..controls(0.006,0.577)and(0.009,0.597)..(0.014,0.617)--cycle;
\fill[black](0.176,0.12)..controls(0.185,0.112)and(0.201,0.11)..(0.224,0.117)..controls(0.263,0.067)and(0.311,0.034)..(0.36,0.02)..controls(0.292,0.04)and(0.229,0.074)..(0.176,0.12)--cycle;
\fill[black](0.72,0.051)..controls(0.72,0.051)and(0.72,0.051)..(0.72,0.051)..controls(0.719,0.05)and(0.718,0.05)..(0.717,0.049)..controls(0.718,0.05)and(0.719,0.05)..(0.72,0.051)--cycle;
\fill[black](0.945,0.273)..controls(0.896,0.177)and(0.815,0.097)..(0.72,0.051)..controls(0.748,0.065)and(0.764,0.113)..(0.759,0.195)..controls(0.853,0.201)and(0.917,0.232)..(0.945,0.273)--cycle;
\fill[black](0.424,0.246)..controls(0.508,0.317)and(0.604,0.414)..(0.682,0.509)..controls(0.727,0.391)and(0.754,0.276)..(0.759,0.195)..controls(0.665,0.187)and(0.542,0.205)..(0.424,0.246)--cycle;
\end{tikzpicture}
\end{minipage}\hfill
\begin{minipage}{0.31\linewidth}
\centering
\begin{tikzpicture}[x=\linewidth,y=\linewidth]
\useasboundingbox (0,0) rectangle (1,1);
\clip (0,0) rectangle (1,1);
\definecolor{cell4}{RGB}{255,205,95}
\definecolor{cell5}{RGB}{255,140,140}
\definecolor{cell6}{RGB}{135,215,135}
\fill[cell4](0.38,0.756)..controls(0.474,0.748)and(0.578,0.726)..(0.671,0.694)..controls(0.681,0.744)and(0.687,0.788)..(0.691,0.826)..controls(0.652,0.863)and(0.611,0.894)..(0.571,0.917)..controls(0.511,0.885)and(0.443,0.829)..(0.38,0.756)--cycle;
\fill[cell5](0.559,0.31)..controls(0.633,0.277)and(0.702,0.253)..(0.758,0.238)..controls(0.822,0.28)and(0.877,0.33)..(0.916,0.381)..controls(0.91,0.454)and(0.886,0.536)..(0.846,0.615)..controls(0.797,0.644)and(0.737,0.671)..(0.671,0.694)..controls(0.648,0.576)and(0.606,0.431)..(0.559,0.31)--cycle;
\fill[cell4](0.116,0.591)..controls(0.142,0.565)and(0.176,0.537)..(0.215,0.508)..controls(0.257,0.595)and(0.316,0.683)..(0.38,0.756)..controls(0.285,0.764)and(0.201,0.757)..(0.14,0.74)..controls(0.127,0.695)and(0.118,0.644)..(0.116,0.591)--cycle;
\fill[cell5](0.268,0.12)..controls(0.329,0.109)and(0.4,0.11)..(0.474,0.124)..controls(0.501,0.172)and(0.531,0.236)..(0.559,0.31)..controls(0.439,0.364)and(0.309,0.439)..(0.215,0.508)..controls(0.185,0.446)and(0.163,0.386)..(0.151,0.332)..controls(0.178,0.25)and(0.22,0.176)..(0.268,0.12)--cycle;
\fill[cell6](0.841,0.135)..controls(0.865,0.157)and(0.885,0.187)..(0.898,0.224)..controls(0.87,0.217)and(0.822,0.221)..(0.758,0.238)..controls(0.673,0.182)and(0.572,0.142)..(0.474,0.124)..controls(0.444,0.069)and(0.417,0.035)..(0.396,0.022)..controls(0.43,0.008)and(0.463,0.000)..(0.496,0)..controls(0.622,-0.002)and(0.749,0.048)..(0.841,0.135)--cycle;
\fill[cell5](0.372,0.017)..controls(0.379,0.015)and(0.387,0.017)..(0.396,0.022)..controls(0.352,0.042)and(0.308,0.075)..(0.268,0.12)..controls(0.222,0.128)and(0.183,0.143)..(0.15,0.163)..controls(0.158,0.141)and(0.17,0.124)..(0.184,0.113)..controls(0.239,0.068)and(0.304,0.035)..(0.372,0.017)--cycle;
\fill[cell6](0.078,0.232)..controls(0.095,0.206)and(0.119,0.182)..(0.15,0.163)..controls(0.136,0.203)and(0.135,0.262)..(0.151,0.332)..controls(0.124,0.415)and(0.112,0.506)..(0.116,0.591)..controls(0.073,0.632)and(0.051,0.666)..(0.046,0.69)..controls(0.036,0.679)and(0.029,0.668)..(0.025,0.656)..controls(-0.022,0.516)and(-0.002,0.356)..(0.078,0.232)--cycle;
\fill[cell5](0.048,0.714)..controls(0.045,0.707)and(0.044,0.699)..(0.046,0.69)..controls(0.066,0.711)and(0.097,0.728)..(0.14,0.74)..controls(0.173,0.85)and(0.237,0.928)..(0.309,0.961)..controls(0.302,0.959)and(0.297,0.957)..(0.291,0.954)..controls(0.185,0.906)and(0.097,0.819)..(0.048,0.714)--cycle;
\fill[cell5](0.916,0.381)..controls(0.945,0.42)and(0.966,0.459)..(0.978,0.497)..controls(0.993,0.472)and(0.998,0.449)..(0.995,0.428)..controls(0.985,0.363)and(0.963,0.299)..(0.928,0.242)..controls(0.923,0.234)and(0.913,0.227)..(0.898,0.224)..controls(0.914,0.268)and(0.921,0.321)..(0.916,0.381)--cycle;
\fill[cell6](0.691,0.826)..controls(0.696,0.887)and(0.693,0.93)..(0.683,0.952)..controls(0.704,0.952)and(0.722,0.948)..(0.737,0.94)..controls(0.859,0.875)and(0.951,0.758)..(0.984,0.624)..controls(0.994,0.586)and(0.993,0.543)..(0.978,0.497)..controls(0.956,0.535)and(0.911,0.576)..(0.846,0.615)..controls(0.807,0.693)and(0.752,0.767)..(0.691,0.826)--cycle;
\fill[cell5](0.309,0.961)..controls(0.315,0.964)and(0.322,0.967)..(0.328,0.97)..controls(0.436,1.009)and(0.558,1.01)..(0.666,0.972)..controls(0.673,0.969)and(0.679,0.963)..(0.683,0.952)..controls(0.652,0.952)and(0.613,0.94)..(0.571,0.917)..controls(0.473,0.974)and(0.378,0.986)..(0.309,0.961)--cycle;
\fill[black](0.671,0.694)..controls(0.737,0.671)and(0.797,0.644)..(0.846,0.615)..controls(0.807,0.693)and(0.752,0.767)..(0.691,0.826)..controls(0.687,0.788)and(0.681,0.744)..(0.671,0.694)--cycle;
\fill[black](0.758,0.238)..controls(0.822,0.221)and(0.87,0.217)..(0.898,0.224)..controls(0.914,0.268)and(0.921,0.321)..(0.916,0.381)..controls(0.877,0.33)and(0.822,0.28)..(0.758,0.238)--cycle;
\fill[black](0.14,0.74)..controls(0.201,0.757)and(0.285,0.764)..(0.38,0.756)..controls(0.443,0.829)and(0.511,0.885)..(0.571,0.917)..controls(0.473,0.974)and(0.378,0.986)..(0.309,0.961)..controls(0.237,0.928)and(0.173,0.85)..(0.14,0.74)--cycle;
\fill[black](0.474,0.124)..controls(0.572,0.142)and(0.673,0.182)..(0.758,0.238)..controls(0.702,0.253)and(0.633,0.277)..(0.559,0.31)..controls(0.531,0.236)and(0.501,0.172)..(0.474,0.124)--cycle;
\fill[black](0.215,0.508)..controls(0.309,0.439)and(0.439,0.364)..(0.559,0.31)..controls(0.606,0.431)and(0.648,0.576)..(0.671,0.694)..controls(0.578,0.726)and(0.474,0.748)..(0.38,0.756)..controls(0.316,0.683)and(0.257,0.595)..(0.215,0.508)--cycle;
\fill[black](0.151,0.332)..controls(0.163,0.386)and(0.185,0.446)..(0.215,0.508)..controls(0.176,0.537)and(0.142,0.565)..(0.116,0.591)..controls(0.112,0.506)and(0.124,0.415)..(0.151,0.332)--cycle;
\fill[black](0.046,0.69)..controls(0.051,0.666)and(0.073,0.632)..(0.116,0.591)..controls(0.118,0.644)and(0.127,0.695)..(0.14,0.74)..controls(0.097,0.728)and(0.066,0.711)..(0.046,0.69)--cycle;
\fill[black](0.396,0.022)..controls(0.417,0.035)and(0.444,0.069)..(0.474,0.124)..controls(0.4,0.11)and(0.329,0.109)..(0.268,0.12)..controls(0.308,0.075)and(0.352,0.042)..(0.396,0.022)--cycle;
\fill[black](0.15,0.163)..controls(0.183,0.143)and(0.222,0.128)..(0.268,0.12)..controls(0.22,0.176)and(0.178,0.25)..(0.151,0.332)..controls(0.135,0.262)and(0.136,0.203)..(0.15,0.163)--cycle;
\fill[black](0.496,0)..controls(0.463,0.000)and(0.43,0.008)..(0.396,0.022)..controls(0.387,0.017)and(0.379,0.015)..(0.372,0.017)..controls(0.412,0.006)and(0.454,0.000)..(0.496,0)--cycle;
\fill[black](0.184,0.113)..controls(0.17,0.124)and(0.158,0.141)..(0.15,0.163)..controls(0.119,0.182)and(0.095,0.206)..(0.078,0.232)..controls(0.106,0.187)and(0.142,0.146)..(0.184,0.113)--cycle;
\fill[black](0.025,0.656)..controls(0.029,0.668)and(0.036,0.679)..(0.046,0.69)..controls(0.044,0.699)and(0.045,0.707)..(0.048,0.714)..controls(0.039,0.695)and(0.031,0.676)..(0.025,0.656)--cycle;
\fill[black](0.309,0.961)..controls(0.315,0.964)and(0.322,0.967)..(0.328,0.97)..controls(0.316,0.965)and(0.303,0.96)..(0.291,0.954)..controls(0.297,0.957)and(0.302,0.959)..(0.309,0.961)--cycle;
\fill[black](0.683,0.952)..controls(0.679,0.963)and(0.673,0.969)..(0.666,0.972)..controls(0.691,0.963)and(0.715,0.952)..(0.737,0.94)..controls(0.722,0.948)and(0.704,0.952)..(0.683,0.952)--cycle;
\fill[black](0.978,0.497)..controls(0.993,0.543)and(0.994,0.586)..(0.984,0.624)..controls(1.001,0.561)and(1.004,0.494)..(0.995,0.428)..controls(0.998,0.449)and(0.993,0.472)..(0.978,0.497)--cycle;
\fill[black](0.898,0.224)..controls(0.913,0.227)and(0.923,0.234)..(0.928,0.242)..controls(0.905,0.202)and(0.875,0.166)..(0.841,0.135)..controls(0.865,0.157)and(0.885,0.187)..(0.898,0.224)--cycle;
\fill[black](0.846,0.615)..controls(0.911,0.576)and(0.956,0.535)..(0.978,0.497)..controls(0.966,0.459)and(0.945,0.42)..(0.916,0.381)..controls(0.91,0.454)and(0.886,0.536)..(0.846,0.615)--cycle;
\fill[black](0.571,0.917)..controls(0.613,0.94)and(0.652,0.952)..(0.683,0.952)..controls(0.693,0.93)and(0.696,0.887)..(0.691,0.826)..controls(0.652,0.863)and(0.611,0.894)..(0.571,0.917)--cycle;
\end{tikzpicture}
\end{minipage}
\caption{Arrangements for \(n=3\), \(n=5\), and \(n=7\).  The first two
realize the polyhedral groups \(S_4\) and \(A_5\); \(n=7\) gives
\(D_2\), one of only two non-cyclic dihedral \(G_P\) in the enumeration
(the other is \(D_3\) at \(n=17\)).}
\label{fig:symmetry-n3-n5-n7}
\end{figure}

\begin{figure}[H]
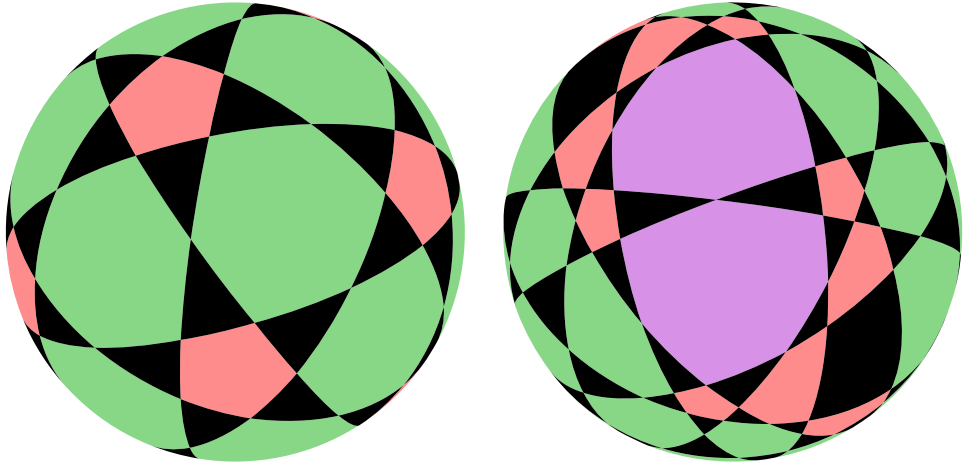

\centering
\begin{minipage}{0.48\linewidth}
\centering
\input{figures/9}
\end{minipage}\hfill
\begin{minipage}{0.48\linewidth}
\centering
\input{figures/13_2tcv981y0y3sl}
\end{minipage}
\caption{Arrangements for \(n=9\) and \(n=13\) \((C_2)\).  At \(n=9\) the
polyhedral group \(A_5\) appears a second time.}
\label{fig:symmetry-n9-n13}
\end{figure}

\begin{figure}[H]
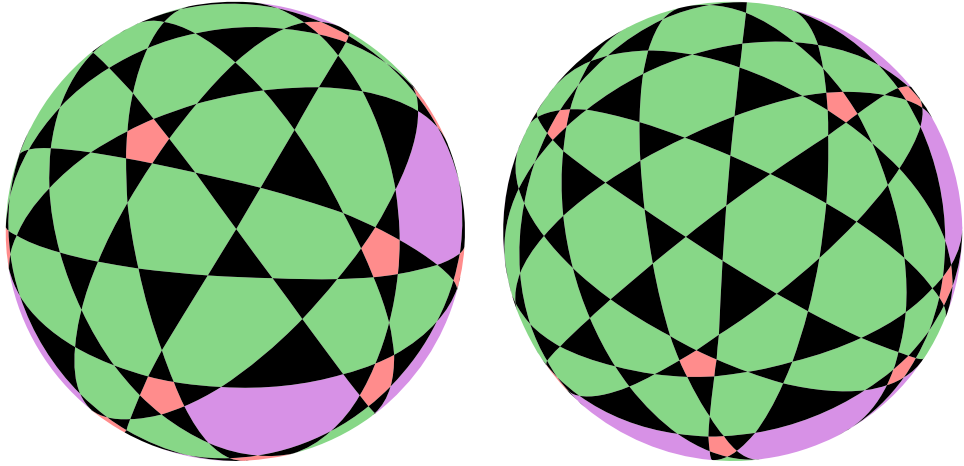

\centering
\begin{minipage}{0.48\linewidth}
\centering
\input{figures/15}
\end{minipage}\hfill
\begin{minipage}{0.48\linewidth}
\centering
\input{figures/17_35u8pww448fx5}
\end{minipage}
\caption{Arrangements for \(n=15\) \((C_5)\) and \(n=17\) \((D_3)\), the
second of the two non-cyclic dihedral \(G_P\).  The \(n=15\) arrangement
illustrates the Euclidean stabilizers of
\cref{subsec:symmetries-stabilizers}: on the sphere it has \(n+1=16\)
lines and \(C_5\) symmetry about the polar axis.  One line is the
equator; taking it as the line at infinity gives a \(C_5\)-symmetric
affine chart --- a single \(E\)-class with \(H_E=C_5\) and multiplicity
\(m(E)=1\).  The other \(15\) lines fall by proximity to the pole into
three clusters of five lines.
Cutting along any line of a cluster gives
the same asymmetric \(E\)-class \((H_E=C_1)\), encountered
five times, so \(m(E)=5\).
Three clusters correspond to three different \(E\)-classes.
The four \(E\)-classes exhaust the \(16\) affine
charts: \(1\cdot1+3\cdot5=n+1\).}
\label{fig:symmetry-n15-n17-a}
\end{figure}

\begin{figure}[H]
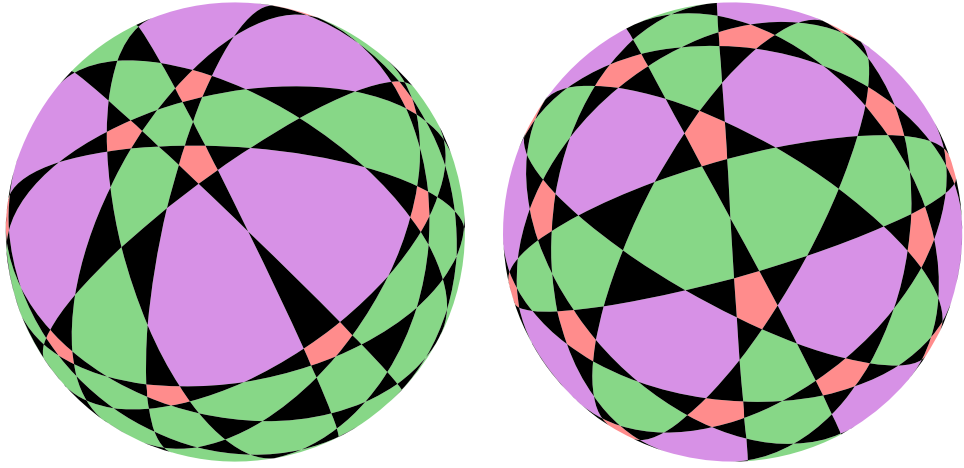

\centering
\begin{minipage}{0.48\linewidth}
\centering
\input{figures/17_3cnjb4fdm3aip}
\end{minipage}\hfill
\begin{minipage}{0.48\linewidth}
\centering
\input{figures/17_ukrnlxg98ae9}
\end{minipage}
\caption{Arrangements for \(n=17\) with symmetries \(D_3\) and \(A_4\).
The first one can be obtained from the \(n=9\) arrangement 
(\cref{fig:n7-n9-samples}) by the doubling construction of
Forge--Ramírez Alfonsín~\cite{forge-ramirez-1998} or
Bartholdi--Blanc--Loisel~\cite{bartholdi-blanc-loisel-2007}; in the
projective plane the two constructions give the same result.
The tetrahedral group \(A_4\) appears here for the first time.  On the
sphere the \(A_4\) arrangement shows the same local picture from six
directions, each view carrying \(D_2\) symmetry; just as a
tetrahedron has six edges, each edge-on view having \(D_2\)
symmetry (\(6\cdot|D_2|=24=2|A_4|\)).
}
\label{fig:symmetry-n17-b}
\end{figure}

\begin{figure}[H]
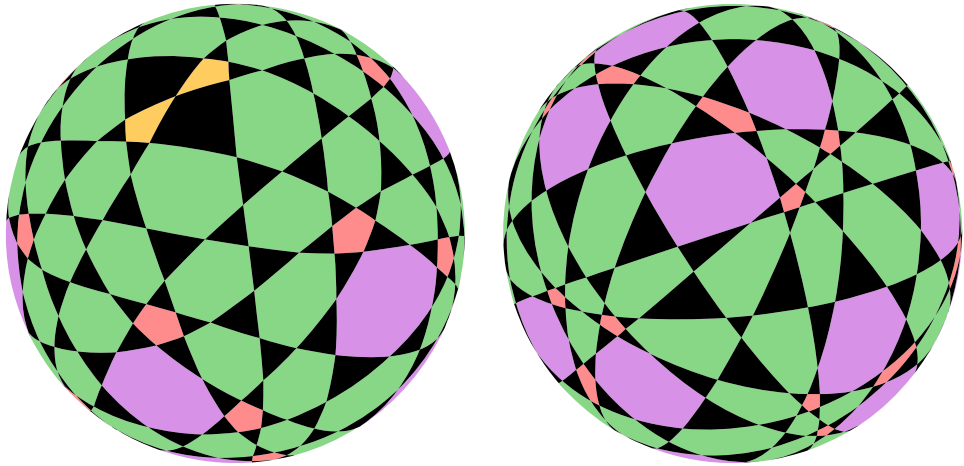

\centering
\begin{minipage}{0.48\linewidth}
\centering
\input{figures/19_3hogzjbmfysij}
\end{minipage}\hfill
\begin{minipage}{0.48\linewidth}
\centering
\input{figures/21_2wdhca8g2cjbp}
\end{minipage}
\caption{Arrangements for \(n=19\) with \(C_2\) symmetry
(its affine charts give \(E\)-classes with \(H_E=D_1\) and \(H_E=C_1\))
and for \(n=21\) with \(A_4\).  No larger arrangement in the
enumeration (\(n=23\), \(25\), or \(27\)) has a polyhedral \(G_P\).}
\label{fig:symmetry-n19-3hogzjbmfysij}
\end{figure}

\begin{figure}[H]
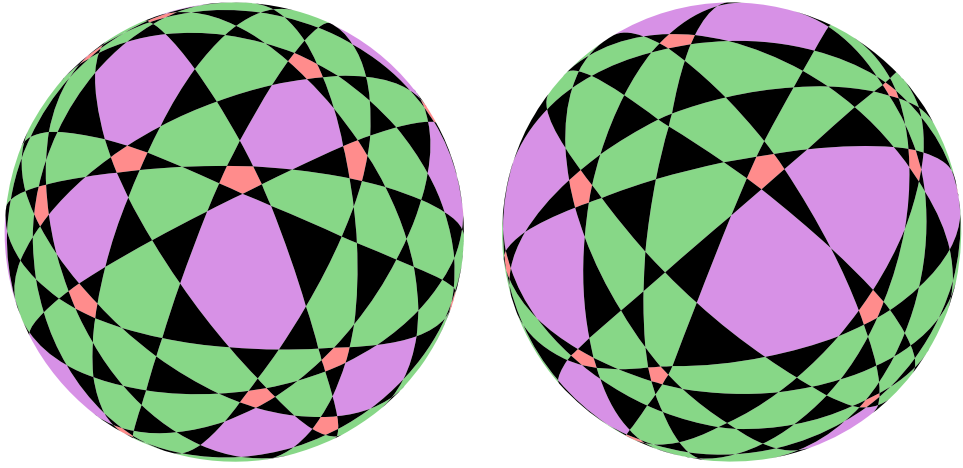

\centering
\begin{minipage}{0.48\linewidth}
\centering
\input{figures/21_2wyxe08jij93p}
\end{minipage}\hfill
\begin{minipage}{0.48\linewidth}
\centering
\input{figures/21_3ip8x8zug58f5}
\end{minipage}
\caption{Arrangements for \(n=21\) with symmetries \(C_2\) and \(C_3\).
\(n=21\) carries four distinct \(G_P\) (\(C_1\), \(C_2\),
\(C_3\), \(A_4\)), and the
\(A_4\) profile has two \(P\)-classes, each splitting into three \(H_E\)
classes (\(C_3\), \(D_1\), \(C_1\)) --- one projective object cut into
three inequivalent affine charts.}
\label{fig:symmetry-n21-c2-c3}
\end{figure}

\begin{figure}[H]
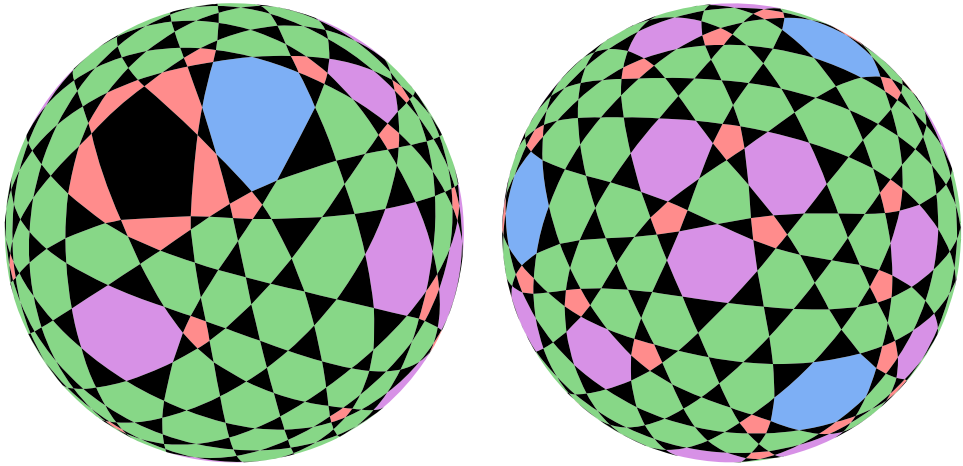

\centering
\begin{minipage}{0.48\linewidth}
\centering
\input{figures/25-b5}
\end{minipage}\hfill
\begin{minipage}{0.48\linewidth}
\centering
\input{figures/27}
\end{minipage}
\caption{Arrangements for \(n=25\) with one pentagonal defect \((C_1)\)
and for \(n=27\) \((C_3)\).  The entire single-pentagon family at
\(n=25\) is asymmetric: all \(2324\) \(P\)-classes have trivial \(G_P\).
One might expect a \(C_5\) rotation centered on the pentagonal defect,
but no such arrangement appears.  This is the only family in the
enumeration with no nontrivial symmetry; every other \(n\), perfect or
defective, contributes at least one nontrivial \(G_P\).}
\label{fig:symmetry-n25-b5-n27}
\end{figure}

\endgroup

\section{Additional prunings}
\label{app:prunings}

\subsection{Forced descent of a suffix-maximal wire}
\label{app:forced-descent}

This additional pruning is used in the perfect search
(\cref{sec:perfect-search}) and partially in the 2-defective search
(\cref{sec:2-defective-search}).

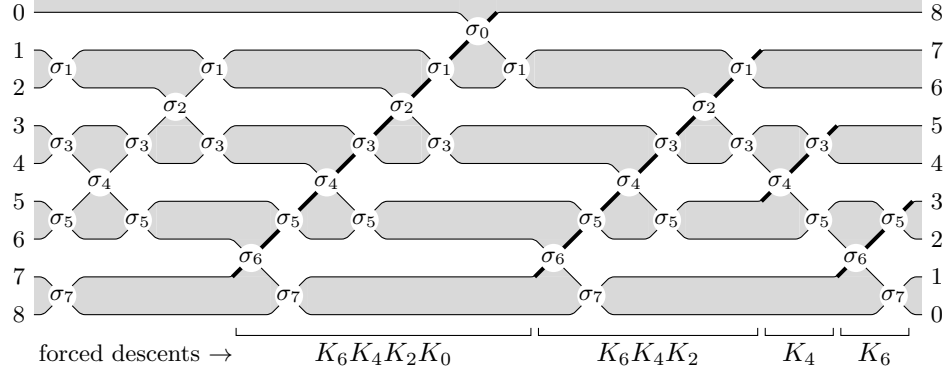
\begin{figure}[ht]
\centering
\begin{tikzpicture}[x=0.5cm,y=0.5cm]
  \PseudolineWiringDiagram[fill]{9}{{7,5,3,1},4,{3,5},2,{1,3},6,{5,7},4,{3,5},2,{1,3},0,1,6,{5,7},4,{3,5},2,{1,3},4,{3,5},6,{5,7}}{wirefrom/8/6,wirefrom/7/14,wirefrom/5/20,wirefrom/3/22}
  \foreach \xl/\xr/\lab in {%
      5.35/13.15/{K_6K_4K_2K_0},
      13.35/19.15/{K_6K_4K_2},
      19.35/21.15/{K_4},
      21.35/23.15/{K_6}}{%
    \draw (\xl,-0.35) -- (\xl,-0.55) -- (\xr,-0.55) -- (\xr,-0.35);
    \node at ({(\xl+\xr)/2},-1.05) {$\lab$};
  }
  \node at (2.7,-1.00) {\small forced descents $\rightarrow$};
\end{tikzpicture}%
\captionsetup{width=0.88\linewidth}
\caption{A perfect arrangement for $n=9$ with forced descents.}
\label{fig:forced-descent}
\end{figure}

\begin{definition}[Suffix-maximal wire]
\label{def:suffix-maximal}
In a permutation $a_0,\ldots,a_{n-1}$ of the wires, a wire $w=a_i$ is
\emph{suffix-maximal} if $a_i=\max\{a_i,a_{i+1},\ldots,a_{n-1}\}$.
\end{definition}

\begin{lemma}[Forced descent of a suffix-maximal wire]
\label{lem:forced-descent}
Let a canonical reduced word in composite form~\eqref{eq:factorization} encode a
perfect arrangement, and let $K_g$ ($g\ge 2$) be a composite generator
acting on the permutation $a_0,\ldots,a_{n-1}$ reached just before it.
$K_g$ moves the wire $w=a_{g+1}$ from slot $g+1$ to slot $g-1$.
If $w$ is suffix-maximal in the permutation $a_0,\ldots,a_{n-1}$, then $w$
moves monotonically to its final slot $s=n-1-w$, and $s$ is either $0$
or odd.  The composite generators starting from $K_g$ are forced to be
\[
  K_{g}\,K_{g-2}\,\cdots\,K_f,
  \qquad
  f=\begin{cases}
    s+1, & s\text{ odd},\\
    0, & s=0.
  \end{cases}
\]
No other composite generator is admissible until $w$ reaches $s$.
\end{lemma}

\begin{proof}
We track the wire $w=a_{g+1}$ through its descent: it starts at slot $w$ in the
identity permutation, occupies slot $g+1$ just before $K_g$, and ends at slot
$s=n-1-w$ in the final permutation, the full reversal.

\emph{Step 1: monotone descent.}
Since $w$ is suffix-maximal, the wires $a_{g+2},\ldots,a_{n-1}$ after it are
smaller, so they started at smaller slot indices and have already crossed it.
By reducedness, none crosses it again, so $w$'s slot index does not increase.
As long as $w$
has not reached $s$, its next crossing therefore exchanges it with the smaller
wire at the adjacent lower slot, which afterward occupies a higher slot index
than $w$; hence $w$ is again suffix-maximal, and the same argument applies at
the next slot.  By induction the index falls monotonically from $g+1$ to $s$,
where $w$ ends.

\emph{Step 2: consecutive crossings.}
While $w$ occupies slot $k$, its next crossing is $\sigma_{k-1}$, moving it to
slot $k-1$.  Both $\sigma_{k-1}$ and $\sigma_k$ act on a pair touching slot $k$
and would cross $w$, so every letter $\sigma_a$ occurring between $w$'s arrival
in slot $k$ and that $\sigma_{k-1}$ has $a>k$ or $a<k-1$.  Write
$C_k=\sigma_g\sigma_{g-1}\cdots\sigma_k$ for the descent chain down to
$\sigma_k$ (base case $C_g=\sigma_g$), and suppose $C_k$ is already consecutive.
The two kinds of intervening letter differ in index by at least $3$, hence
commute; we sort them as $\sigma_{b_1}\cdots\sigma_{b_p}$ ($b_i<k-1$) followed
by $\sigma_{c_1}\cdots\sigma_{c_q}$ ($c_i>k$).  Each $\sigma_{b_i}$ commutes
with all of $C_k$ (indices $\ge k$) and moves left past it; each
$\sigma_{c_i}$ commutes with $\sigma_{k-1}$ and moves right past it.  This
removes every intervening letter and extends $C_k$ to $C_{k-1}$; iterating down to
$C_{s}=\sigma_g\cdots\sigma_{s}$ makes $w$'s crossings consecutive
(\cref{fig:forced-descent}).

\emph{Step 3: the chain in composite form.}
Bring the word obtained in Step 2 to the form~\eqref{eq:factorization} by
\cref{lem:composite-necessity}.  An even $s>0$ is impossible:
then the final crossing of $w$ would be the white letter $\sigma_s$, but its
packaging into $K_s=\sigma_s\sigma_{s-1}\sigma_{s+1}$
(\cref{lem:composite-necessity}) requires a later $\sigma_{s-1}$, crossing
$w$ again after it has reached slot $s$.  For $s=0$, $K_0=\sigma_0\sigma_1$
moves the wire $n-1$ to its final slot $0$.  Hence $s$ is odd, or $s=0$.

The white generators in $C_s$ are then
$\sigma_g,\sigma_{g-2},\ldots,\sigma_f$, where $f=s+1$ when $s$ is odd and
$f=0$ when $s=0$.  They are consecutive among the white generators of the word,
and the composite indices follow their order, so the composite form contains the
contiguous chain $K_gK_{g-2}\cdots K_f$.  If $f=g$, this chain has only its head
and the forced continuation after $K_g$ is empty.

\emph{Step 4: the chain in canonical form.}
Reduce the composite form to the canonical word by the distant
commutations~\eqref{eq:distant-commute}, which recover the canonical form
uniquely (\cref{lem:composite-distant-commutation,lem:composite-canonical-form});
it is therefore the given word.  Each replacement $K_aK_b\to K_bK_a$
($a\ge b+4$), resolving the forbidden pair, moves the larger index to the
right and the smaller to the left.  The chain's consecutive indices differ by
$2<4$, so its internal pairs never commute and are never swapped.

Because the canonical form is unique, it suffices to exhibit one sequence of
commutations that removes every forbidden pair: we resolve forbidden pairs among
the foreign generators in any order, and a forbidden pair between a foreign
generator and a chain generator by commuting that foreign generator through the
entire chain in a single run, without disturbing the other pairs.
This resolution terminates, since each swap strictly decreases the
distant-inversion count used in \cref{lem:composite-canonical-form}.

Since the chain is contiguous in the composite form obtained above, every other
composite generator starts outside it, left of the head $K_g$ or right of the tail $K_f$.
Take a foreign composite generator $K_h$ next to the head: $K_hK_g\cdots K_f$.
If $h\le g+2$ the pair $K_hK_g$ is canonical and $K_h$ stays left; if $h\ge g+4$
then $h$ exceeds every chain index by at least $4$, so $K_h$ is swapped past the
whole chain, beyond its right end.  Symmetrically, $K_h$ next to the tail
$K_g\cdots K_fK_h$ stays right if $h\ge f-2$ and passes beyond its left end if
$h\le f-4$.  No foreign composite generator ever stops between two links, so the
chain $K_gK_{g-2}\cdots K_f$ stays contiguous in the canonical word, that is, no
other composite generator is admissible until $w$ reaches $s$.
\end{proof}

\begin{remark}
The forced run is a \emph{descent} because the canonical-form rule of
\cref{lem:composite-canonical-form} is asymmetric: consecutive composite
indices may decrease by at most~$2$ but may increase freely.  The mirror
notion of a prefix-minimal wire forces no analogous ascent: an ascending
chain would be broken up by foreign generators stopping between its links
in Step~4.
\end{remark}

\begin{remark}[Forced-descent pruning in the 2-defective search]
\label{rem:2-defective-forced-descent}
The forced-descent pruning of \cref{lem:forced-descent} is used in the
2-defective search too, but only among the ordinary generators
that follow the special one.  By that point the special generator has
fixed the defect it creates, and the generators after it are applied
as in the perfect search.  Otherwise, the pruning would exclude
one of the rhombus variants \eqref{eq:special-commutation} together with
its defect jump, which the search must keep.

Turning this into a proof would mean re-deriving Steps~3--4 in the proof
of \cref{lem:forced-descent} for a word of the form
\eqref{eq:special-word-form}: packaging the defect-adjacent generators
by the composition rules of \cref{lem:special-factorization} rather than
\cref{lem:composite-necessity}, and re-establishing, on the ordinary
tail, that its composite form is determined up to distant commutation
(\cref{lem:composite-distant-commutation}) once the position and variant
of the special generator are held fixed.  We omit it.
The pruning only removes search branches, so it cannot cause
an invalid word to be emitted.  For the exhaustive enumerations reported in
\cref{subsec:exhaustive-enumerations} we have run the search with the
pruning disabled and enabled and obtained identical output streams,
so empirically no coverage is lost.
\end{remark}

\subsection{Wall debt prunings}
\label{subsec:crossing-debt}

The forced descent acts locally on the candidate $K_g$, along one descending
wire.  The next pruning gives a global constraint on the multiset of remaining
composite generators.

\begin{wrapfigure}[5]{r}{0.30\linewidth}
\vspace{-1.25\baselineskip}
\begin{tikzpicture}[x=0.8cm,y=1cm,>=stealth,baseline=(current bounding box.north)]
  \foreach \wi/\dv in {0/1, 1/2, 2/2, 3/1}{%
    \pgfmathsetmacro{\wx}{\wi+0.5}%
    \draw[dashed,line width=0.3pt] (\wx,-0.15) -- (\wx,1.4);%
    \node[scale=0.8] at (\wx,1.53) {$\delta_{\wi}{=}\dv$};%
  }
  \foreach \s/\v in {0/1, 1/0, 2/3, 3/4, 4/2}{%
    \node[fill=white,draw,inner sep=3pt,font=\scriptsize] (cur) at (\s,1) {$\v$};%
  }
  \foreach \s/\v in {0/4, 1/3, 2/2, 3/1, 4/0}{
    \node[font=\scriptsize] at (\s,0) {$\v$};%
  }
  \foreach \xt/\xb in {2/1, 3/0, 4/2}{%
    \draw[dotted, shorten >=4pt, shorten <=6pt, ->] (\xt,1) -- (\xb,0);%
  }
  \foreach \xt/\xb in {0/3, 1/4}{%
    \draw[line width=0.3pt, ->, shorten >=4pt, shorten <=6pt] (\xt,1) -- (\xb,0);
  }
\end{tikzpicture}
\end{wrapfigure}

Call the boundary between slots $i$ and $i+1$ \emph{wall}~$i$.
The generator $\sigma_i$ is the only one that carries wires across this wall.
The pruning compares the number of wires that need to pass through the walls
with the number of remaining composite generators.

In a simple arrangement, every pair of wires crosses exactly once and
the final permutation is the reversal $a_i=n-1-i$; that is, wire $w$ finishes
in slot $s=n-1-w$.

\begin{definition}[Wall debt]
\label{def:wall-debt}
For a wall $i$ and the permutation $a_0,\ldots,a_{n-1}$ at the partial
word $W_\ell=\sigma_{g_1}\cdots\sigma_{g_\ell}$, the \emph{wall debt} at $i$ is
\[
  \delta_i \;=\; \#\{\, j\le i : i < n-1-a_j \,\}, \quad i=0,\ldots,n-2,
\]
the number of wires in slots $0,\ldots,i$ whose final slot $s = n-1-a_j$ exceeds $i$.
\end{definition}

\begin{lemma}[Wall-debt constraint for simple arrangements]
\label{lem:wall-debt-simple}
In a reduced word encoding a simple arrangement, the letters after the partial
word $W_\ell=\sigma_{g_1}\cdots\sigma_{g_\ell}$ include at least
$\delta_i$ occurrences of $\sigma_i$, for every wall $i=0,\ldots,n-2$.
\end{lemma}

\begin{proof}
Call slots $0,\ldots,i$ the \emph{block}.  Among all generators only
$\sigma_i$ changes which wires are in the block: $\sigma_i$ swaps the wires in
slots $i$ and $i+1$, moving one wire out of the block and one in, while every
$\sigma_a$ with $a<i$ permutes wires within the block and every $\sigma_a$
with $a>i$ permutes wires outside it.  At the end, the block must
hold the final wires $\{w:n-1-w\le i\}$.  The $\delta_i$ wires currently in it with
$n-1-a_j>i$ are not among the final wires and must leave the block.
Each $\sigma_i$ removes
exactly one wire from the block, and each of the $\delta_i$ leaving wires leaves
at least once, so at least $\delta_i$ occurrences of $\sigma_i$ remain.
\end{proof}

\begin{corollary}[Wall-debt constraint for perfect arrangements]
\label{cor:wall-debt-perfect}
Let the partial word be of the composite form~\eqref{eq:factorization}
with $k$ composite generators placed and $B=N_{\mathrm{perf}}-k$ remaining, and
let $x_g$ be the number of remaining generators equal to $K_g$.
Every $K_g$ ($g\ge2$) contributes one letter at each of the
walls $g-1,g,g+1$, and $K_0=\sigma_0\sigma_1$ at walls $0,1$.  The even wall
$g$ receives letters only from $K_g$, and an odd wall $j$ only from $K_{j-1}$
and $K_{j+1}$.  Applying \cref{lem:wall-debt-simple} at every wall gives
\[
  x_g\ge \delta_g\ (g\text{ even}),\qquad
  x_{j-1}+x_{j+1}\ge \delta_j\ (j\text{ odd}),\qquad
  \sum_g x_g=B .
\]
If no nonnegative integers $x_g$ satisfy these, the branch encodes no perfect
arrangement and is pruned.
\end{corollary}

\begin{remark}[Wall sealing]
\label{rem:wall-sealing}
\cref{lem:trailing-digon-pruning} implies a pruning without giving a way to
check it.  The wall debt provides one.  Once a pair of wires $(2m,2m+1)$ has
crossed at its trailing digon by $\sigma_{n-2-2m}$, the wires no longer move,
so the walls $n-3-2m,n-2-2m,n-1-2m$ admit no further crossing (for $m=0$ the
wall $n-1$ lies outside the range $0,\ldots,n-2$ and is omitted).
If the debt at such a \emph{sealed} wall is positive, the word cannot be
constructed and the branch is pruned.
Conversely, if $\delta_i=0$, slots $0,\ldots,i$ hold exactly the $i+1$ wires
that finish there, so $a_i>a_{i+1}$ and any other $\sigma_i$ violates reducedness.
Sealing with zero debt thus decides at $W_\ell$ what
\cref{lem:trailing-digon-pruning} states about the completed word.
\end{remark}

\cref{cor:wall-debt-perfect} generalizes to the 2-defective search, with
$K_g^{\pm}$ counted in $x_g$ by their white index $g$. However, both
implementations evaluate only the even-wall bound
$\sum_{g\text{ even}}\delta_g\le B$. In our runs the odd-wall constraints
were not tighter for the perfect search and pruned negligibly in the
2-defective search.

\section{The dihedral reduction of the \texorpdfstring{$O$}{O}-matrix operations}
\label{app:symmetries}

\Cref{sec:classification} introduces the operations \(\rho\), \(\mu\),
\(\pi_p\) on \(O\)-matrices descriptively. This appendix gives their explicit
row-level formulas and establishes the reduction underlying the classification
of \cref{sec:classification} and \cref{app:canonicalization}.

Unlike \(\rho\) and \(\mu\), the operation \(\pi_p\) relabels lines depending on
the matrix it acts on (see \eqref{eq:pi-formula}),
so an operation sequence \(g\in\mathcal{W}\)
(\cref{subsec:symmetries-stabilizers}) does not act on the labels by a
well-defined permutation \(\tau(g)\).
However, for any fixed \(O\)-matrix, \cref{lem:chart-reduction} shows that
every operation sequence factors as a projective rotation \(\pi_s\),
sending the chosen line \(s\) to infinity, followed by an element of the
group \(D_{2n}\).
This reduction is the basis for the symmetry conclusions of
\cref{sec:classification}.

\subsection{Row formulas}
\label{app:symmetries:row-formulas}

In what follows \(O\in\mathcal{O}_n\) has rows \(O_0,\ldots,O_{n-1}\), and
each row has length \(n-1\) indexed by \(j=0,\ldots,n-2\).

The formulas can be re-derived as follows.  Pass to the
projective closure (\cref{subsec:checkerboard}): each of the \(n+1\)
projective lines is a closed curve, and its crossings with the other lines
follow one another along it in a cyclic order that depends on no choice of
chart or labeling.  These cyclic orders are what the extension \(O^+\) of
\eqref{eq:projective-closure} below records; an \(O\)-matrix reads them off
in one affine chart.  The chart cuts each closed line at its crossing with
the line at infinity, turning the cyclic order into a row, and the labeling
choices of \cref{subsec:order-matrices} (which line is \(0\), an end of it,
and a direction around infinity) fix the direction in which each row is
read.  The operations change only this reading data: \(\mu\) reverses the
direction of the labeling, \(\rho\) moves the sweep start past one line
end, and \(\pi_p\) re-chooses which line is cut out as the line at
infinity.  The row formulas record where each row is cut, which rows
reverse, and how the lines are renamed.  They
were obtained, and can be rechecked, by relabeling a drawn arrangement and
reading off its new matrix; the adjacency arguments of
\cref{app:symmetries:adjacency} rest on the same invariance of the cyclic
orders.

The chart change behind \(\pi_p\) stays within pseudoline arrangements: a
pseudoline of the projective closure can be straightened by a homeomorphism
of the projective plane and sent to infinity by a projective
transformation, turning the remaining \(n\) lines into an affine pseudoline
arrangement in the new chart (for these standard facts
see~\cite{grunbaum-1972,felsner-goodman-2017}).

\emph{Reflection \(\mu\)} (restated from \cref{def:reflection}):
\begin{equation}
\label{eq:mu-formula}
  (\mu O)_{i,j} = (n-1) - (O_{n-1-i})_j .
\end{equation}

\emph{Euclidean rotation \(\rho\)} (\cref{def:euclidean-rotation}):
\begin{equation}
\label{eq:rho-formula}
  (\rho O)_{k,j} =
  \begin{cases}
    \bigl((O_{k+1})_j - 1\bigr)\bmod n, & 0\le k\le n-2,\\
    \bigl((O_0)_{n-2-j} - 1\bigr)\bmod n, & k=n-1.
  \end{cases}
\end{equation}
Thus \(\rho\) moves each row up one position and decrements its entries
mod \(n\); the wrapped row (former row \(0\)) is additionally reversed.

\emph{Projective rotation \(\pi_p\).}
For \(p\in\{0,\ldots,n-1\}\), we construct \(\pi_p O\) (\cref{def:projective-rotation}),
using an auxiliary projective closure on the extended label set
\(\{0,\ldots,n\}\). The label \(n\) denotes the line at infinity
(in \cref{subsec:order-matrix-operations} we wrote \(\infty\) for clarity).

\emph{Projective rotation: \(O\)-matrix extension.}
Form \(O^+\in\mathcal{O}_{n+1}\) by
\begin{equation}
\label{eq:projective-closure}
  (O^+)_i = (O_i,\,n)\quad\text{for }i\in\{0,\ldots,n-1\},\qquad
  (O^+)_n = (0,1,\ldots,n-1),
\end{equation}
where \((O_i,n)\) denotes the row \(O_i\) with \(n\) appended.
Membership \(O^+\in\mathcal{O}_{n+1}\) follows from adding the line at infinity
as a last wire to the wiring diagram of \(O\).

\emph{Projective rotation: relabeling.}
Write \(L_j=(O_p)_j\) for the entries of row \(p\) of \(O\), so that
\(L_0\), \(L_1\), \(\ldots\), \(L_{n-2}\) is a sequence of distinct
labels from \(\{0,\ldots,n-1\}\setminus\{p\}\).  Define
\(\tau\colon\{0,\ldots,n\}\setminus\{p\}\to\{0,\ldots,n-1\}\) by
\begin{equation}
\label{eq:pi-relabeling}
  \tau(L_j)=j\quad(0\le j\le n-2),\qquad \tau(n)=n-1.
\end{equation}
For each new label \(k\in\{0,\ldots,n-1\}\), the corresponding old label
is \(v(k)=\tau^{-1}(k)\): namely \(v(k)=L_k\) for \(k\le n-2\) and
\(v(n-1)=n\).

\emph{Projective rotation: entry formula.}
For \(k\in\{0,\ldots,n-1\}\), let \(q_k\) be the position of \(p\) in row
\((O^+)_{v(k)}\) (which exists and is unique because \(v(k)\neq p\)),
and for \(j\in\{0,\ldots,n-2\}\) let
\[
  c_k(j) = \begin{cases}
    (q_k+j+1)\bmod n, & v(k)>p,\\[2pt]
    (q_k-j-1)\bmod n, & v(k)<p.
  \end{cases}
\]
As \(j\) runs over \(0,\ldots,n-2\), the cyclic index \(c_k(j)\) sweeps,
in the cyclic direction fixed by the two cases above, the \(n-1\) positions
of row \((O^+)_{v(k)}\) other than \(q_k\) (the position holding \(p\)).
The branch \(v(k)>p\) covers the row at infinity \(v(k)=n\) as well.  Then
the resulting matrix \(\pi_p O\) lies in \(\mathcal{O}_n\) and is given by
\begin{equation}
\label{eq:pi-formula}
  (\pi_p O)_{k,j} = \tau\bigl((O^+)_{v(k),\,c_k(j)}\bigr),
  \qquad j=0,\ldots,n-2.
\end{equation}

\subsection{Dihedral relations and the action of \texorpdfstring{\(D_{2n}\)}{D2n}}
\label{app:symmetries:group-relations}

We verify the dihedral relations directly from
\cref{app:symmetries:row-formulas}.  Involutivity \(\mu^2=\mathrm{id}\)
is immediate from \eqref{eq:mu-formula}:
\((\mu^2 O)_{i,j}=(n-1)-(\mu O)_{n-1-i,j}=(n-1)-\bigl((n-1)-(O_i)_j\bigr)=(O_i)_j\).

\begin{lemma}[No palindromes in \(O_i\)]
\label{lem:no-palindrome}
No row of an \(O\)-matrix equals its reversal.
\end{lemma}

\begin{proof}
By \cref{def:order-matrix} its \(n-1\ge 2\) entries are pairwise distinct.
\end{proof}

\begin{lemma}[Order of \(\rho\)]
\label{lem:rho-order}
\(\rho^{2n}=\mathrm{id}\), and \(\rho^{n}\) reverses every row of an
\(O\)-matrix, leaving the row labels and entry values unchanged.
\end{lemma}

\begin{proof}
Track row \(i\) under iterations of \(\rho\) \eqref{eq:rho-formula}.
One application sends a row in position \(i\) to position \((i-1)\bmod n\)
and decrements each entry by \(1\bmod n\); the application sending
position \(0\) to position \(n-1\) additionally reverses the row.

In any block of \(n\) consecutive applications, every row position is
visited once, so each row wraps exactly once.  The cumulative entry shift
over \(n\) iterations is \(-n\equiv 0\pmod n\).  Hence \(\rho^n\) acts
as ``reverse every row'' on \(\mathcal{O}_n\) (with row labels and
entries unchanged), and \(\rho^{2n}\) reverses every row twice, giving
\(\rho^{2n}=\mathrm{id}\).
\end{proof}

\begin{corollary}[Action of \(\rho^{n}\)]
\label{cor:rho-n-no-fixed}
\(\rho^{n}\) fixes no \(O\)-matrix.
\end{corollary}

\begin{proof}
By \cref{lem:rho-order} the operation \(\rho^{n}\) reverses every row, and
by \cref{lem:no-palindrome} no row of an \(O\)-matrix equals its reversal;
hence \(\rho^{n}O\neq O\).
\end{proof}

\begin{lemma}[Conjugation relation]
\label{lem:conjugation}
\(\mu\rho\mu=\rho^{-1}\).
\end{lemma}

\begin{proof}
We compute \(\mu\rho\mu\) and compare with \(\rho^{-1}\).  From the
formula for \(\rho\),
\[
  (\rho^{-1}O)_{l,j} =
  \begin{cases}
    \bigl((O_{l-1})_j+1\bigr)\bmod n, & 1\le l\le n-1,\\
    \bigl((O_{n-1})_{n-2-j}+1\bigr)\bmod n, & l=0.
  \end{cases}
\]

Set \(A=\mu O\), \(B=\rho A\), \(C=\mu B\).  Then \(A_{i,j}=(n-1)-(O_{n-1-i})_j\)
and, by the formula for \(\rho\), for \(0\le k\le n-2\)
\[
  \begin{aligned}
    B_{k,j}
    &= \bigl((A_{k+1})_j - 1\bigr)\bmod n
     = \bigl((n-1)-(O_{n-2-k})_j - 1\bigr)\bmod n\\
    &= \bigl(n-2-(O_{n-2-k})_j\bigr)\bmod n,
  \end{aligned}
\]
and for \(k=n-1\)
\[
  B_{n-1,j} = \bigl((A_0)_{n-2-j} - 1\bigr)\bmod n
  = \bigl(n-2-(O_{n-1})_{n-2-j}\bigr)\bmod n.
\]
Consider \(C_{i,j}=(n-1)-B_{n-1-i,j}\).  For \(i=0\), \(n-1-i=n-1\), so
\[
  C_{0,j} = (n-1)-\Bigl(\bigl(n-2-(O_{n-1})_{n-2-j}\bigr)\bmod n\Bigr)
         = \bigl((O_{n-1})_{n-2-j}+1\bigr)\bmod n.
\]
For \(1\le i\le n-1\), \(0\le n-1-i\le n-2\), so
\[
  C_{i,j} = (n-1)-\Bigl(\bigl(n-2-(O_{i-1})_j\bigr)\bmod n\Bigr)
         = \bigl((O_{i-1})_j+1\bigr)\bmod n.
\]
Both expressions match \(\rho^{-1}O\), proving the claim.
\end{proof}

\begin{definition}[Dihedral action]
\label{def:dihedral-action}
Let \(D_{2n}=\langle\rho,\mu\mid\rho^{2n}=\mu^{2}=(\mu\rho)^{2}=e\rangle\) be
the abstract dihedral group of order \(4n\).  It acts on \(\mathcal{O}_n\)
with each generator acting as the operation of the same name.
\end{definition}

This action is well defined: by the universal property of the presentation it
suffices that the operations \(\rho\) and \(\mu\) satisfy the three defining
relations in \(\mathrm{Sym}(\mathcal{O}_n)\), namely \(\mu^{2}=\mathrm{id}\)
(immediate from \eqref{eq:mu-formula}), \(\rho^{2n}=\mathrm{id}\)
(\cref{lem:rho-order}), and \((\mu\rho)^{2}=\mathrm{id}\), which follows from
\(\mu\rho\mu=\rho^{-1}\) (\cref{lem:conjugation}).

\begin{remark}
\label{rem:action-not-faithful}
The action need not be faithful.  For \(n=3\), the set \(\mathcal{O}_3\) has
just two elements (corresponding to the reduced words \(\sigma_0\sigma_1\sigma_0\) and
\(\sigma_1\sigma_0\sigma_1\)), and \(\rho^{2}\) acts trivially on it (indeed
\(\mu=\rho\) as maps), so the image of \(D_{6}\) in
\(\mathrm{Sym}(\mathcal{O}_3)\) has order \(2\), not \(12\).  This is why
\(D_{2n}\) is taken abstractly: the counting below uses only its order \(4n\)
and the orbit--stabilizer theorem, both of which hold for any action.  One
consequence of non-faithfulness is that the Euclidean stabilizer \(H_E\)
(\cref{def:euclidean-stabilizer}) may be larger than the number of distinct
\(O\)-matrices in an orbit.
\end{remark}

\subsection{Induced line permutations}
\label{app:symmetries:induced-permutations}

We make explicit the induced line permutation \(\tau(g;O)\in S_{n+1}\)
introduced in \cref{subsec:symmetries-stabilizers} for \(\rho\), \(\mu\),
and \(\pi_p\), where the \(n+1\) projective lines are labeled
\(\{0,\ldots,n\}\) (\(n\) is the current line at infinity).

For \(\rho\) and \(\mu\) the induced permutation is determined by the
operation alone; for \(\pi_p\) it also depends on row \(p\) of \(O\) via
the relabeling defined in \eqref{eq:pi-relabeling}.

\emph{Reflection \(\mu\).} The line at infinity is unchanged.
Finite labels are reflected:
\begin{equation}
\label{eq:tau-mu}
  \tau(\mu;O) = \begin{pmatrix}
    0   & 1   & 2   & \cdots & n-1 & n \\
    n-1 & n-2 & n-3 & \cdots & 0   & n
  \end{pmatrix}.
\end{equation}

\emph{Euclidean rotation \(\rho\).} The finite lines are relabeled
\(i\mapsto(i-1)\bmod n\).  The line at infinity is fixed,
since rotation moves the sweep start within the affine chart.
\begin{equation}
\label{eq:tau-rho}
  \tau(\rho;O) = \begin{pmatrix}
    0   & 1 & 2 & \cdots & n-1 & n \\
    n-1 & 0 & 1 & \cdots & n-2 & n
  \end{pmatrix}.
\end{equation}

\emph{Projective rotation \(\pi_p\).}
Writing out the relabeling \(\tau\) of \eqref{eq:pi-relabeling}, extended
to send \(p\mapsto n\), the induced permutation \(\tau(\pi_p;O)\in S_{n+1}\)
reads
\begin{equation}
\label{eq:tau-pi}
  \tau(\pi_p;O) = \begin{pmatrix}
    (O_p)_0 & (O_p)_1 & \cdots & (O_p)_{n-2} & n   & p \\
    0       & 1       & \cdots & n-2         & n-1 & n
  \end{pmatrix}.
\end{equation}
By construction, row \(p\) of \(O\) is a permutation of
\(\{0,\ldots,n-1\}\setminus\{p\}\), so the top row of \eqref{eq:tau-pi}
ranges over \(\{0,\ldots,n\}\) without repetition and
\(\tau(\pi_p;O)\) is well-defined.
Geometrically, line \(p\) becomes the new line at infinity and takes the label
\(n\).  The previous line at infinity takes the finite label \(n-1\).
The other finite lines are renamed by their order of intersection along \(p\).

\begin{lemma}[Kernel of the induced-permutation map]
\label{lem:tau-kernel}
The map \(\Phi\colon D_{2n}\to S_{n+1}\), \(h\mapsto\tau(h)\), is a group
homomorphism with kernel \(\{\mathrm{id},\rho^{n}\}\).
\end{lemma}

\begin{proof}
By \eqref{eq:tau-mu} and \eqref{eq:tau-rho}, \(\tau(\mu)\) is a reflection
on the \(n\) finite labels and \(\tau(\rho)\) is an \(n\)-cycle.  They satisfy
the dihedral relations
\(\tau(\rho)^{2n}=\tau(\mu)^{2}=(\tau(\mu)\tau(\rho))^{2}=\mathrm{id}\)
(the first because \(\tau(\rho)^n=\mathrm{id}\)), so by the universal property
of the presentation \(\Phi\colon D_{2n}\to S_{n+1}\), \(\rho\mapsto\tau(\rho)\),
\(\mu\mapsto\tau(\mu)\), is a well-defined homomorphism.  (Concretely,
\(\tau(h_1h_2)=\tau(h_1)\circ\tau(h_2)\) follows from the composition law
\eqref{eq:tau-composition} and matrix-independence of \eqref{eq:tau-mu} and
\eqref{eq:tau-rho}.)  Its image \(\langle\tau(\rho),\tau(\mu)\rangle\) is
dihedral of order \(2n\), so \(|\ker\Phi|=4n/2n=2\) by \(|D_{2n}|=4n\).  Since
\(\tau(\rho^{n})=\tau(\rho)^{n}=\mathrm{id}\) while \(\rho^{n}\) acts
nontrivially (\cref{cor:rho-n-no-fixed}), so that \(\rho^{n}\neq\mathrm{id}\) in
\(D_{2n}\), the kernel is exactly \(\{\mathrm{id},\rho^{n}\}\).
\end{proof}

\subsection{Crossing-adjacency and the reduction of operations}
\label{app:symmetries:adjacency}

Each row of the projective closure \(O^+\) \eqref{eq:projective-closure}
has length \(n\) and records, read
cyclically, the order in which a projective line meets the other \(n\).
We say that labels \(b,c\) are \emph{adjacent along \(a\)} in \(O^+\) if
they occupy consecutive positions of row \((O^+)_a\), counted cyclically
(positions \(n-1\) and \(0\) being consecutive).  Geometrically, the
crossings of line \(a\) with \(b\) and \(c\) are consecutive along
\(a\), with no other crossing between them.

The operations all re-encode the same projective arrangement, so the
transformed \(O\)-matrix is read off the relabeled arrangement.
We claim that this transformation preserves crossing-adjacency
(\cref{lem:adjacency-preservation}), and in particular the cyclic order
of the labels along the line \(s\) sent to infinity.  Given that,
the labelings of \(gM\) and \(\pi_s M\) can differ only by an element
\(h\in D_{2n}\).  Crossing-adjacency and the characteristic property
of \(O\)-matrices (\cref{lem:o-matrix-determinacy}) force \(gM\) and
\(h\pi_s M\) to coincide (\cref{lem:chart-reduction}): every operation
acts within a single chart as a dihedral symmetry composed with the
projective rotation \(\pi_s\) that selects the chart's line at infinity.

\begin{lemma}[Operations preserve crossing-adjacency]
\label{lem:adjacency-preservation}
Let \(T\in\{\rho,\mu\}\cup\{\pi_p:0\le p\le n-1\}\) and \(O\in\mathcal{O}_n\),
and write \(t=\tau(T;O)\in S_{n+1}\) for the induced line permutation
(\cref{app:symmetries:induced-permutations}).  If \(b,c\) are adjacent
along \(a\) in \(O^+\), then \(t(b),t(c)\) are adjacent along \(t(a)\) in
\((TO)^+\).
\end{lemma}

\begin{proof}
By their definitions
(\cref{def:reflection,def:euclidean-rotation,def:projective-rotation}),
\(\rho\), \(\mu\), and \(\pi_p\) act on the
rows of the projective closure \(O^+\) by relabeling the entries and, within
each row, a reversal, a cyclic shift, or both.  None of these alters the
undirected cyclic consecutivity of a row's entries, so adjacency is preserved.
\end{proof}

By the groupoid law \eqref{eq:tau-composition}, crossing-adjacency is then
preserved by every \(g\in\mathcal{W}\): if \(b,c\) are
adjacent along \(a\) in \(O^+\), then \(\tau(g;O)(b),\tau(g;O)(c)\) are
adjacent along \(\tau(g;O)(a)\) in \((gO)^+\).

\begin{lemma}[Determinacy of the \texorpdfstring{$O$}{O}-matrix]
\label{lem:o-matrix-determinacy}
Let \(O_1,O_2\in\mathcal{O}_n\) be such that each row of \(O_2\) equals the
corresponding row of \(O_1\) or its reversal.
Then \(O_2=O_1\) or \(O_2=\rho^{n}O_1\).
\end{lemma}

\begin{proof}
Let \(S=\{\,a:(O_2)_a=\operatorname{rev}((O_1)_a)\,\}\) be the set of reversed
rows (well-defined by \cref{lem:no-palindrome}).  Then \(O_2\) is \(O_1\)
with the rows in \(S\) reversed.

Both \(O_1,O_2\in\mathcal{O}_n\) satisfy the triple condition of
\cite[Theorem~5.2.10]{felsner-goodman-2017}: for every \(i<j<k\) the
relative orders of the pair \(\{i,j\}\) in row \(k\), of \(\{i,k\}\) in
row \(j\), and of \(\{j,k\}\) in row \(i\) all agree.  Reversing row \(r\)
flips every relative order within it, hence flips the listed order exactly
for the triples whose corresponding index lies in \(S\).  For \(O_2\) to
meet the condition at \(i<j<k\), the indicators \([i\in S],[j\in S],[k\in
S]\) must coincide; since every two indices lie in a common triple (as
\(n\ge 3\)), all indices share the same membership in \(S\), so
\(S=\varnothing\) or \(S=\{0,\ldots,n-1\}\).  In the first case \(O_2=O_1\); in the second
\(O_2=\rho^{n}O_1\), since \(\rho^{n}\) reverses every row
(\cref{lem:rho-order}).
\end{proof}

\begin{lemma}[Operation sequences are dihedral within a chart]
\label{lem:chart-reduction}
Let \(M\in\mathcal{O}_n\) and \(g\in\mathcal{W}\), and let
\(s=\tau(g;M)^{-1}(n)\) be the label that \(g\) carries to the line at
infinity (so \(s=n\) when \(\tau(g;M)\) fixes \(n\); we set
\(\pi_n=\mathrm{id}\)).  Then there is an \(h\in D_{2n}\) for which
\(h\pi_s\) and \(g\) agree on \(M\) both as matrices and as induced
line permutations:
\begin{equation}
\label{eq:chart-reduction}
  gM=h\pi_s M,\qquad
  \tau(g;M)=\tau(h;\pi_s M)\circ\tau(\pi_s;M).
\end{equation}
In particular \(gM\sim_E\pi_s M\).
\end{lemma}

\begin{proof}
All operation sequences in \(\mathcal{W}\) preserve crossing-adjacency
(\cref{lem:adjacency-preservation} and the groupoid law
\eqref{eq:tau-composition}), so \(gM\) and \(\pi_s M\) each carry the same
projective arrangement as \(M\), relabeled by \(\tau(g;M)\) and
\(\tau(\pi_s;M)\) respectively.  By construction both \(g\) and \(\pi_s\)
(with \(\pi_n=\mathrm{id}\)) send the line \(s\) to infinity, so \(gM\) and
\(\pi_s M\) are \(O\)-matrices in the same affine chart.  Their labelings differ by the
bijection \(r=\tau(g;M)\circ\tau(\pi_s;M)^{-1}\) of the lines, which fixes the
infinity label \(n\) and permutes the finite labels.  Being the label
correspondence between two encodings of one arrangement, \(r\) preserves
crossing-adjacency.  Adjacency along the infinity line is the cycle
\(0\!-\!1\!-\!\cdots\!-\!(n-1)\!-\!0\), so on the finite labels \(r\) is a
rotation or reflection and equals \(\tau(h;\cdot)\) for some \(h\in D_{2n}\)
by \eqref{eq:tau-mu} and \eqref{eq:tau-rho}.

We now verify that this \(h\) satisfies both identities of
\eqref{eq:chart-reduction}.  Its line
permutation is \(\tau(h;\pi_s M)=r\), so by the groupoid law
\eqref{eq:tau-composition} and the definition of \(r\),
\(\tau(h;\pi_s M)\circ\tau(\pi_s;M)=r\circ\tau(\pi_s;M)=\tau(g;M)\).  Hence
\(h\pi_s M\) and \(gM\) are relabeled from \(M\) by the same line permutation,
so identically labeled lines coincide and their rows agree up to reversal; by
\cref{lem:o-matrix-determinacy} the two matrices are equal or differ by
\(\rho^{n}\).  In the latter case replace \(h\) by
\(\rho^{n}h\in D_{2n}\), which leaves \(\tau(h;\cdot)\) unchanged since
\(\rho^{n}\) induces the trivial line permutation
(\cref{lem:tau-kernel}).  Either way \(gM=h\pi_s M\), and in particular
\(gM\sim_E\pi_s M\).
\end{proof}

\begin{corollary}[Operation sequence with inverse action]
\label{cor:sequence-reverse-action}
For every $g\in\mathcal{W}$ and $O\in\mathcal{O}_n$ there is
a $w\in\mathcal{W}$ with $w(gO)=O$, hence each $P$-class is
a single $\mathcal{W}$-orbit.
\end{corollary}

\begin{proof}
For $\rho,\mu$ take $\rho^{2n-1}$ and $\mu$; for $\pi_p$,
\cref{lem:chart-reduction} gives $\pi_{n-1}\pi_pO=hO$ with $h\in D_{2n}$
(by \eqref{eq:tau-pi}, \(\pi_p\) sends infinity label \(n\) to \(n-1\)
and \(\pi_{n-1}\) restores it), so $w=h^{-1}\pi_{n-1}$ serves.
Composing these generator inverses, each taken at the intermediate matrix
it acts on (right to left), gives $w$ for every $g\in\mathcal{W}$.
Hence forward reachability is symmetric; being also reflexive and transitive,
it is an equivalence relation containing the generators of $\sim_P$, so it
coincides with $\sim_P$ and each $P$-class is a single $\mathcal{W}$-orbit.
\end{proof}

\begin{corollary}[The \(n+1\) charts exhaust a \(P\)-class]
\label{cor:chart-exhaustion}
For every \(M\in\mathcal{O}_n\), each \(O\)-matrix \(P\)-equivalent to
\(M\) is \(E\)-equivalent to one of the \(n+1\) \(O\)-matrices
\(M\), \(\pi_0 M\), \ldots, \(\pi_{n-1}M\).
\end{corollary}

\begin{proof}
Let \(N\sim_P M\). By \cref{cor:sequence-reverse-action}, \(N=gM\)
for some \(g\in\mathcal{W}\). Take \(s=\tau(g;M)^{-1}(n)\).
By \cref{lem:chart-reduction}, \(N=gM\sim_E\pi_s M\).
\end{proof}

\begin{lemma}[Symmetry group is a \texorpdfstring{$P$}{P}-class invariant]
\label{lem:gp-conjugacy}
If \(M,M'\in\mathcal{O}_n\) lie in the same \(P\)-class, then \(G(M)\) and
\(G(M')\) \eqref{eq:stabilizer-image} are conjugate subgroups of \(S_{n+1}\);
in particular \(|G(M)|\) depends only on the \(P\)-class.
\end{lemma}

\begin{proof}
By \cref{cor:sequence-reverse-action} there are \(a,\bar a\in\mathcal{W}\)
with \(aM=M'\) and \(\bar aM'=M\).  Put \(\gamma=\tau(a;M)\in S_{n+1}\).
Since \(\bar aa\) fixes \(M\), \(\beta=\tau(\bar aa;M) \in G(M)\),
and the groupoid law \eqref{eq:tau-composition} applied to
\(\bar aa\) gives \(\beta=\tau(\bar a;M')\,\gamma\), that is
\(\tau(\bar a;M')=\beta\,\gamma^{-1}\).

Now take \(\alpha=\tau(g;M)\in G(M)\) (so \(gM=M\)).  The sequence
\(g^\ast=ag\bar a\in\mathcal{W}\) fixes \(M'\), and the groupoid law gives
\[
  \tau(g^\ast;M')=\tau(a;M)\,\tau(g;M)\,\tau(\bar a;M')
  =\gamma\,\alpha\,\beta\,\gamma^{-1}.
\]
As \(G(M)\) is a group (\cref{lem:gp-is-subgroup}), \(\alpha\beta\) ranges
over \(G(M)\) with \(\alpha\), so \(\gamma\,G(M)\,\gamma^{-1}\subseteq G(M')\).
Exchanging the roles of \(M,M'\) gives the reverse inclusion, whence
\(G(M')=\gamma\,G(M)\,\gamma^{-1}\) by finiteness.
\end{proof}

\section{Canonicalization and the symmetry detector}
\label{app:canonicalization}

\Cref{sec:classification} and \cref{app:symmetries} define the $E$- and
$P$-equivalences and the symmetry data $G_P$, $H_E$ abstractly.  This
appendix describes the implementation details: a computable canonical
representative for each class, used for equivalence testing and
deduplication; a short hash-based identifier for cross-referencing;
and the algorithm computing $G_P$ and $H_E$.

Monospace names below (e.g., \texttt{exact\_min\_o\_key})
refer to functions and classes of the
reference implementation under
\texttt{combinatorics/sorter/}~\cite{parpalak-utkin-pseudoline-algorithms}.

\subsection{Canonical representatives and deduplication}
\label{app:canonicalization:forms}

Deduplicating the search output into \(E\)- and \(P\)-classes
requires deciding when two
\(O\)-matrices are \(E\)- or \(P\)-equivalent
(\cref{def:e-equivalence,def:p-equivalence}).  We map each
\(O\)-matrix to the \emph{canonical representative} of the class: one
distinguished member, chosen so that two matrices are
equivalent if and only if the canonical representatives are identical.
Equivalence testing then becomes byte-for-byte comparison, and
deduplication becomes grouping by canonical representative.

The implementation represents \(O\in\mathcal{O}_n\) as a flat byte
buffer of length \(n(n-1)\) obtained by concatenating the rows in
order \(O_0,O_1,\ldots,O_{n-1}\).  Lexicographic order on these buffers gives a total order on
\(\mathcal{O}_n\); we write \(<_{\mathrm{lex}}\) for this order.

\emph{Euclidean canonical representative.}
For \(O\in\mathcal{O}_n\), the \emph{E-canonical representative} of the
\(E\)-class \([O]_E\) is
\begin{equation}
\label{eq:e-canonical}
  E\text{-can}(O)\ =\ \min{}_{<_{\mathrm{lex}}}\bigl\{\,h(O)\ :\ h\in D_{2n}\,\bigr\}.
\end{equation}
The orbit has size at most \(4n\) (\cref{def:dihedral-action}); the
implementation enumerates the \(2n\) rotations \(\rho^k\) and the
\(2n\) reflected rotations \(\mu\rho^k\) directly on the byte buffer
via the primitives \texttt{fixed\_o::rotate} and
\texttt{fixed\_o::mirror}, and keeps the lexicographically smallest
result (\texttt{exact\_min\_o\_key}).

\emph{Projective canonical representative.}
For \(O\in\mathcal{O}_n\), the \emph{P-canonical representative} of the
\(P\)-class \([O]_P\) is
\begin{equation}
\label{eq:p-canonical}
  P\text{-can}(O)\ =\ \min{}_{<_{\mathrm{lex}}}\bigl\{\,E\text{-can}(O),\ \,
                   E\text{-can}(\pi_0 O),\ \ldots,\ E\text{-can}(\pi_{n-1}O)\,\bigr\}.
\end{equation}
The set under the minimum has \(n+1\) elements: the input chart and its
\(n\) projective rotations. By \cref{cor:chart-exhaustion} these
exhaust the \(P\)-class up to \(E\)-equivalence, so the \(P\)-class is
the union of their \(D_{2n}\)-orbits and \eqref{eq:p-canonical} is the
lex-minimum over the full \(P\)-orbit.
The implementation \texttt{get\_projective\_orbit\_keys} calls
\texttt{exact\_min\_o\_key} on each \(\pi_p O\) and takes the byte-wise
minimum across the \(n+1\) results.

In both cases the canonical representative is a concrete matrix in
\(\mathcal{O}_n\); the byte buffer it produces is the \emph{exact key}
used by the deduplication stage (\texttt{sorter}) to partition the
input into commutation classes (\(D\)), \(E\)-classes, or \(P\)-classes via dictionary lookup on
the corresponding key (raw bytes for \(D\), E-canonical bytes for
\(E\), P-canonical bytes for \(P\); selected by
\texttt{-{}-uniq D|E|P}).

The exact key is a canonical handle in \(\mathcal{O}_n\), not a reduced
word.  When \texttt{sorter} prints a class as a reduced word
(\cref{app:data-formats}), the printed word is the first input word
whose canonicalization yielded this key, not a word reconstructed from
the canonical $O$-matrix (\cref{app:word-reconstruction}); class equality is decided against the key, not
against the printed word.

\subsection{Class identifiers}
\label{app:canonicalization:identifiers}

The exact key of \cref{app:canonicalization:forms} fully
identifies a class but is too long for cross-referencing
(\(n(n-1)\) bytes per $O$-matrix).  The classifier pairs each class with a
short hash-derived label.

\begin{definition}[Class identifiers]
\label{def:class-identifiers}
For \(C\in\{E,P\}\) and the corresponding canonical key \(K\) of
\cref{app:canonicalization:forms}, the \emph{\(C\)-class
identifier} is the string
\[
  \texttt{<C><n>-<base36(FNV1a64(\(K\)))>},
\]
where \(\mathtt{FNV1a64}\) is the 64-bit FNV-1a hash (offset basis
\(\mathtt{0xcbf29ce484222325}\), prime \(\mathtt{0x100000001b3}\)),
encoded in lowercase base~36.  We write \emph{\(E\)-id} (or
\texttt{eid}) for \(C=E\) and \emph{\(P\)-id} (or \texttt{pid}) for
\(C=P\).
\end{definition}

Equal classes have equal keys and therefore equal identifiers.  The
converse can fail to a 64-bit hash collision; at the enumeration scale
of \cref{sec:computational-results} (\(\#P\le 5.7\times 10^4\) at
\(n=27\)) the expected number of identifier collisions across all
\(P\)-classes is below \(10^{-10}\), and we observe none.  Class
equality is verified against the exact key, not against the
identifier.

\subsection{Construction of \texorpdfstring{$G_P$}{G\_P} and
\texorpdfstring{$H_E$}{H\_E}}
\label{app:canonicalization:groups}

We now describe the algorithm used by the symmetry detector
\texttt{sym\_detect} to compute the symmetry group \(G_P\) of a
\(P\)-class (\cref{def:p-symmetry-group}) and the Euclidean
stabilizer \(H_E\subseteq D_{2n}\) of each \(E\)-class contained in
it (\cref{def:euclidean-stabilizer}).

\emph{Input.} A representative \(O\)-matrix \(M\) of the \(P\)-class.
Any representative is acceptable: a different one conjugates \(G_P\) in
\(S_{n+1}\) (\cref{lem:gp-conjugacy}) and spans the same \(E\)-classes
(\cref{cor:chart-exhaustion}), so the group orders, the isomorphism types,
and the multiplicities \(m(E)\) are unchanged.
The implementation builds \(M\)
from a reduced word encoding it, via the \(O(W)\) construction of
\cref{subsec:order-matrices}.

\emph{Output.} The group \(G_P\), as a subgroup of \(S_{n+1}\) (order,
isomorphism type, generators, and element-order histogram); and, for
each \(E\)-class \(E\subseteq P\), its multiplicity \(m(E)\) and
Euclidean stabilizer \(H_E\) (order, type, and generators).

The algorithm runs in four steps. Step~1 builds the \(O\)-matrices in the \(n+1\)
affine charts of \(P\) with their induced line permutations. Steps~2--4
consume these to produce \(G_P\), the partition of charts into
\(E\)-classes, and the per-class stabilizers. A final block verifies
several invariants.

\emph{Step 1: the \(n+1\) charts.}
For each \(s\in\{0,\ldots,n\}\), let \(M_s:=\pi_s M\) (\(\pi_n:=\mathrm{id}\))
be the \(O\)-matrix of \(P\) in the chart with line \(s\) at infinity.
Record each induced line permutation \(\theta_s=\tau(\pi_s;M)\in
S_{n+1}\) of \eqref{eq:tau-pi}, with \(\theta_n=\mathrm{id}\). The
implementation stores these as \texttt{chart\_perms[s]}, computed via
\texttt{traced\_projective\_rotation}.

\emph{Step 2: enumerate Euclidean overlaps for \(G_P\).}
For every \(M_s\) and every
\(
h\in D_{2n}=\{\rho^k:0\le k<2n\}\cup\{\mu\rho^k:0\le k<2n\},
\)
test whether
\(hM_s=M\) by direct byte comparison. A match witnesses an element
\(g_{s,h}=h\pi_s\) of \(\operatorname{Stab}(M)\) (with
\(\pi_n:=\mathrm{id}\)); its induced line permutation is
\begin{equation}
\label{eq:gp-element}
  \tau(g_{s,h};M)\ =\ \tau(h;M_s)\circ\theta_s\ \in\ S_{n+1},
\end{equation}
by the groupoid law \eqref{eq:tau-composition}, where
\(\tau(h;\,\cdot)\) for \(h\in D_{2n}\) is determined by
\eqref{eq:tau-mu} and \eqref{eq:tau-rho} (and their compositions)
independently of the matrix it acts on.

By \cref{lem:chart-reduction} every operation acts on \(M\) as some
\(h\pi_s\), so the enumerated matches are exactly the operations fixing
\(M\): the matches are \(\operatorname{Stab}(M)\), and their induced
permutations are \(G_P\). Each such permutation \eqref{eq:gp-element}
sends \(s\) to the infinity label \(n\), so it is the identity only when
\(s=n\); there \(g_{s,h}=h\in D_{2n}\) with \(hM=M\), and
\(\tau(h;M)=\mathrm{id}\) forces \(h=\mathrm{id}\) by
\cref{lem:tau-kernel}, since its other kernel element \(\rho^n\) does not
fix \(M\) (\cref{cor:rho-n-no-fixed}). Thus the identity is the only
element of \(\operatorname{Stab}(M)\) inducing the trivial permutation,
so \(\tau\) is injective there and the enumeration lists each element of
\(G_P\) once. The implementation collects the \(\tau\)-images and reads off
generators, element orders, and the isomorphism type via
\texttt{LinePermutationGroup}.

\emph{Step 3: partition charts into \(E\)-classes.}
For each \(M_s\), compute its E-canonical representative
\(E\text{-can}(M_s)\). Charts whose \(M_s\) have the same E-canonical representative lie in
the same \(E\)-class; the partition realizes
\cref{def:multiplicity}, with \(m(E)\) equal to the number of charts
collapsing to a single E-canonical representative. The check
\(\sum_{E\subseteq P} m(E)=n+1\) is verified after the pass.

\emph{Step 4: \(H_E\) per \(E\)-class.}
Pick one chart \(s^\ast\) per \(E\)-class. Enumerate all
\(h\in D_{2n}\) and test
\(hM_{s^\ast}=M_{s^\ast}\) by byte comparison. The matches form
\(H_E\subseteq D_{2n}\). To recover \(H_E\) as a subgroup of
\(S_{n+1}\), the implementation records the induced permutation
\(\tau(h;M_{s^\ast})\) for each match; these are pairwise distinct, as in
Step~2, since the only nontrivial kernel element \(\rho^n\) does not fix
\(M_{s^\ast}\) (\cref{cor:rho-n-no-fixed}). Each element is tagged as a
\emph{rotation} (\(h=\rho^k\)) or a \emph{reflection}
(\(h=\mu\rho^k\)); the resulting reflection count disambiguates cyclic
\(C_k\) from dihedral \(D_k\) stabilizers when the abstract group
alone would not.

\emph{Consistency checks.}
The implementation verifies several invariants on every output:
\(\sum_{E\subseteq P} m(E)=n+1\) (charts partition into
\(E\)-classes);
\(m(E)\cdot|H_E|=|G_P|\) for every \(E\subseteq P\)
(\cref{lem:orbit-stabilizer-charts});
\(\operatorname{rep}(E)\cdot|H_E|=2n\)
(\cref{lem:parity-fixed-representatives}); that the symmetry enumerations of
Steps~2 and~4 yield as many distinct line permutations as there are
matches, witnessing \(\rho^n\notin\operatorname{Stab}(M)\) and
\(\rho^n\notin H_E\) on the actual data; and that the recovered
generators close to the full group with the recorded element-order
histogram. A failure of any check is logged with the position of the
offending \(P\)-class.

\section{Self-consistency of the 2-defective search}
\label{app:special-self-consistency}

This appendix records consistency checks for the 2-defective search
(\cref{subsec:2-defective-algorithm}).  Although the emitted words
depend on the parameter $X$, the covered \(P\)-classes and the number
of wiring-diagram representatives per \(P\)-class do not
(\cref{thm:2-defective-completeness,rem:2-defective-representatives}). We check this
invariance and record the results.

We consider three search modes: the $K_g^+$-only and $K_g^-$-only
experimental modes restrict the special generator alphabet to a single variant,
while the both mode has no restrictions
(binaries \texttt{two\_defect\_inc}, \texttt{two\_defect\_dec}, and
\texttt{two\_defect})~\cite{parpalak-utkin-pseudoline-algorithms}.
Only the both mode guarantees completeness
(\cref{thm:2-defective-completeness}).

Throughout this appendix, $\#D^\star$ and $\#E^\star$ count the wiring
diagrams and $E$-classes \emph{realized by the words emitted by the
2-defective search}. The totals $\#D$ and $\#E$ are recorded in
\cref{tab:complete-counts} and are in general
larger (for $n=19$, the both mode gives $\#E^\star=9236$ at any $X$
versus $\#E=26\,084$ in total).

\subsection{Two quadrilateral defects}
\label{subsec:special-self-consistency-two-quad}

\cref{tab:special-self-consistency-n19} reports the output sizes ($\#\text{words}$) of
all three modes at $n=19$ for every $X\in\{1,3,\ldots,17\}$,
and the $\#D^\star$ and $\#E^\star$
counts after deduplication (\cref{app:canonicalization:forms}).  The
number of $P$-classes is $\#P=1312$ throughout and is
omitted from the table.

\emph{For the both mode}, the triple $(\#D^\star, \#E^\star, \#P) = (20\,864, 9236, 1312)$
is invariant in $X$, illustrating
\cref{thm:2-defective-completeness,rem:2-defective-representatives}: each
run already exhausts the $P$-class set, and contributes the same number of
wiring diagrams.

\emph{The experimental $K_g^+$-only and $K_g^-$-only modes} also achieve
$\#P=1312$ at every $X$, but their $\#D^\star$ and $\#E^\star$ counts vary
with $X$, consistent with each experimental mode at fixed $X$ missing a subset
of wiring diagrams and covering every $P$-class only through alternative
representatives.  The equality $\#\text{words}=\#D^\star$ in every row
shows that these modes do not emit commutation-equivalent
words~\eqref{eq:special-commutation}.

\begin{table}[H]
\centering
\captionsetup{width=0.9\linewidth}
\caption{Output sizes of the 2-defective search at $n=19$ by skipped
generator $X$ and search mode.  The constant column $\#P=1312$ is
omitted.}
\label{tab:special-self-consistency-n19}
\small
\setlength{\tabcolsep}{4pt}
\begin{tabular*}{\linewidth}{@{}r@{\extracolsep{\fill}}rrrrrrrrr@{}}
\toprule
 & \multicolumn{3}{c}{$K_g^+$-only} & \multicolumn{3}{c}{$K_g^-$-only} & \multicolumn{3}{c}{both} \\
\cmidrule(lr){2-4}\cmidrule(lr){5-7}\cmidrule(lr){8-10}
$X$ & $\#\text{words}$ & $\#D^\star$ & $\#E^\star$ & $\#\text{words}$ & $\#D^\star$ & $\#E^\star$ & $\#\text{words}$ & $\#D^\star$ & $\#E^\star$ \\
\midrule
 1 & 15\,502 & 15\,502 & 8\,452 & 16\,346 & 16\,346 & 9\,009 & 31\,798 & 20\,864 & 9\,236 \\
 3 & 15\,629 & 15\,629 & 8\,854 & 15\,815 & 15\,815 & 8\,860 & 31\,292 & 20\,864 & 9\,236 \\
 5 & 15\,838 & 15\,838 & 9\,027 & 15\,867 & 15\,867 & 8\,936 & 31\,533 & 20\,864 & 9\,236 \\
 7 & 15\,885 & 15\,885 & 9\,215 & 15\,882 & 15\,882 & 9\,124 & 31\,583 & 20\,864 & 9\,236 \\
 9 & 15\,931 & 15\,931 & 9\,236 & 16\,059 & 16\,059 & 9\,236 & 31\,766 & 20\,864 & 9\,236 \\
11 & 15\,789 & 15\,789 & 9\,025 & 15\,893 & 15\,893 & 9\,187 & 31\,532 & 20\,864 & 9\,236 \\
13 & 15\,741 & 15\,741 & 8\,779 & 15\,754 & 15\,754 & 8\,851 & 31\,283 & 20\,864 & 9\,236 \\
15 & 15\,839 & 15\,839 & 8\,909 & 15\,722 & 15\,722 & 8\,916 & 31\,333 & 20\,864 & 9\,236 \\
17 & 16\,196 & 16\,196 & 8\,740 & 16\,546 & 16\,546 & 8\,722 & 31\,714 & 20\,864 & 9\,236 \\
\bottomrule
\end{tabular*}
\end{table}

To illustrate the counts in \cref{rem:2-defective-representatives}, we group
the output of the both mode at $n=13$ by individual $P$-classes.
After attaching the canonical $P$-class identifier to every emitted word
(\cref{def:class-identifiers}; concretely
\verb|sorter -p pid ')' gens|), we count
$\#\text{words}$, $\#D^\star$, and $\#E^\star$ within each $P$-class.

The $\#D^\star$ and $\#E^\star$ columns in \cref{tab:special-self-consistency-n13} are invariant in $X$ for every individual
$P$-class.  The contributed wiring diagrams themselves differ with $X$
(a different skipped generator \(\sigma_X\) anchors the defect
at a different place on the initial boundary),
but their number \(\#D^\star\) is preserved. In contrast,
the same $E$-classes recur for every $X$, since the diagrams contributed
at different $X$ are related by $\rho^k$ and hence Euclidean-equivalent
(\cref{def:e-equivalence}).
The word count $\#\text{words}$, however, varies with $X$, reflecting
different commutation-class representatives in the search output.

Each $P$-class has either $16$ or $8$ representative wiring diagrams
($\#D^\star$), realizing the two-quadrilateral count $16/|G_P|$ of
\cref{rem:2-defective-representatives}: the two classes with $\#D^\star=16$ have trivial
$G_P$, and the four with $\#D^\star=8$ are halved by a
symmetry of order~$2$ in $G_P$.  This matches the
$G_P$ distribution for $n=13$ in \cref{tab:sym-combined}: two
$P$-classes with $G_P\cong C_1$ and four with $G_P\cong C_2$.

\begin{table}[H]
\centering
\captionsetup{width=0.8\linewidth}
\caption{Output sizes of the both mode at $n=13$ for
$X\in\{1,3,\ldots,11\}$.  $\#\text{words}$, $\#D^\star$ and $\#E^\star$
are reported per $P$-class, labeled by identifier
(\cref{def:class-identifiers}).}
\label{tab:special-self-consistency-n13}
\small
\setlength{\tabcolsep}{4pt}
\begin{tabular}{lrrrrrrrr}
\toprule
& \multicolumn{6}{c}{$\#\text{words}$ at $X=$} & \multicolumn{2}{c}{invariant in $X$} \\
\cmidrule(lr){2-7}\cmidrule(lr){8-9}
$P$-class & 1 & 3 & 5 & 7 & 9 & 11 & $\#D^\star$ & $\#E^\star$ \\
\midrule
\texttt{P13-1lj52719vw4m1} & 12 & 13 & 12 & 11 & 11 & 12 &  8 & 4 \\
\texttt{P13-1ypd38fakt521} & 13 & 11 & 12 & 11 & 11 & 12 &  8 & 4 \\
\texttt{P13-2jyo1nd6n4f39} & 23 & 24 & 25 & 24 & 20 & 24 & 16 & 7 \\
\texttt{P13-2tcv981y0y3sl} & 14 & 13 & 13 & 13 & 10 & 16 &  8 & 4 \\
\texttt{P13-5rnnoook1aox} & 10 &  9 & 10 & 10 &  9 & 10 &  8 & 3 \\
\texttt{P13-85c0n1xk2un9} & 20 & 21 & 21 & 21 & 21 & 20 & 16 & 6 \\
\bottomrule
\end{tabular}
\end{table}

\subsection{One pentagonal defect}
\label{subsec:special-self-consistency-pentagon}

The both mode at $n=25$, $X=5$, restricted to the
single-pentagon family (\verb|two_defect -defects b5|), produces $35\,314$ words
across $23\,240$ wiring diagrams, $11\,620$ $E$-classes, and $2324$
$P$-classes.  Every $P$-class contributes exactly $10$ wiring diagrams,
the single-pentagon count of \cref{rem:2-defective-representatives}: every
$G_P$ is trivial.  The $10$ diagrams form $5$
$E$-classes, each a mirror-image pair.

The only quantity that varies from $P$-class to $P$-class is
$\#\text{words}$: each of the $10$ wiring diagrams is reached by one or
two commutation-equivalent words~\eqref{eq:special-commutation}, so the
count ranges from $10$ to $20$.  Its distribution over the $2324$
$P$-classes is $51, 173, 812, 944, 312, 32$ classes at
$\#\text{words}=10, 12, 14, 16, 18, 20$, respectively, summing to
$10\cdot 51 + 12\cdot 173 + 14\cdot 812 + 16\cdot 944 + 18\cdot 312
+ 20\cdot 32 = 35\,314$ words.

At the boundary values $X=1$ and $X=n-2$
relation~\eqref{eq:special-commutation} does not apply to the pentagonal
anchor (proof of \cref{lem:special-coverage-internal}).  Each run at $X=1$
and $X=23$ emits $23\,240$ words, matching $\#D^\star$ exactly,
while $\#E^\star=11\,620$ and $\#P=2324$ agree with $X=5$.  In this case
every wiring diagram is reached by exactly one word.

\section{Reconstructing a reduced word from an \texorpdfstring{$O$}{O}-matrix}
\label{app:word-reconstruction}

\Cref{subsec:order-matrices} maps a reduced word $W$ to its $O$-matrix
$O(W)$.  We record the inverse: given $O\in\mathcal{O}_n$, the procedure
below produces a reduced word $W$ with $O(W)=O$, unique up to commutations
(\cref{lem:o-matrix-commutation}).

\begingroup\ttfamily
\begin{tabbing}
xxxx\=xxxx\=\kill
Initialize $a=(0,1,\ldots,n-1)$ and $p_i=0$ for $i=0,\ldots,n-1$;\\
for step $=1,\ldots,n(n-1)/2$:\\
\>find the smallest $g$ with $O_{a_g}[p_{a_g}]=a_{g+1}$ and $O_{a_{g+1}}[p_{a_{g+1}}]=a_g$;\\
\>output $\sigma_g$; increment $p_{a_g}$, $p_{a_{g+1}}$; swap $a_g\leftrightarrow a_{g+1}$.
\end{tabbing}
\endgroup

\section{Reproducing the results: a worked example}
\label{app:data-formats}

The pipeline producing the numbers and tables of \cref{sec:computational-results}
is a chain of four programs:
\[
  \langle\textit{search}\rangle \;\longrightarrow\;
  \texttt{sorter} \;\longrightarrow\;
  \texttt{sym\_detect} \;\longrightarrow\;
  \texttt{sym\_detect\_aggregate.py}.
\]
Here $\langle\textit{search}\rangle$ is the instance-specific search binary:
\texttt{perfect} for the perfect search (\cref{sec:perfect-search}) or
\texttt{two\_defect} for the 2-defective
search (\cref{sec:2-defective-search}).
The command lines below assume all four are on \texttt{\$PATH}.
\begin{description}
  \item[$\langle\textit{search}\rangle$] prints one reduced word per line with a
    metadata prefix.
  \item[\texttt{sorter}] deduplicates the words under one of three
    equivalences, selected by the \verb|--uniq| flag, and prints one
    representative per class (\cref{app:canonicalization:forms}).  The
    classes are commutation classes (mode \verb|D|: words modulo the
    commutations~\eqref{eq:commute}, equivalently $O$-matrices by
    \cref{lem:o-matrix-commutation}), $E$-classes (mode \verb|E|,
    \cref{def:e-equivalence}), or $P$-classes (mode \verb|P|,
    \cref{def:p-equivalence}).
  \item[\texttt{sym\_detect}] reads $P$-class representatives and prints
    one JSON record per class, with $G_P$, the $H_E$ for every
    contained $E$-class, and the orbit--stabilizer consistency checks of
    \cref{subsec:symmetries-stabilizers}; the algorithm computing $G_P$
    and $H_E$ is given in \cref{app:canonicalization:groups}.
  \item[\texttt{sym\_detect\_aggregate.py}] turns these records into either
    an ASCII profile table or the LaTeX source of \cref{tab:sym-combined}.
\end{description}

This appendix runs the full pipeline on the two smallest instances with
$\#D>1$: $n=9$ for the perfect search and $n=7$ with $X=5$ for the
2-defective search, showing every command and its output.

\paragraph{Step 1: enumerate.}
The perfect search of \cref{sec:perfect-search} is implemented as
\texttt{perfect}.  Run on $n=9$:

\begin{small}
\begin{verbatim}
$ perfect -n 9 | sed -n 's/^.*) //p' > 9.words.txt
$ cat 9.words.txt
7 5 3 1 6 5 7 4 3 5 2 1 3 0 1 4 3 5 2 1 3 6 5 7 4 3 5 2 1 3 6 5 7 4 3 5
7 5 3 1 4 3 5 2 1 3 6 5 7 4 3 5 2 1 3 0 1 6 5 7 4 3 5 2 1 3 4 3 5 6 5 7
7 5 3 1 2 1 3 4 3 5 6 5 7 4 3 5 2 1 3 0 1 2 1 3 4 3 5 6 5 7 4 3 5 2 1 3
\end{verbatim}
\end{small}
The \texttt{sed} filter strips the metadata prefix (\verb|[time] A=... K=... i=...) |)
that the search writes ahead of each word.  Integers are
elementary-generator indices.  At $n=9$ the search prints three reduced
words, matching $\#D=3$ in \cref{tab:complete-counts}.

\paragraph{Step 2: deduplicate.}
The \texttt{sorter} canonicalizes under the requested equivalence.  At
$\sim_P$ on $n=9$ the three reduced words collapse to a single
$P$-class:

\begin{small}
\begin{verbatim}
$ sorter -i 9.words.txt --uniq P --exact-cache --print gens > 9.uniq-p.txt
$ cat 9.uniq-p.txt
7 5 3 1 6 5 7 4 3 5 2 1 3 0 1 4 3 5 2 1 3 6 5 7 4 3 5 2 1 3 6 5 7 4 3 5
\end{verbatim}
\end{small}
The \verb|--print| flag selects which fields to print per representative;
\verb|gens| prints the generator sequence on its own line.

\paragraph{Step 3: analyze symmetries.}
The \texttt{sym\_detect} program analyzes the symmetries of each
representative, printing one JSON record per $P$-class.  At $n=9$ it
produces a single record (formatted here for readability; the on-disk
JSONL is one line per record):

\begin{small}
\begin{verbatim}
$ sym_detect -i 9.uniq-p.txt
{"n": 9,
 "pid": "P9-1xw8b9ugsk55",
 "gp_ord": 60,
 "gp_type": "A5",
 "gp_gens": [[0,1,7,8,6,9,4,2,3,5], [1,3,0,4,2,6,9,8,5,7]],
 "gp_elt_orders": {"1": 1, "2": 15, "3": 20, "5": 24},
 "chart_to_e": [0,0,0,0,0,0,0,0,0,0],
 "e_classes": [
   {"eid": "E9-1xw8b9ugsk55",
    "m": 10, "reps": 3,
    "he_type": "D3", "he_ord": 6,
    "he_gens": [[1,0,8,7,6,5,4,3,2,9], [3,4,5,6,7,8,0,1,2,9]]}
 ],
 "checks": {"sum_m_ok": true, "orbit_stabilizer_ok": true,
            "representatives_ok": true, "generators_close": true,
            "e_generators_close": true, "chart_to_e_ok": true}}
\end{verbatim}
\end{small}
The fields are: $n$, the $P$-class identifier \texttt{pid}, the symmetry
group $G_P$ (\cref{def:p-symmetry-group}) given by its order
\texttt{gp\_ord}, a heuristic isomorphism-type name \texttt{gp\_type},
generators \texttt{gp\_gens} as line permutations of the $n+1$ projective
lines, and an element-order histogram \texttt{gp\_elt\_orders}; the array
\texttt{chart\_to\_e} maps each of the $n+1$ affine charts to the local
$E$-class identifier within this
$P$-class, so the multiplicity $m(E)$ of a class is the number of
occurrences of its id in \texttt{chart\_to\_e}; the array has $n+1$
entries, so the multiplicities sum to $n+1$ (here a single id $0$ fills
all $n+1=10$);
\texttt{e\_classes} is the per-$E$-class structure: identifier
\texttt{eid}, multiplicity $m(E)$ from \cref{def:multiplicity}, the count
$\operatorname{rep}(E)$ of \cref{lem:parity-fixed-representatives}, the
Euclidean stabilizer $H_E$ of \cref{def:euclidean-stabilizer} by order
and isomorphism type, and its generators as line permutations fixing the
line at infinity; the
\texttt{checks} block records the consistency checks of
\cref{subsec:symmetries-stabilizers}.

The single $E$-class spans $m(E)=10$ of the $n+1=10$ affine charts of
this $P$-class, and the perfect search emits it
$\operatorname{rep}(E)=3$ times, matching the three reduced
words ($\#D = \operatorname{rep}(E)\cdot\#E/P = 3\cdot 1 = 3$).

\paragraph{Step 4: aggregate.}
The Python post-processor groups records by ($G_P$, $m$-profile,
$(\operatorname{rep}(E), H_E)$-profile) and prints a per-profile table
whose columns mirror those of \cref{tab:sym-combined}.  For $n=9$:

\begin{small}
\begin{verbatim}
$ sym_detect -i 9.uniq-p.txt | sym_detect_aggregate.py
n = 9, P-classes = 1, profiles = 1

G_P | |G_P| | #E/P | m(E) | rep(E) | H_E | rep(P) | #P | #E | #D
----+-------+------+------+--------+-----+--------+----+----+---
A5  |    60 |    1 |   10 |      3 | D3  |      3 |  1 |  1 |  3
----+-------+------+------+--------+-----+--------+----+----+---
\end{verbatim}
\end{small}
The single profile gives $\#D=3$.

The \verb|G_P| label \verb|A5| comes directly from \texttt{sym\_detect}, which
names the isomorphism types it expects ($C_k$, $D_k$, $A_4$, $S_4$, $A_5$) and
otherwise prints a plain order descriptor \texttt{"nonabelian order $N$"} or
\texttt{"abelian order $N$"}; the group itself is recoverable from
\texttt{gp\_gens}.  The \verb|--format latex| flag typesets the same profile as
a per-$n$ table; \cref{tab:sym-combined} stacks all $n$ via
\verb|--format latex-combined|.

\paragraph{2-defective search at $n=7$, $X=5$}
The 2-defective search of \cref{sec:2-defective-search} is implemented
as \verb|two_defect|; the parameter $X$ is exposed as
\verb|-X|.  The pipeline runs identically except for the search
binary.  At $n=7$ with $X=5$ the search emits each accepted leaf
together with its defect-jump image of
\cref{subsec:2-defective-algorithm} (where one exists), printing five
reduced words:

\begin{small}
\begin{verbatim}
$ two_defect -n 7 -X 5 | sed -n 's/^.*) //p' > 7.words.txt
$ cat 7.words.txt
3 1 4 5 2 1 3 4 3 5 2 1 3 0 1 2 1 3 4 3 5
3 1 4 3 5 2 1 3 4 3 5 2 1 3 0 1 2 3 4 3 5
3 1 2 1 3 0 1 4 3 5 2 1 3 4 3 5 2 1 3 0 1
3 1 2 1 3 4 3 5 4 3 2 1 3 0 1 2 1 3 4 3 5
3 1 2 1 4 3 5 4 3 5 2 1 3 0 1 2 1 3 4 3 5
\end{verbatim}
\end{small}
The five words collapse to a single $P$-class:

\begin{small}
\begin{verbatim}
$ sorter -i 7.words.txt --uniq P --exact-cache --print gens > 7.uniq-p.txt
$ cat 7.uniq-p.txt
3 1 4 5 2 1 3 4 3 5 2 1 3 0 1 2 1 3 4 3 5
\end{verbatim}
\end{small}
Symmetry analysis on this representative:

\begin{small}
\begin{verbatim}
$ sym_detect -i 7.uniq-p.txt
{"n": 7,
 "pid": "P7-1xi1hu47otg9n",
 "gp_ord": 4, "gp_type": "D2",
 "gp_gens": [[0,4,2,5,1,3,7,6], [2,6,0,5,7,3,1,4]],
 "gp_elt_orders": {"1": 1, "2": 3},
 "chart_to_e": [1,0,1,2,0,2,0,0],
 "e_classes": [
   {"eid": "E7-21x35repcgpxf", "m": 4, "reps": 14,
    "he_type": "C1", "he_ord": 1, "he_gens": []},
   {"eid": "E7-yxwe78lyhblz", "m": 2, "reps": 7,
    "he_type": "D1", "he_ord": 2,
    "he_gens": [[6,5,4,3,2,1,0,7]]},
   {"eid": "E7-1xi1hu47otg9n", "m": 2, "reps": 7,
    "he_type": "D1", "he_ord": 2,
    "he_gens": [[2,1,0,6,5,4,3,7]]}
 ],
 "checks": {"sum_m_ok": true, "orbit_stabilizer_ok": true,
            "representatives_ok": true, "generators_close": true,
            "e_generators_close": true, "chart_to_e_ok": true}}
\end{verbatim}
\end{small}
The single $P$-class splits into three $E$-classes. 
\texttt{chart\_to\_e = [1,0,1,2,0,2,0,0]} shows the
distribution: id $0$ (i.e., \texttt{E7-21x35repcgpxf}) occurs $4$ times, ids $1$ and $2$ twice each,
giving the multiplicities $m(E_0)=4$, $m(E_1)=m(E_2)=2$ recorded in
\texttt{e\_classes}, with sum $n+1=8$.  The aggregator prints
one main row plus a sub-row, one per $(\operatorname{rep}(E), H_E)$ pair:

\begin{small}
\begin{verbatim}
$ sym_detect -i 7.uniq-p.txt | sym_detect_aggregate.py
n = 7, P-classes = 1, profiles = 1

G_P | |G_P| | #E/P | m(E) | rep(E) | H_E | rep(P) | #P | #E | #D
----+-------+------+------+--------+-----+--------+----+----+---
D2  |     4 |    2 |    2 |      7 | D1  |     28 |  1 |  3 | 28
    |       |    1 |    4 |     14 | C1  |        |    |    |
----+-------+------+------+--------+-----+--------+----+----+---
\end{verbatim}
\end{small}
This profile is the one tabulated in \cref{tab:sym-combined}: a single $P$-class with symmetry group
$G_P\cong D_2$ that splits into $\#E=3$ $E$-classes and spans
$\#D=28$ wiring diagrams.  The last is the sum
$\sum_{\text{rows}} \operatorname{rep}(E)\cdot \#E/P = 2\cdot 7 + 14$ over the
two rows.

\clearpage
\bibliographystyle{unsrt}
\bibliography{references}

\end{document}